\documentclass[10pt]{amsart}
\usepackage{amssymb,amsmath}
\usepackage{geometry}    
\usepackage{amsfonts}
\usepackage{mathrsfs}

\usepackage{url}

\usepackage{graphicx}
\usepackage{mathabx}
\usepackage{ytableau}
\usepackage{vmargin}

\usepackage{tikz}  
\usetikzlibrary{backgrounds}
\usetikzlibrary{trees,arrows}
\usepackage[colorlinks=true, linkcolor=blue, citecolor=blue,
pagebackref=true]{hyperref}

\usepackage{xcolor}  

\newtheorem{theorem}{Theorem}[section]
\newtheorem{lemma}[theorem]{Lemma}
\newtheorem{proposition}{Proposition}
\newtheorem{corollary}[theorem]{Corollary}
\theoremstyle{definition}

\newtheorem{remark}{Remark}

\newcommand{\scp}[1]{\langle#1\rangle}
\newcommand{\ux}{\boldsymbol{\xi}}

\newcommand{\ee}{\boldsymbol{e}}

\newcommand{\uz}{\boldsymbol{\zeta}}

\newcommand{\w}{\boldsymbol{w}}

\newcommand{\vv}{\boldsymbol{v}}

\newcommand{\uxs}{\boldsymbol{\xi}^{\ast}}

\newcommand{\bb}{\boldsymbol{b}}

\newcommand{\R}{\mathbb{R}}

\newcommand{\Z}{\mathbb{Z}}

\newcommand{\N}{\mathbb{N}}

\newcommand{\Q}{\mathbb{Q}}

\newcommand{\ug}{\boldsymbol{\gamma}}

\newcommand{\xx}{\boldsymbol{x} }

\makeatletter
\newcommand{\xleftrightarrow}[2][]{\ext@arrow 3359\leftrightarrowfill@{#1}{#2}}
\makeatother

\begin{document}

\title[On the Folklore set and Dirichlet spectrum for matrices]{On the Folklore set and Dirichlet spectrum for matrices}

\author[Mumtaz Hussain]{Mumtaz Hussain}
\address{Mumtaz Hussain,  Department of Mathematical and Physical Sciences,  La Trobe University, Bendigo 3552, Australia. }
\email{m.hussain@latrobe.edu.au}

\author[Johannes Schleischitz]{Johannes Schleischitz}
\address{Johannes Schleischitz, Middle East Technical University, Northern Cyprus Campus, Kalkanli, G\"uzelyurt.}
\email{johannes@metu.edu.tr; jschleischitz@outlook.com}

\author[Benjamin Ward]{Benjamin Ward}
\address{Benjamin Ward,  University of York, Heslington, York, YO10 5DD, United Kingdom.}
\email{benjamin.ward@york.ac.uk; ward.ben1994@gmail.com}

\begin{abstract}
We study the Folklore set of Dirichlet improvable matrices in $\R^{m\times n}$ which are neither singular nor badly approximable. We prove the non-emptiness for all positive integer pairs $m,n$ apart from $\{m,n\}=\{ 1,1\}$ and $\{m,n\}=\{ 2,3\}$ in a constructive manner.  For a wide range of integer pairs $(m,n)$ we construct subsets of the Folklore set with an exact prescribed Dirichlet constant (in some right neighbourhood of $0$). This enables us to provide information on the Dirichlet Spectrum of matrices. The key technique of our construction is to build  first vectors of a given Diophantine type, and then 
to show that most `liftings' to matrices will preserve this Diophantine type. 
This is a variant of a method 
introduced by Moshchevitin for uniform approximation.
Our technique is often also applicable to arbitrary norms. As a corollary we obtain lower bounds on the Hausdorff dimension of these sets. These statements complement previous results of the middle-named author (Selecta Math. 2023), Beresnevich et. al. (Adv. Math. 2023), and Das et. al. (Adv. Math. 2024).

\end{abstract}

\maketitle

{\footnotesize{

Math Subject Classification 2010: 11J13, 11J82, 11J83, 11K60}}

\vspace{1mm}

\section{ Uniform approximation for Systems of linear forms and arbitrary norms} \label{s12}

Let $n,m \in \N$, and let $\|\cdot\|_{1}$ be any norm on $\R^{n}$ and $\|\cdot\|_{2}$ any norm on $\R^{m}$. For a given real $m\times n$ matrix $\Omega \in \R^{m\times n}$ 
and a parameter $t\ge 1$ define the piecewise constant approximation function
\[
\psi_{\Omega}^{m\times n}(\|\cdot\|_{1},\|\cdot\|_{2}, t):=
\underset{ 1 \leq \Vert \boldsymbol{\hat{b}}\Vert_{1}\leq t }{\min_{\boldsymbol{b}=(\hat{\boldsymbol{b}},\tilde{\boldsymbol{b}})\in \Z^{n+m}:}}  \left\|\Omega \cdot \widehat{\boldsymbol{b}} + \tilde{\boldsymbol{b}} \right\|_{2}
=
\underset{ 1 \leq \Vert \boldsymbol{\hat{b}}\Vert_{1}\leq t }{\min_{\boldsymbol{b}=(\hat{\boldsymbol{b}},\tilde{\boldsymbol{b}})\in \Z^{n+m}:}}  \left\|\sum_{i=1}^{n}\Omega_{i}b_{i} + \tilde{\boldsymbol{b}} \right\|_{2},
\]
where $\Omega_{i}$ denotes the $i$th column of the matrix $\Omega$, and 
\begin{equation} \label{eq:hatnot}
    \hat{\boldsymbol{b}}=(b_{1},\dots,b_{n}), \qquad \tilde{\boldsymbol{b}}=(b_{n+1},\dots,b_{n+m}),\qquad \boldsymbol{b}=(\hat{\boldsymbol{b}},\tilde{\boldsymbol{b}}) 
\end{equation}
interpreted as column vectors. The column vectors $\boldsymbol{b}$ for which the minimum is attained for some $t$, are often called best approximations of $\Omega$. See works of Lagarias \cite{Lagarias1,Lagarias2} and the survey articles \cite{moshchevitinsurvey, Chevalliersurvey} for the history and further details on best approximations in dimensions greater than $1$.  \par
Define the Dirichlet constant of a matrix as
\[
\Theta^{m\times n}(\Vert\cdot\Vert_1,\Vert\cdot\Vert_2,\Omega)= \limsup_{t\to\infty}\; t^{n/m} \psi_{\Omega}^{m\times n}(\Vert\cdot\Vert_1,\Vert\cdot\Vert_2,t),
\]
and further let
\[
\widetilde{\Theta}^{m\times n}(\Vert\cdot\Vert_1,\Vert\cdot\Vert_2,\Omega)= \liminf_{t\to\infty}\; t^{n/m} \psi_{\Omega}^{m\times n}(\Vert\cdot\Vert_1,\Vert\cdot\Vert_2,t).
\]
The values $\Theta^{m\times n}$ and $\widetilde{\Theta}^{m\times n}$ 
compare the approximation quality with
the standard Dirichlet function $t^{-n/m}$ in the uniform and ordinary settings, respectively. By equivalence of norms and Dirichlet's Theorem, for each pair $(m,n)$ and each pair of norms $\|\cdot\|_{1},\|\cdot\|_{2}$ we can derive the constants
\begin{equation} \label{eq:obben}
D=D^{m\times n}(\Vert\cdot\Vert_1,\Vert\cdot\Vert_2):= \sup_{\Omega\in\R^{m\times n}} \Theta^{m\times n}(\Vert\cdot\Vert_1,\Vert\cdot\Vert_2,\Omega) < \infty.
\end{equation}
For $\|\cdot\|_{1}$ and $\|\cdot\|_{2}$ the standard maximum norms, we have $D=1$, with the upper estimate coming from Minkowski's Theorem for systems of linear forms. This value is attained for Lebesgue almost all matrices by well-known metrical results. For any pair $(m,n)$ and norms $\|\cdot\|_{1},\|\cdot\|_{2}$ we follow \cite{akmo} and define the \textit{Dirichlet Spectrum} to be the set
\begin{equation} \label{eq:dirspe}
    \mathcal{D}_{m,n}(\Vert\cdot\Vert_1,\Vert\cdot\Vert_2):= \left\{ \Theta^{m\times n}(\Vert\cdot\Vert_1,\Vert\cdot\Vert_2,\Omega) : \Omega \in \R^{m\times n} \right\} \subseteq [0,D]\, .
\end{equation}
When the norms are clear we will omit them from notation. For results on the Dirichlet spectrum with respect to the maximum norm see \cite[Appendix]{j1} and more completely \cite{j2}. In \cite{j2} it is proven that the spectrum is the entire unit interval in the dual setting with the max norm,
i.e. $\mathcal{D}_{1,n}(\Vert\cdot\Vert_{\infty})=[0,1]$. For results on the Dirichlet spectrum in the simultaneous approximation setting for various norms on $\R^{2}$ see the recent articles \cite{algo, kleinbock} for the current state of the art. In short, it was shown in \cite[Corollary 1.2]{kleinbock} that the upper bound $D^{2\times 1}(\vert\cdot\vert,\Vert\cdot\Vert_2)$ defined in \eqref{eq:obben} is an accumulation point for any norm on $\R^2$, and in \cite{algo} it is proven that for any arbitrary norm the Dirichlet spectrum contains an interval starting at the origin, see \cite[Corollary 1]{algo}. Both of these results are for simultaneous approximation on $\R^{2}$. We should stress that results for specific norms were known prior to the statements of \cite{algo,kleinbock}, see for example \cite{akmo, akh,j1}. See \S~\ref{recentadvances} for a very recent result on the Dirichlet Spectrum. \par

The Dirichlet constant can also be used to categorize matrices into certain well-known sets in Diophantine approximation. For $0\le c\le D$, define the following sets
\begin{align*}
    Di_{m,n}(\|\cdot\|_{1},\|\cdot\|_{2},c)&:=\left\{ \Omega \in \R^{m\times n} : \Theta^{m\times n}(\|\cdot\|_{1},\|\cdot\|_{2},\Omega)\leq c \right\}, \\
    DI_{m,n}(\|\cdot\|_{1},\|\cdot\|_{2},c)&:=\left\{ \Omega \in \R^{m\times n} : \Theta^{m\times n}(\|\cdot\|_{1},\|\cdot\|_{2},\Omega)= c \right\}.
\end{align*}
The set $DI_{m,n}(\|\cdot\|_{1},\|\cdot\|_{2},c)$ denotes the set of points that have exact Dirichlet constant $c$. Note that trivially $DI_{m,n}(\|\cdot\|_{1},\|\cdot\|_{2},c)=\emptyset$ if $c \not\in \mathcal{D}^{m\times n}(\Vert\cdot\Vert_1,\Vert\cdot\Vert_2)$. The set $Di_{m,n}(\|\cdot\|_{1},\|\cdot\|_{2},c)$ represents the set of $c$\textit{-Dirichlet improvable matrices with respect to norms} $\Vert\cdot\Vert_{1},\Vert\cdot\Vert_{2}$. The more widely known set of \textit{Dirichlet improvable matrices} is denoted as
\begin{equation*}
    Di_{m,n}(\|\cdot\|_{1},\|\cdot\|_{2}):=\left\{ \Omega \in \R^{m\times n} : \Theta^{m\times n}(\|\cdot\|_{1},\|\cdot\|_{2},\Omega)< D \right\}=\bigcup_{0\;<\;c\;<\;D}Di_{m,n}(\|\cdot\|_{1},\|\cdot\|_{2},c).  \\
\end{equation*}
Note that $DI_{m,n}(\|\cdot\|_{1},\|\cdot\|_{2},c)\subseteq Di_{m,n}(\|\cdot\|_{1},\|\cdot\|_{2},c) \subseteq Di_{m,n}(\|\cdot\|_{1},\|\cdot\|_{2})$. When it it clear from context what norms are being used we will drop the notation. As well as the set of Dirichlet improvable matrices one can use the Dirichlet constant (the uniform and asymptotic one) to characterize the set of \textit{Singular matrices} and the set of \textit{Badly approximable matrices}. Define
\begin{align*}
    Sing_{m,n}&:=DI_{m,n}(\Vert\cdot\Vert_{1},\Vert\cdot\Vert_{2},0)=\bigcap_{0\;<\;c\;<\;D} Di_{m,n}(\|\cdot\|_{1},\|\cdot\|_{2},c),\\
    Bad_{m,n}&:=\left\{ \Omega \in \R^{m\times n} : \widetilde{\Theta}^{m\times n}(\|\cdot\|_{1},\|\cdot\|_{2},\Omega)> 0 \right\}.
\end{align*}
 Note that the set of Singular and Badly approximable matrices do not depend on norms, by the equivalence of norms. Moreover it follows from Mahler's Dual Convex Body Theory that
 \[
 Bad_{m,n}=Bad_{n,m}, \quad \text{ and } \quad Sing_{m,n}=Sing_{n,m}.
 \] 
 Trivially, for any $0<c<D$ we have that $Sing_{m,n} \subseteq Di_{m,n}(c)\subseteq Di_{m,n}$, and for $Bad_{m,n}$ we have the relations
 \begin{equation*}
     Bad_{m,n}\subseteq Di_{m,n} \quad \text{ and } \quad Bad_{m,n}\cap Sing_{m,n}=\emptyset\, .
 \end{equation*}
 The left was shown by Davenport and Schmidt \cite[Theorem 2]{DavenportSchmidt3} and the right is obvious. The sets $Bad_{m,n}$ and $Sing_{m,n}$ are non-empty, as was classically shown by Jarnik and Khintchine (see \cite{Khinchin3} for the case of singular vectors). Furthermore, it is well known that all of the sets highlighted above are Lebesgue nullsets. This can be proven through a variety of methods such as the Dani Correspondence principle and ergodicity, or geometric measure theory, see for example \cite{BDGW_null}. So to give any meaningful notion of size to these sets we consider the Hausdorff dimension, denoted $\dim_H$. See \cite{Falconer_book2013} for the definition and properties of $\dim_H$. The following results have been proven:
\begin{align*}
\dim_H Bad_{m,n} &= mn\, ,\;\;   \text{for  $m=n=1$ in \cite{Jarnik3} and for general $m,n$ in \cite{schmidt1960,schmidtsystemoflinearforms}},\\[2ex]
\dim_H Di_{m,n} &= mn\, ,  
    \text{ \cite{schmidt1960,schmidtsystemoflinearforms} combined with \cite[Theorem 2]{DavenportSchmidt3}}, \\[2ex]
    \dim_H Sing_{m,n} &= mn\left(1-\frac{1}{m+n}\right)\, ,\;\;
    \begin{array}{c}\text{for  $m=2, n=1$ in \cite{Cheung}, for $m\geq 2, n=1$ in \cite{CheungChevallier},}\\
    \text{and for general $m,n$ in~\cite[Theorem~1.1]{Dasetal}}.
    \end{array}
\end{align*}
The dimensions of $Di_{m,n}(\Vert\cdot\Vert_{1},\Vert\cdot\Vert_{2},c)$ are not yet established. In the case of maximum norms, $m\geq 2, n=1$, and any constant $C>0$, for $c>0$ sufficiently small and for any real number $t>m$ the bounds
\begin{equation*}
    \tfrac{m^{2}}{m+1}+c^{t}\leq \dim_H Di_{m,1}(c)\leq \tfrac{m^{2}}{m+1}+C\cdot c^{m/2}
\end{equation*}
were proven in \cite[Theorem 1.3]{CheungChevallier}. See also \cite[Corollary 6.12]{CheungChevallier} for an improvement on the lower bound.
The reader will notice that the dimension results for $Sing_{m,n}$ and $Di_{m,n}(c)$ hold only for $\min\{m,n\}\geq 2$. For $m=n=1$ it was proven by Khintchine \cite{Khinchin3} that $Sing_{1,1}=\Q$,
in fact $Di_{1,1}(c)=\Q$ for any $c<\frac{1}{2}$.
By combining this with \cite[Theorem 1]{DavenportSchmidt3},  we obtain
\begin{equation*}
    Di_{1,1}=\Q\cup Bad_{1,1}.
\end{equation*}
That is, any point in the set of one dimensional Dirichlet improvable points is either Badly approximable or Singular (i.e. rational in this case). In higher dimensions much less was known, until the recent paper of Beresnevich et. al. \cite{BerGuaMarRamVel}, who showed that, for $n\geq 2$, the \textit{Folklore set}
\begin{equation*}
    FS_{m,1}:=Di_{m,1}\backslash(Bad_{m,1}\cup Sing_{m,1})\, 
\end{equation*}
has a continuum of points. They mention that this question was first asked in \cite{FabianSuessThesis} (see the remark immediately following \cite[Theorem 4.6]{FabianSuessThesis}) and furthermore conjectured that, for $n\geq 2$, we should have $\dim_H FS_{m,1}=m$, see \cite[Problem 3.1]{BerGuaMarRamVel}. Given this setup we define the Folklore set on $m\times n$ matrices with arbitrary norms as
\begin{equation*}
    FS_{m,n}(\|\cdot\|_{1},\|\cdot\|_{2})=Di_{m,n}(\|\cdot\|_{1},\|\cdot\|_{2})\backslash \left(Bad_{m,n} \cup Sing_{m,n} \right)
\end{equation*}
and, for any $0<c<D$, the smaller set
\begin{equation*}
     FS_{m,n}(\|\cdot\|_{1},\|\cdot\|_{2},c)=DI_{m,n}(\|\cdot\|_{1},\|\cdot\|_{2},c)\backslash Bad_{m,n}.
\end{equation*}
When the norms are the maximum norms we denote the sets by $FS_{m,n}$ and $FS_{m,n}(c)$ respectively. Again, it was conjectured that the set $FS_{m,n}(\|\cdot\|_{1},\|\cdot\|_{2})$ has full Hausdorff dimension $mn$ \cite[Problem 4.1]{BerGuaMarRamVel}. See \S~\ref{recentadvances} for more results on the Folklore set. In this article, we prove a range of results that extend the theorems given in \cite{BerGuaMarRamVel} in two different ways:
\begin{itemize}
    \item In Section~\ref{dual results} we begin by considering the Folklore set in the dual approximation setting ($m=1$) and add an additional condition that the point has some specific Dirichlet constant $c$. That is, we consider the set $FS_{1,n}(\Vert\cdot\Vert_{1},\Vert\cdot\Vert_{2},c)$. We prove this set has Hausdorff dimension at least $n-2$ for all constants $c$ in some right neighbourhood of $0$. In particular the Dirichlet spectrum $\mathcal{D}_{m,n}$ in such a setting contains a proper interval
    starting at $0$.
    \item In Section~\ref{matrix results} we consider the Folklore set in the setting of systems of linear forms. In almost all cases of $m,n$, we prove non-emptiness (Theorem~\ref{non-empty main}) and in most of these cases positive Hausdorff dimension 
 {(Theorems~\ref{arbn},~\ref{ganze}, \ref{thm5}, \ref{tT}).} Again, in several cases we can describe any small enough positive exact Dirichlet constant, even when considering a setup with arbitrary norms.
\end{itemize}
Prior to stating our metrical results we give the following lemma, which is a simple consequence of German's transference principle \cite[Theorem 7]{german}.
 \begin{lemma} \label{Folklore transpose}
 { Given $m,n$, there exists $\delta=\delta_{m,n}>0$ so that
     if $\Omega \in FS_{m,n}(\|\cdot\|_{1},\|\cdot\|_{2}, c)$ for some $0<c<\delta$, then its transpose $\Omega^T$ satisfies $\Omega^{T} \in FS_{n,m}(\|\cdot\|_{2},\|\cdot\|_{1})$.}
 \end{lemma}

 We will provide a proof at the end in Section~\ref{Sect10}. Note when considering the transpose we lose control over the precise Dirichlet constant.
 
\subsection{A note on recent advances} \label{recentadvances}
This area of research has recently gained a significant amount of interest. In particular, since the preparation of this article there has been two remarkable results within this area that improve upon several aspects of some of our results. We highlight these below and where relevant within the article.
\begin{itemize}
    \item The article of Das, Fishman, Simmons and Urbanski on the Variational principle on the Parametric Geometry of Numbers \cite{Dasetal} is a breakthrough piece of work that takes recent work of Roy \cite{Roy} and Schmidt and Summerer \cite{SchmidtSummerer} and provides a general framework to calculate Hausdorff dimension formula. The Folklore set described above is such a set that fits into their very general framework. We note that the previous iteration of the variational principle had not been immediately applicable to the Folklore set, thus not yet implying \cite{BerGuaMarRamVel}, see \cite[Remark 4.4]{BerGuaMarRamVel}. In particular, they are able to prove the Hausdorff dimension bound
    \begin{equation*} \label{eq:daset2}
        \dim_H FS_{m,n}(\Vert\cdot\Vert_{\infty}, \Vert\cdot\Vert_{\infty}) \geq mn\left(1-\frac{1}{m+n}\right)\, , \quad  \quad \text{\cite[Theorem 3.15]{Dasetal}}. 
    \end{equation*}
    That is, the set is at least as large as the set of Singular matrices, in terms of Hausdorff dimension. At the time of preparing this manuscript the authors were not aware of the recent advances in \cite{Dasetal}. Our proof technique is different and more constructive in nature, but the Hausdorff dimension lower bounds of \cite{Dasetal}, in the max norm case, are better than ours. But what we can provide, which \cite{Dasetal} do not, is a Hausdorff dimension bound on the set $FS_{m,n}(\|\cdot\|_{1},\|\cdot\|_{2},c)$. That is, we can prescribe exact Dirichlet constants, even for arbitrary norms, in certain cases. See for example our Theorem~\ref{thm1} below. We have attempted to make clear throughout where the bound \eqref{eq:daset2} improve on our statements and where our results are new. Note that the conjectured (full) dimension still remains elusive. \par 
    
    \item In a recent remarkable preprint of Agin and Weiss \cite{AginWeiss2024} the Dirichlet Spectrum of any $m,n$ and arbitrary norms was calculated. Precisely, in notation \eqref{eq:dirspe}, they have proven that
    \begin{equation*}
    \mathcal{D}_{m,n}(\Vert\cdot\Vert_1,\Vert\cdot\Vert_2)=[0,D]\, , \quad \quad \text{\cite[Theorem 1]{AginWeiss2024}}.
    \end{equation*}
    Moreover they state that for each $c\in[0,D]$ the sets $DI_{m,n}(\Vert\cdot\Vert_1,\Vert\cdot\Vert_2,c)$ are uncountable and dense in $\R^{m\times n}$. What they do not prove, which we are able to, are Hausdorff dimension lower bounds on the sets $DI_{m,n}(\Vert\cdot\Vert_1,\Vert\cdot\Vert_2,c)$. See for example the latter part of our Theorem~\ref{jia} below. Their results also allow for simplifications of some parts of our proofs, and in some cases even improve upon our results. See Theorem~\ref{AWe} in the appendix for one such improvement. Generally, since \cite{AginWeiss2024} appeared subsequent to the present paper and uses a different method, we will usually discard it from the considerations and take as a foundation results established prior to the present article, only occasionally making reference to \cite{AginWeiss2024}.
\end{itemize}

\noindent{\bf Acknowledgments.} The research of Mumtaz Hussain and Benjamin Ward is supported by the Australian
Research Council Discovery Project (200100994). Some of this work was carried out when Johannes Schleischitz visited La Trobe University supported by an International Visitor Program of the Sydney Mathematical Research Institute (SMRI) and La Trobe University, we thank them for their generous support. Finally, we thank Tushar Das and David Simmons for pointing out some relevant results from their paper \cite{Dasetal}.
 
\section{On the Folklore set in the Dual approximation setting}
\label{dual results}

\subsection{Maximum norm}
Let us, for now, consider the classical dual approximation and maximum norm setting. That is $m=1$, $\|\cdot\|_{1}=\|\cdot\|_{\infty}=\max_{1\leq i \leq n}|\cdot|$, and $\|\cdot\|_{2}=|\cdot|$. For ease of notation,
for $c\in (0,1)$ we write $$
DI_{n}(c)
:=DI_{1,n}(\Vert\cdot\Vert_{\infty},\vert\cdot\vert,c),
\qquad FS_{n}(c):=DI_{n}(c)\backslash \left( Bad_{1,n} \cup Sing_{1,n} \right).
$$ 
It was shown in~\cite{j2} by the second name author that the set $FS_n(c)$ is non-empty for any $c\in (0,1)$, refining the aforementioned result from~\cite{BerGuaMarRamVel}
and also complementing~\cite{Dasetal}. \par 

Our new result for the vector case and maximum norm is the following.

\begin{theorem}  \label{thm1}
 	For any natural number $n\ge 4$, there exists explicitly computable $c_n\in~(0,1)$ such that for any $c\in (0,c_n)$ we have
  \begin{equation*}
      \dim_H FS_{n}(c)\geq n-2.
  \end{equation*}
 \end{theorem}

\begin{remark}
	We believe the claim holds for $n=3$ as well, probably improving on the bound that can be obtained from the methods in~\cite{j1, j2}. However, problems occur in our proof in this case stemming from 
 auxiliary results from~\cite{ngm}, leaving this case open.
\end{remark}

\begin{remark} \label{remark4}
	In~\cite{j1} a lower bound of weaker order $(3/8)n+o(n)$ as $n\to\infty$ is shown for the corresponding sets with respect to simultaneous approximation for any $c\in (0,1)$. 
	In~\cite{j2} it is shown that $FS_n(c)$ for a linear form as defined above
	is not empty for any $c\in (0,1)$, and reference to~\cite{j1} is given on how to obtain metrical results in a similar way. However, again these bounds will turn out considerably weaker than in Theorem~\ref{thm1} above. The bound of Theorem~\ref{thm1} can in fact be further improved a little with some effort, see Remark~\ref{rrr} in Section~\ref{proof dual case}.
\end{remark}

While the method of the 
proof of \cite[Theorem~3.15]{Dasetal} 
shows that any set
\[
Di_{n}(c)\backslash \left( Bad_{n} \cup Sing_{n} \right)\supseteq FS_n(c),\,\qquad c\in (0,1) 
\]
has Hausdorff dimension bounded from below as in \eqref{eq:daset2} no matter how small $c$ is, the Dirichlet constant cannot be fixed in their work. So our Theorem~\ref{thm1} that 
permits control over the precise Dirichlet constant, at least in some interval,
is new. We may take
\[
c_n =
\frac{n-3}{8(n-2)^{3/2}}\cdot \sqrt{\pi}\cdot \frac{ \Gamma(n-\frac{1}{2})}{\Gamma(n)}
\]
with $\Gamma$ the Euler Gamma function. 
As $c\to 0^{+}$, for the Hausdorff dimension we cannot hope for anything better than $n-1+\frac{1}{n+1}$ the value in \eqref{eq:daset2},
as Cheung and Chevallier~\cite{CheungChevallier} showed this as the asymptotics for
the Hausdorff dimension of the larger set $Di_n(c)\supseteq FS_n(c)$. Hence we consider our bound reasonably good. Similar to~\cite{j1, j2} the expected increase of the Hausdorff dimension of $FS_n(c)$ as a function of $c\in[0,1]$ is not reflected in Theorem~\ref{thm1}. Using \cite[Proposition 1.1 and Lemma 4.9]{BerGuaMarRamVel} we infer the following corollary towards the original conjecture \cite[Problem 3.1]{BerGuaMarRamVel}.

\begin{corollary} \label{Folklore set dimension lower bound}
Let $n\geq 4$. Then 
    we have
\begin{equation*}
    \dim_H FS_{n} \geq n-2,\,\qquad  \dim_H FS_{n,1}\geq n-2\, .
\end{equation*}
\end{corollary}

 This result is stated more for sake of completeness as the metrical bound is weaker than \eqref{eq:daset2} proved in \cite{Dasetal}. However our bound is obtained by a constructive method and does not use the variational principle.

\begin{remark}
We are not able to deduce an equivalent claim for any set $
        FS_{n}^{*} (c):=FS_{n,1}(c)$ with $c>0$.
    The reason is that the
    sets $FS_{n}^{*} (c)$ do
    not coincide with $FS_{n}(c)$ for $c>0$, indeed Lemma~\ref{Folklore transpose} is too weak.
    On the other hand, similar
    to~\cite{BerGuaMarRamVel}
    or~\cite[Section~11]{j2}
    it allows for deducing
    non-emptyness (Hausdorff  dimension $\ge n-2$ again) for sets
    $Di_{n,1}(c)\backslash (Bad_{n,1} \cup Di_{n,1}(c^{\prime}))$
    where suitable $c^{\prime}\in (0,c)$ can
    be explicitly computed as a function of $n\ge 2,c\in (0,1)$. 
    An analogous claim, with stronger Hausdorff dimension estimate as in \eqref{eq:daset2} and a different
    (most likely smaller) constant $c^{\prime}$, can be obtained
    from the uniform version of the variational principle in~\cite{Dasetal} as well.
\end{remark}

 \subsection{Arbitrary norms}

Given any norm $\Vert\cdot\Vert=\Vert\cdot\Vert_1$ on $\R^n$ and $c>0$,
let us write
\[
DI_{n}(\Vert\cdot\Vert,c)
:=DI_{1,n}(\Vert\cdot\Vert,\vert\cdot\vert,c),
\qquad FS_{n}(\Vert \cdot \Vert, c):=DI_{n}(\Vert\cdot\Vert,c)\backslash \left( Bad_{1,n} \cup Sing_{1,n} \right).
\]
Our second result extends Theorem~\ref{thm1} to arbitrary norms. 

\begin{theorem} \label{jia}
	Let $n\ge 4$ and $\|\cdot\|$ be any norm on $\R^n$. Then there is $c_n=c_n(\Vert\cdot\Vert)>0$ 
	such that for any $c\in[0,c_n]$ we have 
	\[
	\dim_H FS_{n}(\Vert\cdot\Vert,c) \geq n-2\, .
	\]
    In particular the Dirichlet
	spectrum 
 \[
 \mathcal{D}_{n}(\Vert\cdot\Vert):= \{\Theta(\Vert\cdot\Vert,\ux): \ux\in \R^n \}:= \{ \Theta^{1\times n}(\Vert\cdot\Vert, |\cdot|,\ux): \ux\in \R^n \} \subseteq [0,D],
 \]
 contains some right neighborhood of $0$, and for each $c\in [0,c_{n}]$, $\dim_H DI_{n}(\|\cdot\|,c)\geq n-2$.
\end{theorem}

 Clearly this result is again independent from the work of Das et al~\cite{Dasetal}. The aforementioned preprint~\cite{AginWeiss2024} that appeared after our work proves that $\mathcal{D}_{n}(\|\cdot\|)=[0,D]$, but without the metrical estimate provided above. Using the special case $n=2$ from the subsequent preprint~\cite{AginWeiss2024} would lead to a considerable shortcut of our proof for the general case. We will highlight in the proof below where simplifications can be made.

\begin{remark} Since $FS_{n}(\Vert\cdot\Vert,c)\subseteq FS_{n}(\Vert\cdot\Vert)$ the lower bound on the Hausdorff dimension applies to $FS_{n}(\Vert\cdot\Vert)$ as well, which is however weaker than \eqref{eq:daset2}. \par 
\end{remark}
\begin{remark}
For $n=3$ and a general norm, even non-emptiness of the set $FS_n(\Vert\cdot\Vert,c)$ for a given $c\in(0,D)$, with $D$ as
in \eqref{eq:obben}, remains unknown. In particular \cite{AginWeiss2024} tells us nothing about this. For certain norms, this can be obtained from the method in~\cite{j2}. Again we expect the Hausdorff dimension to rise as a function of $c\in[0,D]$, not reflected in the theorem.
\end{remark}

\begin{remark}
In \cite{j2} it is proven that the spectrum is the entire unit interval, i.e. $\mathcal{D}_n(\Vert\cdot\Vert_{\infty})=[0,1]$. Theorem~\ref{jia} tells us the weaker statement that the spectrum contains some small interval starting at the origin, however, our result holds for any norm. 
\end{remark}

\begin{remark}
By Lemma~\ref{Folklore transpose}, from Theorem~\ref{jia} we may deduce an analogous metrical result to Corollary~\ref{Folklore set dimension lower bound} on the simultaneous approximation Folklore set 
for an arbitrary norm and $m\ge 4$.
This will be covered by the more general Theorem~\ref{tT} below.
\end{remark}

\section{The Folklore set  for a system of linear forms } \label{matrix results}


\subsection{Two fundamental results for maximum norm} \label{Sek3.1}

In this section we state some
results for the maximum norm that can be deduced from 
our more general results following in the subsequent sections. Our first result below on the Folklore set
is (mostly) weaker than the claims in \cite{Dasetal}. However, we still want to highlight it, as it will be a consequence of our considerably more general results below that indeed do not follow from work of Das et al~\cite{Dasetal}.

\begin{theorem} \label{non-empty main}
    Let $m,n$ be positive integers 
    satisfying $\{ m,n\}\notin \left\{ \{1,1\}, \{2,3\} \right\}$. Then
    \[
    FS_{m,n}=FS_{m,n}(\Vert \cdot\Vert_{\infty},\Vert \cdot\Vert_{\infty}) \ne \emptyset.
    \]
      In fact
      \[
      \dim_H FS_{m,n} \geq mn-2\max\{m,n\},
      \]
      and all claims hold when we choose the involved matrices
    of full $\R$-rank $\min\{m,n\}$.
\end{theorem}


\begin{remark}
Regarding the metrical bound appearing in Theorem~\ref{non-empty main}, 
we have previously mentioned that the Folklore set is conjectured to have full dimension \cite[Problem 3.1]{BerGuaMarRamVel}, and more generally that 
$$\dim_H FS_{m,n}(\Vert\cdot\Vert_1,\Vert\cdot\Vert_2,c)=mn-o(1)$$ as $c\to D^{-}$ for $D$ from \eqref{eq:obben}, for any pair of norms. However we cannot expect the dimension of $Di_{m,n}(\Vert\cdot\Vert_1,\Vert\cdot\Vert_2,c) \supseteq FS_{m,n}(\Vert\cdot\Vert_1,\Vert\cdot\Vert_2,c)$ to exceed the value from \eqref{eq:daset2} as $c\to 0^{+}$. 
\end{remark}

Our result follows from a combination of Theorems~\ref{sganz},~\ref{tT} below combined with the cases 
$m=1, n\ge 2$ and $m\ge 2, n=1$ already settled in~\cite{j1, j2}. The metrical result is
non-trivial as soon as $\min\{ m,n\}\ge 3$.

We stress again that Theorem~\ref{non-empty main} as such is, up to the claim on the rank, implied by the stronger bound \eqref{eq:daset2} from~\cite{Dasetal} which also applies to the cases 
$\{ m,n\}= \{2,3\}$ we miss. However, 
we emphasize that considerably refined information is provided by the partial results formulated below. The refinements will provide us with information on one or more of the following properties: stronger Hausdorff dimension results in several cases; arbitrary norms; more information on the Dirichlet spectrum, ideally imposing an exact Dirichlet constant akin to Theorem~\ref{jia}; and $\Q$-linear independence of matrix entries together with $\{1\}$. The exact claims we obtain are subject to the dimensions $m,n$ of the matrix.

Regarding the full rank
claim, by considering determinants of maximum quadratic minors we see that the Hausdorff dimension of the set of rank deficient real $m\times n$ matrices equals $mn-(\max\{ m,n\}-\min\{m,n\}+1)$. Hence it seems only implied by \eqref{eq:daset2} if
 \begin{equation} \label{eq:bedingung1}
     mn-\frac{mn}{m+n} >  mn-(\max\{ m,n\}-\min\{m,n\}+1).
 \end{equation}
 Asymptotically this is equivalent to
 \begin{equation} \label{eq:yyx}
     \frac{\max\{ m,n\}}{\min\{m,n\}} > \frac{\sqrt{5}+1}{2}-o(1), \qquad m,n\to\infty,
 \end{equation}
 so roughly speaking, the matrices should not be close to being square matrices.

 \par


The following theorem states some of our results on the Dirichlet spectrum \eqref{eq:dirspe}, with respect to maximum norms.

\begin{theorem} \label{neuth}
Let $k,m,n$ be positive integers.
\begin{itemize}
    \item If $\tfrac{m}{n} \in (0,\tfrac{1}{2}]\cup (1,\infty)$, $
(m,n)\notin \{ (3,2), (5,2), (7,2),\ldots \}$, then $\mathcal{D}_{m,n}$
contains some right neighborhood $[0,c_{m,n}]$ of $0$. 
\item If $m\vert n$ or $n\vert m$ and $(m,n)\ne (1,1)$ we have $\mathcal{D}_{m,n}=[0,1]$. 
\item Moreover \label{eq:teilen}
$\mathcal{D}_{m,n} \subseteq \mathcal{D}_{km,kn}$.

\end{itemize}

\end{theorem}

\begin{remark}
 The aformentioned preprint \cite{AginWeiss2024} has since improved upon these results. In particular they show that for any $m,n$ the spectrum is the full interval, even with respect to arbitrary norms. Nevertheless our results are proven via different methods, moreover they give rise to metrical estimates for the set of involved matrices. The first claim follows more generally from our results for any pair of expanding norms, see Section \ref{se3.1} for a definition. In most cases $\Vert\cdot \Vert_1$ can even be arbitrary. The first claim follows directly from a combination of Theorem~\ref{ganze} and Theorem~\ref{thm5} below, where more general norms are considered, apart from the vector cases $\min\{m,n\}=1$ which are implied by the second claim of the theorem.
The last claim is Corollary \ref{multiples} below, and the second claim follows in turn via the work in \cite{j1, j2}, and since $D=1$ for maximum norms. Note that the last claim implies that we can restrict to coprime $m,n$ for the study of the Dirichlet spectrum with respect to maximum norms. 
\end{remark}

Towards the proof of Theorems~\ref{non-empty main} and~\ref{neuth}, we consider two cases. Firstly when $\tfrac{m}{n}\in [\tfrac{1}{2},2]$, and secondly when $\tfrac{m}{n}\in (0,\tfrac{1}{2})\cup(2,\infty)$. 
Within these cases, we again split into further two subcases 
depending which variable is larger and transition
by using Lemma~\ref{Folklore transpose}. For example in the first case we begin with $\tfrac{m}{n}\in[1,2]$ and then use Lemma~\ref{Folklore transpose} to obtain results for when $\tfrac{m}{n}\in[\tfrac{1}{2},1]$.

\subsection{The case $m/n\in [1/2,2]$ } \label{se3.1}

We first treat the case $m=n$.
Here we keep restricting to maximum norms.

\begin{theorem}  \label{arbn}
    For $m=n\ge 2$, 
    the set $FS_{n,n}=FS_{n,n}(\Vert \cdot \Vert_{\infty},\Vert \cdot \Vert_{\infty})$ 
    satisfies
    \[
    \dim_H FS_{n,n} \ge (n-1)^2+1.
    \]
    Moreover, we can choose the matrices
    of full $\R$-rank $n$ (i.e. invertible).
\end{theorem}

 Again, up to the rank claim, this is implied by the result
from \cite{Dasetal}. 
The rank claim indeed appears not to be covered by~\cite{Dasetal}
as \eqref{eq:bedingung1} is easily checked to fail.
The construction of our proof gives very limited control of the exact Dirichlet constant here. Even $0$ being an accumulation point for the Dirichlet spectrum needs some additional conjectural assumption, which on the other hand can be inferred from~\cite[Theorem~3.15]{Dasetal} and its proof method.


Now let us study the case $m\ne n$.
Following~\cite{j1}, let us call a norm $\Vert\cdot\Vert$ on $\mathbb{R}^{k}$ expanding
if $\Vert \textbf{x}\Vert\ge \Vert \pi_j(\textbf{x})\Vert$ for any $\textbf{x}\in \mathbb{R}^k$ and every $1\le j\le k$, where $$\pi_j: (x_1,\ldots,x_k)\to (0,\ldots,0,x_j,0,\ldots,0)$$ is the projection to the $j$-th coordinate axis. As remarked
in~\cite{j1} this is a rather mild condition that in particular includes the maximum norm and all $p$-norms $\Vert \textbf{x}\Vert=(\sum |x_i|^p)^{1/p}$, $p\ge 1$. Recall the definition \eqref{eq:dirspe} of
the Dirichlet spectrum in dimensions $m,n$. Denote further by $\lceil x\rceil$ the smallest integer $\ge x$.
We show the following.

\begin{theorem}  \label{ganze}
    Assume $m,n$ are positive integers with
    \[
  1 < \tfrac{m}{n}, \qquad (m,n)\notin \{(3,2), (5,2), (7,2), \ldots \}. 
    \]
    Let $\Vert\cdot\Vert_1$  be an arbitrary norm on $\R^n$
    and $\Vert\cdot\Vert_2$ be any expanding norm on $\R^m$, unless if
    $n=2$ assume both norms $\Vert\cdot\Vert_1$ 
    and $\Vert\cdot\Vert_2$ are expanding.
    
    Then there exists $c_{m,n}=c_{m,n}(\Vert\cdot \Vert_1,\Vert\cdot\Vert_2)>0$ so that
    for any $c\in [0,c_{m,n}]$, we have
    \[
    FS_{m,n}(\Vert \cdot \Vert_{1},\Vert \cdot \Vert_{2},c)\ne \emptyset
    \]
    in fact unless $n=2$ or $(m,n)=(4,3)$
    we have
    \[
    \dim_H FS_{m,n}(\Vert \cdot \Vert_{1},\Vert \cdot \Vert_{2},c)\ge m\cdot (n-1)-\left\lceil\frac{m}{n}\right\rceil, 
    \]
    and for $(m,n)=(4,3)$ we still have the weaker bound
    $ \dim_H FS_{4,3}(\Vert \cdot \Vert_{1},\Vert \cdot \Vert_{2},c)\ge 4$. In particular the Dirichlet spectrum $\mathcal{D}_{m,n}(\Vert\cdot\Vert_1, \Vert\cdot\Vert_2)$ contains a right neighborhood of $0$. 
    Moreover, in these claims we can choose the matrices
    of full $\R$-rank $n$.
\end{theorem}

 When using a result from the aforementioned preprint \cite{AginWeiss2024}, the theorem can be improved by dropping the condition of $\Vert.\Vert_2$ being expanding and providing a stronger metrical bound of $m(n-1)$. Furthermore, using \cite{AginWeiss2024} vastly simplifies the proof, see 
Theorem~\ref{AWe}, (b) in the appendix below. Clearly the claim on $\mathcal{D}_{m,n}(\Vert\cdot\Vert_1, \Vert\cdot\Vert_2)$ is also implied by \cite{AginWeiss2024}, as for all our results below.

\begin{remark}
    We can improve the Hausdorff dimension
formulas in Theorem~\ref{ganze} a little. For
$\Vert\cdot\Vert_2$ the maximum norm, this works via applying~\cite[Theorem~3.1]{j1}
to estimate from below $\dim_H FS_{\ell,1}(c)$
for $c\in [0,1]$ where
$\ell=\lceil m/n\rceil$. This positive value can be added given the method of our proof below and using \eqref{eq:tric}. For 
general expanding $\Vert\cdot\Vert_2$ a similar argument applies
see~\cite[Remark~2]{j1}. 
This resembles Remark~\ref{remark4}.
    In particular in the cases $(m,2)$ for $m\ge 4$ even, we can settle a positive Hausdorff dimension.
\end{remark}

\begin{remark}
    The condition of $\Vert \cdot \Vert_2$
    being expanding can be relaxed to asking that for some space spanned
    by $\ell:=\lceil m/n\rceil$ canonical base vectors
    $\textbf{e}_{i_1},\ldots, \textbf{e}_{i_{\ell} }$
    the projected norm onto it is expanding. For example, for 
    $n<m\le 2n$ and
    $i_1=1, i_2=2$ this means the norm
    $\Vert (x_1,x_2)\Vert_{2}^{\prime}:=\Vert (x_1,x_2,0,\ldots,0)\Vert_2$ on $\R^2$ is expanding.
\end{remark}


As we prescribe exact Dirichlet constant, our Theorem~\ref{ganze} is not implied by \cite{Dasetal}. As we give metrical statements, and indeed show that the matrices are not badly approximable, the theorem is not implied by \cite{AginWeiss2024}. We remark that full $\Q$-rank of columns, a weaker statement, would be equivalent to the matrix giving rise to an infinite (non-terminating) sequence of best approximations. 

For Theorem \ref{non-empty main} we only need the case $m/n\in (1, 2]$, however for the full claim of Theorem \ref{neuth} we require the full statement of Theorem~\ref{ganze}. To that end we now restrict to the case $m/n\in (1,2]$ as the implications below from other cases will be superseded
from results in Section \ref{se3.2}.
By Lemma~\ref{Folklore transpose} we may deduce an analogous result to Theorem~\ref{ganze}, apart from notably not fixing exact Dirichlet constant $c$, in the cases $n\ge 4,\quad  \tfrac{1}{2} \le \tfrac{m}{n} < 1$. Together with the case $m=n$ from
Theorem~\ref{arbn}, we may summarize

\begin{theorem}  \label{sganz}
    Let $m,n$ be positive integers with
    \[
    \frac{m}{n}\in [\frac{1}{2},2],\qquad \{ m,n\}\notin \{ \{ 1,1\}, \{2,3\}\}.
    \]
    Then $FS_{m,n}(\Vert \cdot \Vert_{\infty}, \Vert \cdot \Vert_{\infty}) \ne \emptyset$. In fact
    \begin{equation}  \label{eq:cbi}
    \dim_H FS_{m,n}(\Vert \cdot \Vert_{\infty}, \Vert \cdot \Vert_{\infty})\ge mn-2\max\{m,n\}.
    \end{equation}
      Moreover, we can choose the matrices
    of full $\R$-rank
    $\min\{ m,n\}$.
    If additionally $m\ne n$,
    then more generally for any norms $\Vert \cdot \Vert_1$ on $\R^n$ and $\Vert \cdot \Vert_2$ on $\R^m$
    \[
    FS_{m,n}(\Vert \cdot \Vert_1, \Vert \cdot \Vert_2) \ne \emptyset, \quad 
    \dim_H FS_{m,n}(\Vert \cdot \Vert_{1}, \Vert \cdot \Vert_{2})\ge mn-2\max\{m,n\},
    \]
    and $\mathcal{D}_{m,n}(\Vert\cdot\Vert_1, \Vert\cdot\Vert_2)$ accumulates at $0$, in particular, it is not finite.
\end{theorem}

By contrast to Theorem~\ref{ganze}, this theorem is again, apart from the rank statement in certain cases, implied by \cite[Theorem~3.15]{Dasetal} and its proof. Regarding the rank claim,  by \eqref{eq:yyx} essentially the cases 
$\max\{ m,n\}/\min\{m,n\}\in [1,(1+\sqrt{5})/2=1.61\ldots)$ 
seem not covered by \cite{Dasetal}. 

We remark that \eqref{eq:cbi} can be improved by our method unless
$\{ m,n\}=\{ 2,4\}$ where the right hand side vanishes.
In our constructions of all 
Theorems~\ref{arbn},~\ref{ganze},~\ref{sganz}
the real matrices in our construction have some $0$ entries. In particular the entries
are not $\Q$-linearly independent
as in Theorems~\ref{thm5},~\ref{tT} below,
which would be desirable.

\subsection{The case $m/n\notin (1/2,2)$ }  \label{se3.2}

Now assume $n\ge 2m$.
In this case, we establish the following.

\begin{theorem}  \label{thm5}
	Let $m,n$ be positive integers satisfying
	\begin{equation} \label{matrix dimension condition}
 	\tfrac{m}{n}\le \tfrac{1}{2}, 
 \qquad (m,n)\ne (1,3).
	\end{equation}
	Let $\Vert\cdot\Vert_1$ be any norm on $\R^n$ and $\Vert\cdot\Vert_2$ be any 
	norm 
	on $\R^m$. Then there exists $c_{m,n}=c_{m,n}(\Vert\cdot\Vert_1,\Vert\cdot\Vert_2)>0$ such that for any $c\in[0,c_{m,n}]$ we have
 \begin{equation*}
 \dim_H FS_{m,n}(\Vert\cdot\Vert_1,\Vert\cdot\Vert_2,c) \geq (n-2)m.  
 \end{equation*}
 Analogous claims hold when we restrict to $\Omega\in FS_{m,n}(\Vert\cdot\Vert_1,\Vert\cdot\Vert_2,c)$ so that 
 \begin{itemize}
     \item[(i)]  $\Omega$
     has full $\R$-rank $m$
     \item[(ii)] the set of all its entries $\Omega_{i,j}$ together with $1$ is linearly independent over $\Q$.
 \end{itemize}
\end{theorem}

\begin{remark} \label{REMA}
Similar to Theorem~\ref{ganze}, our proof below can be significantly simplified when taking into account 
results from the preprint \cite{AginWeiss2024}.
Moreover, we may omit 
condition $n\ge 2m$ from \eqref{matrix dimension condition}.
However, as opposed to Theorem~\ref{ganze},  it does not improve the metrical statement. See Theorem \ref{AWe}, (a) below.
\end{remark}

Theorem~\ref{jia} is the special case $m=1$, $\Vert\cdot\Vert_{1}=\Vert\cdot\Vert$ and $\Vert\cdot\Vert_2=|.|$, the absolute value on $\R$. Note when $m=1$ conditions (i), (ii) are automatically satisfied for non-singular vectors. We will use 
Theorem~\ref{jia} to prove Theorem~\ref{thm5}. Due to fixing the Dirichlet constant none of the claims are implied by \cite{Dasetal}, nor is it implied by \cite{AginWeiss2024} due to our metrical claim. \par 
Note that in assumption (i), we impose
independence of rows over $\R$, not $\Q$. In the case $(m,n)=(1,3)$ where the results of Theorem~\ref{thm5} 
are not applicable, avoiding results from~\cite{Dasetal}, we can still deduce that the larger set
\begin{equation} \label{eq:FG}
Di_{1,3}(\Vert\cdot\Vert_{1},\Vert\cdot\Vert_{2},c)\setminus Bad_{1,3}\supseteq FS_{1,3}(\Vert\cdot\Vert_{1},\Vert\cdot\Vert_{2},c)
\end{equation}
is uncountable for arbitrarily small $c>0$, since 
by~\cite{j2} we have 
$FS_{1,3}(\Vert\cdot\Vert_{\infty},\Vert\cdot\Vert_{\infty},c)\ne \emptyset$ for any
$c\in [0,1]$
and by the equivalence of norms.
By~\cite{Dasetal}, in fact, the Hausdorff dimension bound \eqref{eq:daset2} applies for the left-hand side set in \eqref{eq:FG}. Similar to 
Section~\ref{se3.1}, using Lemma~\ref{Folklore transpose} for the reverse case $m\ge 2n$,
we get the following corollary. 

\begin{theorem}  \label{tT}
	Let $m,n$ be positive integers satisfying 
 \begin{equation} \label{eq:jjj}
\frac{m}{n}\notin \left(\tfrac{1}{2} , 2\right) , \qquad  (m,n)\notin \{ (1,3), (3,1), (2,4) , (4,2) \}.
 \end{equation}
	Let $\Vert\cdot\Vert_1$ be any norm on $\R^n$ and $\Vert\cdot\Vert_2$ be any 
	norm on $\R^m$.
	Then 
	\[
	\dim_H FS_{m,n}(\Vert\cdot\Vert_1,\Vert\cdot\Vert_2) \geq (\max\{m,n\}-2)\cdot \min\{m,n\}=mn-2\cdot \min\{m,n\}.
	\]
 Moreover we may again assume (i), (ii), and the Dirichlet spectrum $\mathcal{D}_{m,n}(\Vert\cdot\Vert_1, \Vert\cdot\Vert_2)$ accumulates at $0$, in particular it is not finite.
\end{theorem}

 By contrast to Theorem~\ref{thm5}, since we lose our grip on the exact Dirichlet constant and \eqref{eq:yyx} is implied by \eqref{eq:jjj}, all but the claim (ii)
are directly implied by Das et al \cite{Dasetal}. Since exceptions to (ii) form a countable union of rational hyperplanes thus of Hausdorff dimension $mn-1$ which is not exceeded by
the value from
\eqref{eq:daset2} as soon as $\min\{ m,n\}>1$,
(ii) is indeed not implied unless in the vector cases.

Note also that for the cases $\{m,n\}= \{2,4\}$ and $\{m,n\}=\{1,3\}$ excluded in \eqref{eq:jjj}, we have results via Theorem~\ref{ganze} and~\cite{j1, j2} respectively.

\begin{remark}
    It may be true that $Di_{m,n}=Di_{n,m}$ for every pair $m,n$. If $\min\{m,n\}=1$, one inclusion of the set identity was proved by Davenport and Schmidt~\cite{dav}. 
    Assuming this property would give a shortcut avoiding Lemma~\ref{Folklore transpose} for the proofs of Theorems~\ref{sganz},~\ref{tT}. 
\end{remark}


\section{Brief outline of proofs} \label{outline}
The outline of are proofs are as follows. For clarity we provide a summary graph of how our theorems link at the end of this section.




Theorem~\ref{thm1} is proven in Sections~\ref{proof dual case} and \ref{Sect6}, and Theorem~\ref{jia} is proven in Section~\ref{Sect7}. The dotted arrow in the above figure between Theorem~\ref{jia} and Theorem~\ref{thm5} indicates that the methodology of proof is similar, but by no means does Theorem~\ref{jia} immediately imply Theorem~\ref{thm5}. The same applies for the arrow between \cite[Theorem 2.2]{j1} and Theorem~\ref{ganze}, where besides Lemma~\ref{lemme}
also, the rather easy Lemma~\ref{llemma} below is a crucial ingredient.
The complete proof of Theorem~\ref{thm5} is given in Section~\ref{Sect8}. In Section~\ref{Sect9}, Theorems \ref{arbn} and \ref{ganze} are proven. All other implications are rather immediate.
Thus the main claims to be proved are those highlighted in bold font.
Lemma~\ref{Folklore transpose} is proven in Section~\ref{Sect10}. \par 

All proofs of claims from Sections 
 \ref{dual results} and ~\ref{se3.2} are based on constructions for the case $m=1, n=2$ of one linear form in two variables. For Theorem~\ref{thm1}, we can use a construction from~\cite{j2}, however for the results in arbitrary norms (Theorem~\ref{jia}) we follow 
essentially a different construction from~\cite{j3}, with some refinements.
Then, for either result, similar as in~\cite{j3}, we use a method of Moshchevitin~\cite{ngm}
to extend the claim to larger $n$. In Theorem~\ref{thm5} when $m>1$, 
there is some additional variational argument needed to extend the $m=1, n=2$
case first to arbitrary $m$ and $n=2$ before we use Moshchevitin's argument
to be able to increase $n$ as well arbitrarily. As an artifact of the method
from~\cite{ngm}, 
we lose the case $m=1, n=3$.
On the other hand, $n\ge 2m$ is essentially required to use the first step $m=1, n=2$ to conclude with the method in \cite{ngm}. 

Concerning Section~\ref{se3.1},
for Theorem \ref{ganze} we apply
a result from~\cite{j1} on $m=2, n=1$. 
We also again employ Moshchevitin's construction~\cite{ngm} here.  
We 
lose the case $(m,n)=(3,2)$ in Theorem~\ref{non-empty main} and the cases $(m,2)$ for $m\ge 3$ odd in Theorem~\ref{neuth}
due to problems originating in~\cite{j1}.
The special case $m=n\ge 2$ of Theorem~\ref{arbn} is not covered by the argument, indeed it does not need either of these prerequisite results~\cite{ngm, j1} but uses a different, self-contained construction. Its proof is considerably
shorter and easier than for the other main claims in bold font.

In summary, the claims above are connected in the following way (see below
for Corollary \ref{multiples}):
\medskip

\begin{center}
\begin{tikzpicture}[grow=right,->, >=angle 60, scale=1, framed]

\node at (11.2,0) {\bf Theorem \ref{thm1}};
\node at (4.4,-1) {Corollary \ref{Folklore set dimension lower bound}};

\path[->] (10.6,-.25)  edge  [bend left=30]  node[above,xshift=-.5cm] {\footnotesize Lemma \ref{Folklore transpose}} (5.6,-1.2);

\path[->] (10.7,-.25)  edge  [bend left=20]  node[below] {} (11.4,-5.2);

\begin{scope}[yshift=-4cm]

\node at (-.2,0) (a){[27, Theorem 2.2]};
\node at(4.4,1) (b){\bf Theorem \ref{arbn}};
\node at(4.4,-1) (c){\bf Theorem \ref{ganze}};
\path[dashed, ->] (0.85,-.25)  edge  [bend right=40]  node[below, yshift=-.1cm] {\footnotesize Lemma \ref{lemme}} (3.1,-1.2);

\path[->] (5.75,-1.2)  edge  [bend right=40]  node[below, yshift=-.1cm] {\footnotesize Lemma \ref{Folklore transpose}} (8.3,-.35);

\path[->] (5.75,1.2)  edge  [bend left=40]  node[below] {} (8.3,.35);
\node at (9, 0){Theorem \ref{sganz}}; 

\path[->] (3.75,-1.3)  edge  [bend right=40]  node[below] {} (4,-2.8);
\node at (4.4, -3){ Theorem \ref{neuth}};
\draw [->] (0.8, -3)--(3.2,-3);
\node at (-.6, -3){Corollary \ref{multiples}};

\path[->] (3.75,-4.8)  edge  [bend left=40]  node[below] {} (4,-3.3);
\end{scope}

\draw[dashed, ->] (1, -9)node[left]{\bf Theorem \ref{jia}} -- (3.1, -9) node[right] {\bf Theorem \ref{thm5}};
\draw[->] (6, -9)node[below, xshift=1.1cm]{\footnotesize Lemma \ref{Folklore transpose}}-- (8.2, -9)node[right]{Theorem \ref{tT}};
\node at (2, -9.3){\footnotesize Lemma \ref{lemme}};


\node at (11.2, -5.43) { Theorem \ref{non-empty main}};

\path[->] (10.3,-4)  edge  [bend left=40]  node[below] {} (11.2,-5.15);

\path[->] (10.7,-8.8)  edge  [bend right=40]  node[below] {} (11.5,-5.7);
\end{tikzpicture}
\end{center}

 As indicated before,
many of our relevant proofs can be simplified and certain claims can be improved when using results from the aforementioned preprint \cite{AginWeiss2024}. We will detail this in the relevant places below.


\section{Proof of Theorem~\ref{thm1}}

\subsection{Simplifying the problem  } \label{proof dual case}

In short, the proof combines the special two-dimensional case of~\cite{j2} with an idea of Moshchevitin~\cite{ngm}.
We first recall the result on linear forms in two variables. Throughout this section fix $n\in\N_{\geq 2}$ and let $\|\cdot\|_{1}=\Vert\cdot\Vert_{\infty}$ on $\R^{n}$ and $\|\cdot\|_{2}=|\cdot|$ on $\R$. For $\uz \in \R^{n}$ and $t\ge 1$ for ease of notation write
\begin{equation*}
\psi_{\uz}^{1\times n}(\|\cdot\|_{1},|\cdot|,t)=\psi_{\uz}^{1\times n}(t).
\end{equation*}

 \begin{theorem}[\cite{j2}] \label{thm2}
 	Let $n\ge 2$ and $c\in [0,1]$. Then there
 	exist $\Q$-linearly independent with $1$ real vectors $\ux=(\xi_1,\xi_2)\in\R^{2}$ satisfying
 	\begin{equation} 
  \label{limsup property}
 	\limsup_{t\to\infty} t^n \psi_{\ux}^{1\times 2}(t)= c,
 	\end{equation}
 	and moreover
 	\begin{equation} 
  \label{liminf property}
 	\liminf_{t\to\infty} t^n \psi_{\ux}^{1\times 2}(t)= 0.
 	\end{equation}
 \end{theorem}
 
\begin{remark}
	When $n\ge 3$, the liminf statement \eqref{liminf property} follows from the limsup claim \eqref{limsup property}, for example
	by Marnat and Moshchevitin~\cite{mamo}. The linear independence over $\Q$ of $(\ux,1)$
	is also implied by \eqref{limsup property}, so it is just stated for completeness.
\end{remark}

Observe that for $n>2$ this is not the natural order of 
decay as in Dirichlet's theorem, but considerably faster, so $\ux$ is singular.
This rate of decay will be required to increase the dimension to $n$.

\begin{proof}[Sketch of Proof of Theorem~\ref{thm2}]
We only provide a sketch of the proof since, as previously mentioned, the claim is a special case of~\cite[Theorem~1.1]{j2} for $n=2$. 
As the proof of this result was only sketched in~\cite{j2}, we briefly recall the construction for convenience of the reader, with $\Phi(t)=ct^{-n}$.
Define an integer sequence $(a_j)_{j\ge 1}$ as follows: Take
the initial two terms $a_1, a_2$ large enough with $a_1|a_2$. Now iteratively,
for $k\geq 1$ having constructed the first $2k$ terms $a_1,\ldots,a_{2k}$, let
the next two terms are given by the recursion
\begin{equation*} 
a_{2k+1}= a_{2k}^{M_k}
\end{equation*}
and 
\begin{equation*} 
a_{2k+2}=a_{2(k+1)}= a_{2k+1}\cdot \lceil c^{-1}a_{2k+1}^{n-1}\rceil\in (a_{2k+1}^n,\infty),
\end{equation*}
with integers $M_k\to\infty$ that tend to be infinitely fast enough. It turns 
out that $\ux=(\xi_1,\xi_2)\in \R^{2}$ with 
\[
\xi_j= \sum_{k=0}^{\infty} a_{2k+j}^{-1}, \qquad j=1,2,
\]
has the desired properties of Theorem~\ref{thm2}.
\end{proof}
\begin{remark}
An alternative construction for $c$ small enough can be derived from the method in~\cite{j3}, we will need the latter for the case of arbitrary norms. Yet another proof, for arbitrary norms, can be found in the aforementioned preprint of Agin and Weiss~\cite{AginWeiss2024}.
\end{remark}

Now we want to blow up the vectors in Theorem~\ref{thm2} from $\R^2$ to $\R^n$ by adding components $\xi_3,\ldots,\xi_n$. Let us call the new vector(s) $\uz\in\R^n$. The following is an easy observation.

\begin{proposition}  \label{pro}
	Take any $\ux\in\R^2$ as in Theorem~\ref{thm2}. 
	Any $\uz\in\R^n$ with first two coordinates $\xi_1=\zeta_1, \xi_2=\zeta_2$,
	that is any vector $\uz\in \{ \ux\}\times \R^{n-2}\subseteq \R^n$, 
	satisfies $\uz\in Di_n(c)\setminus Bad_n$. 
\end{proposition}

\begin{proof}
	Let us write $\ux^{\ast}= (\xi_1, \xi_2, 1)$ and similarly 
	$\uz^{\ast}=(\zeta_1,\ldots,\zeta_n,1)$. Denote $\ee_i=(0,0\ldots,0,1,0,\ldots,0)$ the $i$-th canonical base vector
	of $\R^{n+1}$. \par 
 Our claim follows from the fact that 
	\begin{equation} \label{psi decreasing}
	\psi_{\uz}^{1\times n}(t) \le \psi_{\ux}^{1\times 2}(t),\qquad t>0.
	\end{equation}
Indeed, any $t$ induces a best approximation $\bb=(b_{1}, b_{2}, b_{3})\in \Z^3$ with respect to $\ux$ that realizes $\psi_{\ux}(t)$. Then the blown-up vector
\[
\tilde{\bb}:= b_{1} \ee_1 + b_{2} \ee_2+ b_{3} \ee_{n+1}=(b_1, b_2, 0,\ldots,0,b_{3})\in \Z^{n+1}
\]
induces the same maximum norm and approximation quality
\[
\Vert \hat\bb\Vert= \Vert \hat{\tilde\bb}\Vert,  \qquad 
| \tilde{\bb} \cdot \uz^{\ast}|= | \bb \cdot \uxs|= \psi_{\ux}^{1\times 2}(t).
\]  
so the minimum over all integer points defining $\psi_{\uz}(t)$ cannot be larger. Trivially \eqref{liminf property} and \eqref{psi decreasing} implies 
\begin{equation*}
    \liminf_{t\to \infty} t^{n}\psi_{\uz}^{1\times n}(t)=0\, ,
\end{equation*}
so $\uz \not \in Bad_{n}$. Similarly \eqref{limsup property} and \eqref{psi decreasing} implies $\limsup_{t\to \infty} t^{n}\psi_{\uz}^{1\times n}(t)\leq c$, so $\uz \in Di_{n}(c)$.
\end{proof}

Proposition~\ref{pro} does not imply that $\uz \not \in Sing_{n}$. For example, if $n\ge 3$ taking $\uz=(\xi_{1},\xi_{2}, \xi_{1}+\xi_{2}, 0, \dots, 0)\in\R^{n}$ then $\uz \in Sing_{n}$. However, the following Theorem enables us to say that this does not happen for a significantly large subset of $\R^{n-2}$ of ``additional coordinates''. In fact, we can ensure the induced 
Dirichlet constant is not smaller than the one of $\ux$ for generic vectors
of $\R^{n-2}$.

\begin{theorem} \label{thm3}
	Let $n\ge 4$ be an integer. Take any $\ux\in\R^2$ as in Theorem~\ref{thm2} derived from any small enough $c\in(0,c(n))$ as in Theorem~\ref{thm1}. Then for a set 
	\[
	\mathcal{Y}=\mathcal{Y}_{n,\ux}\subseteq \R^{n-2}
	\]
	of positive $n-2$ dimensional Lebesgue measure consisting of $\zeta_3,\ldots,\zeta_n$, the vector 
	$\uz=(\xi_1,\xi_2,\zeta_3,\ldots,\zeta_n)\in\R^n$ satisfies
	\[
	\Theta^{1\times n}(\uz)=\limsup_{t\to\infty} t^n \psi_{\uz}^{1\times n}(t)=\limsup_{t\to\infty} t^n \psi_{\ux}^{1\times 2}(t)= c,
	\]
    or equivalently $\uz\in DI_n(c)$.
\end{theorem} 
\begin{remark}  \label{remarke}
	Presumably the $(n-2)$-dimensional Lebesgue measure of the set $\mathcal{Y}$ is full, however, this is not required for our purposes. 
\end{remark}

As was the case for Theorem~\ref{thm1}, we believe the claim should hold for $n=3$ as well. We will prove Theorem~\ref{thm3} in the next section. Armed with this result we are able to prove our Theorem~\ref{thm1}.

\begin{proof}[Proof of Theorem~\ref{thm1}]
By combining Proposition~\ref{pro} and Theorem~\ref{thm3} it is clear that
\[
\uz\in DI_n(c)\setminus Bad_n= FS_n(c)\, .
\]
Let $\mathcal{Z}_{c}$ denote the set of $\ux \in \R^{2}$ satisfying Theorem~\ref{thm2}. Since $\lambda_{n}(\mathcal{Y})>0$, for any $c\in(0,c_n)$ we have that
\[
\dim_H FS_n \ge \dim_H FS_n(c)  \ge \dim_H \mathcal{Z}_{c}\times \mathcal{Y} \geq  \dim_H(\mathcal{Z}_{c})+ \dim_H \mathcal{Y}  \geq  n-2\, .
\]
The second to last inequality follows from a standard result on the Hausdorff dimension of Cartesian products
\begin{equation}  \label{eq:tric}
    \dim_H A\times B \ge \dim_H A  + \dim_H B 
\end{equation}
for $A,B$ arbitrary sets in Euclidean spaces, see~\cite{tricot}. Hence Theorem~\ref{thm1} is proven modulo Theorem ~\ref{thm3}. 
\end{proof}


\begin{remark}  \label{rrr}
Notice that in the above calculation, we simply use $\dim_H  \mathcal{Z}_{c}\geq 0$. It is here where our lower bound could be improved. 
For $\tau>2$ a parameter, define 
\begin{equation*}
    Sing_{2}(\tau):=\left\{ \ux \in \R^{2}: \limsup_{t\to \infty} t^{\tau}\psi_{\ux}^{1\times 2}(t)=0 \right\}.
\end{equation*}
It can be deduced from~\cite[Theorem 1.4-1.9]{DFSU_singular} by a change of variables
that for any $\tau>2+\sqrt{2}$ we have
\begin{equation*}
     \dim_H Sing_{2}(\tau)=\frac{2}{\tau}.
\end{equation*}
Clearly $\mathcal{Z}_{c} \subseteq Sing_{2}(n-\varepsilon)$ for any $\varepsilon>0$, and so
	\[
	\dim_H  \mathcal{Z}_{c} \leq \frac{2}{n-\varepsilon}.
	\]
 Thus for $n\ge 4>2+\sqrt{2}$ the best we could hope for via this method is
 \begin{equation*}
     \dim_H FS_{n}(c)\geq n-2 + \frac{2}{n}\, .
 \end{equation*}
  On the other hand, a positive lower bound for $\dim_H (\mathcal{Z}_{c})$ can be obtained by a similar variational method as in~\cite[Theorem~2.1]{j3},
 we do not carry it out.

\end{remark}

\subsection{Proof of Theorem~\ref{thm3}} \label{Sect6}

We have to show the Dirichlet constant of many $\uz$  
is at least $c$, the rest is obvious.
We follow ideas from the proof of Moshchevitin \cite[Theorem 12]{ngm} with some twists. Write $\uz=(\xi_{1},\xi_{2},\zeta_{3}, \dots , \zeta_{n})$ and restrict to $(\zeta_{3},\dots, \zeta_{n})$ in the $n-2$-dimensional unit ball $B_{n-2}(\boldsymbol{0},1)\subseteq \R^{n-2}$ with respect to the Euclidean norm (this will be slightly easier to handle than for the maximum norm). Denote by 
\[
\boldsymbol{b}_v=(b_{1,v},b_{2,v},b_{3,v})\in\Z^3, \qquad v\ge 1,
\]
the best approximations with respect to 
$\ux=(\xi_1,\xi_2)$ and derive 
\[
M_v= \max_{j=1,2} |b_{j,v}|= \|\hat{\bb}_{v}\|, \qquad v\ge 1
\] 
and the minimal values
\[
L_v= |b_{1,v} \xi_1 + b_{2,v} \xi_2 + b_{3,v}|, \qquad v\ge 1.
\] 
Clearly $(L_{v})_{v\geq 1}$ is a strictly decreasing sequence and $(M_{v})_{v\geq 1}$ is strictly increasing.
Further, note that we have
\[
\psi_{\ux}^{1\times 2}(t)= L_v, \qquad t\in [M_v, M_{v+1}),
\]
hence
\begin{equation} \label{eq:tT} \limsup_{t\to \infty} t^{n}\psi_{\ux}^{1\times 2}(t) = \limsup_{v\to\infty} L_v M_{v+1}^{n}.
\end{equation}

Denote $\ee_i=(0,0\ldots,0,1,0,\ldots,0)$ the $i$-th canonical base vector
of $\R^{n+1}$. 
To finish the proof, it suffices
to show the following: 

\begin{lemma} \label{lema}
	Let $n\ge 4$ an integer and $\ux=(\xi_1,\xi_2)$ as in Theorem~\ref{thm2}.
	For a positive measure set $\mathcal{Y}\subseteq B_{n-2}(0,1)\subseteq \R^{n-2}$ consisting
	of $\boldsymbol{\gamma}=(\xi_3,\ldots,\xi_n)$, we can choose an infinite subsequence of
	integers $v_r, r\ge 1$, so that 
    letting $T_{v}:=M_{v+1}-1/2$ for $v\ge 1$, we have
      \begin{equation} \label{eq:igno}
    \lim_{r\to\infty} T_{v_r}^n\cdot \psi_{\ux}^{1\times 2}(T_{v_{r}})
    = \limsup_{v\to\infty} T_{v}^n\cdot \psi_{\ux}^{1\times 2}(T_{v})
    =c
    \end{equation}
    i.e. this subsequence realizes the limsup, and 
	the integer best approximation  
	for $\uz=(\ux, \boldsymbol{\gamma})$ at $T_{v_{r}}$ is given as
	\[
	b_{1,_v{r}}\ee_1+  b_{2,v_r}\ee_2+   b_{1,v_r}\ee_{n+1}=
	(b_{1,v_{r}},b_{2,v_r},0,\ldots,0,b_{3,v_r})\in \Z^{n+1}\,, \qquad r\ge 1.
	\]
	Thus the best approximation function at values $T_{v_r}$ equals
	\[
	\psi_{\uz}^{1\times n}(T_{v_{r}})=\psi_{\ux}^{1\times 2}(T_{v_{r}})= L_{v_r},\qquad r\ge 1.
	\] 
\end{lemma}

Observe that since $T_v$ is of the same order as $M_{v+1}$, for $\uz$ as in the lemma,  \eqref{eq:tT}, \eqref{eq:igno} and the corresponding property for $\uz$ readily imply
\[
\Theta^{1\times n}(\uz)\ge\limsup_{v\to\infty} L_v T_{v}^{n}=\limsup_{t\to \infty} t^{n}\psi_{\ux}^{1\times 2}(t),
\]
the reverse inequality is clear by Proposition~\ref{pro}.
Hence, assuming the lemma is true, the claim of Theorem~\ref{thm3} follows.

We prove the preceding lemma via some modification of the method in~\cite{ngm}.

\begin{proof}[Proof of Lemma~\ref{lema}]
 We first note that we may assume
\[
\lim_{v\to\infty} T_v^{n}\psi_{\ux}^{1\times 2}(T_v)=c
\]
and hence ignore \eqref{eq:igno}. Indeed, otherwise we restrict to a subsequence
satisfying \eqref{eq:igno}, which exists by definition of the limsup, and apply the argument below for this subsequence.\par 
We show
that for a positive portion of $B_{n-2}(\boldsymbol{0},1)\subseteq \R^{n-2}$
of vectors $\boldsymbol{\gamma}=(\zeta_3,\ldots,\zeta_n)$,
we have for infinitely many $v$ that
\begin{equation} \label{eq:nega}
\min \left| x_{1}\xi_{1}+x_{2}\xi_{2}+ \sum_{i=3}^{n} x_i \zeta_i - y \right| \ge L_{v}
\end{equation}
with minimum taken over all integer points
\begin{equation} \label{eq:ob}
\boldsymbol{z}= (x_1,\ldots,x_{n},y)\in \Z^{n+1}\setminus \{\boldsymbol{0}\}
\quad \text{ with } \quad 
0<\max_{1\le j\le n} |x_j| \le T_v\, . 
\end{equation}
For $v\ge 1$, denote by $\mathcal{S}_v$ the set of vectors of $\R^{n-2}$
taken over all integer points $\boldsymbol{z}$ obeying \eqref{eq:ob}, i.e.
\begin{equation*}
    \mathcal{S}_{v}:=\left\{ \ug \in B_{n-2}(\boldsymbol{0},1): \underset{0<\max_{1\leq i \leq n}|x_{j}|\leq T_{v}}{\min_{(\boldsymbol{x},y)\in\Z^{n+1}:}}
     \left| x_{1}\xi_{1}+x_{2}\xi_{2}+\sum_{i=3}^{n}x_{i}\zeta_{i} -y\right|\geq L_{v} \right\}.
\end{equation*}
Then our set is the limsup set of the $\mathcal{S}_v$, or its complement
is the set of all vectors that lie only in finitely many $\mathcal{S}_v$, that is, 
\begin{equation} \label{eq:wia}
B_{n-2}(\boldsymbol{0},1)\setminus \mathcal{Y}\subseteq 
\bigcup_{i\ge 1} \bigcap_{j\ge i} \mathcal{S}_j^{c}\subseteq \bigcup_{i\ge 1} \mathcal{S}_i^{c}.
\end{equation}
Assume that any complement 
\[
\mathcal{U}_v:= \mathcal{S}_v^{c}= \left\{ \ug \in B_{n-2}(\boldsymbol{0},1): \underset{0<\max_{1\leq i \leq n}|x_{j}|\leq T_{v}}{\min_{(\boldsymbol{x},y)\in\Z^{n+1}:}}
     \left| x_{1}\xi_{1}+x_{2}\xi_{2}+\sum_{i=3}^{n}x_{i}\zeta_{i} -y\right|< L_{v} \right\} , \quad v\ge 1,
\]
has $n-2$ dimensional Lebesgue measure less than some $\delta$ strictly smaller than the volume of $B_{n-2}(\boldsymbol{0},1)$. 
Then this will also be 
true for $B_{n-2}(\boldsymbol{0},1)\setminus \mathcal{Y}$ by the following easy result.

\begin{proposition}  \label{pp}
	Let $\lambda$ be a measure on a space $\mathcal{T}$.
	Assume $H_i$ are increasing subsets of $\mathcal{T}$, i.e.
	$H_1\subseteq H_2 \subseteq \cdots$. Then
	\[
	\lambda\left(\bigcup_{k=1}^{\infty} H_k\right) = \lim_{k\to\infty} \lambda(H_k).
	\]
\end{proposition}

We will require the estimate $\lambda\left(\bigcup_{k=1}^{\infty} H_k\right) \le \lim_{k\to\infty} \lambda(H_k)$ only.

\begin{proof}
	Writing $G_j=H_j\setminus H_{j-1}$ with $H_0=\emptyset$, we see
    \begin{equation*}
	\lambda\left(\bigcup_{k=1}^{\infty} H_k\right)= \lambda\left(\bigcup_{k=1}^{\infty} G_k\right)= \lim_{N\to\infty} \sum_{k=1}^{N} \lambda(G_k)= \lim_{N\to\infty}  \lambda\left(\bigcup_{k=1}^{N} H_k\right)= \lim_{N\to\infty}  \lambda(H_N).
	\end{equation*}
	\end{proof}

In the sequel denote by $\lambda$ the $(n-2)$-dimensional Lebesgue measure of a set $A\subseteq \R^{n-2}$.
As indicated before, we apply Proposition~\ref{pp} to
\[
H_i= \bigcap_{j\ge i} \mathcal{S}_j^{c}= \bigcap_{j\ge i} \mathcal{U}_j \subseteq \mathcal{U}_i, \qquad i\ge 1,
\]
to see via \eqref{eq:wia} that
\[
\lambda(B_{n-2}(\boldsymbol{0},1)\setminus \mathcal{Y})\le \lim_{v\to\infty} \lambda(\mathcal{U}_v).
\]
Assume we have shown that for some uniform $\Delta$ we have
\begin{equation}  \label{eq:IND}
\lambda(\mathcal{U}_v)\le \Delta< \lambda(B_{n-2}(\boldsymbol{0},1)), \qquad v\ge 1.
\end{equation}
Then we conclude that $$\lambda(B_{n-2}(\boldsymbol{0},1)\setminus \mathcal{Y})\le \Delta<\lambda(B_{n-2}(\boldsymbol{0},1))$$ and thus
\[
\lambda(\mathcal{Y}) = \lambda(B_{n-2}(\boldsymbol{0},1)) - \lambda(B_{n-2}(\boldsymbol{0},1)\setminus \mathcal{Y})\ge \lambda(B_{n-2}(\boldsymbol{0},1)) - \Delta > 0,
\]
as desired.

To verify \eqref{eq:IND},
let us estimate $\lambda(\mathcal{U}_v)$ for fixed $v$. For any
 $\ug=(\zeta_3,\ldots,\zeta_n)\in \mathcal{U}_v$ it follows that 
\begin{equation*}
    \underset{0<\max_{1\leq i \leq n}|x_{j}|\leq T_{v}}{\min_{(\boldsymbol{x},y)\in\Z^{n+1}:}}
     \left| x_{1}\xi_{1}+x_{2}\xi_{2}+\sum_{i=3}^{n}x_{i}\zeta_{i} -y\right|< L_{v}\, .
\end{equation*}
Equivalently
\[
x_3 \zeta_3 + \cdots+ x_n \zeta_n \in J_{v}(x_1,x_2,y)
\]
where $J_{v}(.)$ is given as the interval
\[
J_v(x_1,x_2,y)= ( -y-x_1\xi_1-x_2\xi_2-L_v , -y-x_1\xi_1-x_2\xi_2+L_v  ).
\]
Let 
\[
\Omega_v(x_1,\ldots, x_n,y)=
\left\{ (\zeta_3,\ldots,\zeta_n)\in B_{n-2}(0,1)\subseteq \R^{n-2}: \sum_{i=3}^{n}x_{i}\zeta_{i} \in J_{v}(x_1,x_2,y) \right\}.
\]
Observe that $\Omega_v(x_1,\ldots,x_n,y)$ is empty if $|y|>|x_{1}\xi_{1}|+|x_{2}\xi_{2}|+\sum_{i=3}^{n} |x_i \zeta_i|+L_v$.
Fix for now $\varepsilon>0$.
Since $L_v\to 0$ as $v\to\infty$ and as we may take $\xi_1,\xi_2$
arbitrarily small positive numbers (the $\ux$ from Theorem~\ref{thm2} are easily seen to be dense in $\R$), say smaller than $\varepsilon/8$,
and $(\zeta_3,\ldots,\zeta_n)\in B_{n-2}(\boldsymbol{0},1)$, Cauchy-Schwarz inequality 
yields for any $\varepsilon>0$ this is true if
\[
|y|> (\sqrt{n-2}+\varepsilon/2) T_v
\]
as then
\begin{align*}
|y|&> (\sqrt{n-2}+\varepsilon/2) T_v \\
&\ge (\varepsilon/4)\cdot T_v+ \left(\sum_{i=3}^{n} |x_i|^2\right)^{1/2}\left(\sum_{i=3}^{n} |\zeta_i|^2\right)^{1/2}+L_v
\\ &\ge |\xi_1| |x_1| + |\xi_2| |x_2|+ \left(\sum_{i=3}^{n} |x_i|^2\right)^{1/2}\left(\sum_{i=3}^{n} |\zeta_i|^2\right)^{1/2}+L_v
\\
&\ge |x_{1}\xi_{1}|+|x_{2}\xi_{2}|+ \sum_{i=1}^{n} |x_i\zeta_i|+L_v, \qquad v\ge v_0(\varepsilon).
\end{align*}
The above shows that if we let
\begin{equation}  \label{eq:inu}
\Omega_v= \bigcup_{x_1,\ldots,x_n} \bigcup_{y}\; \Omega_v(x_1,\ldots,x_n,y),
\end{equation}
with unions are taken over integers $x_1, \dots, x_{n}$ as in \eqref{eq:ob} and integers $y$ with $|y|\le (\sqrt{n-2}+\varepsilon/2)T_v$. Then
\[
\mathcal{U}_v \subseteq \Omega_v, \qquad v\geq v_0(\varepsilon).
\]
The distance between the parallel planes defining $\Omega_v(x_1,\ldots,x_n,y)$ equals
$2(x_3^2+\cdots+x_n^2)^{-1/2}$.
Thus each $\Omega_v(x_1,\ldots,x_n,y)$ is contained in a cylinder
of radius one and height $2L_v \cdot (x_3^2+\cdots+x_n^2)^{-1/2}$. Denoting
by $\lambda^{\prime}$ the $(n-3)$-dimensional Lebesgue measure, this has volume 
\begin{align} \label{volcume calculation}
\lambda(\Omega_v(x_{1},\ldots,x_{n},y))&= 2\lambda^{\prime}(B_{n-3}(\boldsymbol{0},1))\cdot L_v(x_3^2+\cdots+x_n^2)^{-1/2}\nonumber \\ &\le 2 L_v\cdot\lambda^{\prime}(B_{n-3}(\boldsymbol{0},1)) \cdot\frac{ 1}{\max_{3\le i\le n} |x_i|}. 
\end{align}
Since, assuming $T_v\ge \varepsilon^{-1}$,
there are at most $$2T_{v}(\sqrt{n-2}+\varepsilon/2)+1\le T_{v}(2\sqrt{n-2}+2\varepsilon)$$
integers $y$ in the inner union of \eqref{eq:inu}, in 
total we get
\begin{align*}
\lambda(\mathcal{U}_v) &\le \lambda(\Omega_v)\\ &\le 
\sum_{x_1, x_2} \sum_y \sum_{x_3,\ldots,x_n} \lambda(\Omega_v(x_1,\ldots,x_n,y))\\ &\le
 2(2\sqrt{n-2}+2\varepsilon)L_vT_v \lambda^{\prime}(B_{n-3}(\boldsymbol{0},1)) \cdot  \sum_{x_1, x_2} \sum_{x_3,\ldots,x_n} \frac{ 1 }{ \max_{3\le i\le n} |x_i| }
\end{align*}
for $v\ge v_0(\varepsilon)$.
The outer sum over $x_1, x_2$ gives a factor $\lfloor T_v\rfloor^2\le T_v^2$ by \eqref{eq:ob}. We
split the inner sum into sets where $\max_{3\le i\le n} |x_i|=k\in [1,\lfloor T_v\rfloor]$ is constant, which correspond to $2(n-2)=2n-4$ faces
on a hypercube for each $k$ according to which variable equals $\pm k$, giving
in total $(2n-4)k^{n-3}$ integer points for each $k$ to sum $1/k$ over. 
Hence we may further estimate
\begin{align*}
\lambda(\mathcal{U}_v) &\le 2(2n-4)(2\sqrt{n-2}+2\varepsilon) \lambda^{\prime}(B_{n-3}(\boldsymbol{0},1))\cdot  L_v T_v^3 \cdot \sum_{k=1}^{\lfloor T_v\rfloor} k^{n-4}\\ &\le 2(2n-4)(2\sqrt{n-2}+2\varepsilon) \lambda^{\prime}(B_{n-3}(\boldsymbol{0},1)) \left(\frac{1}{n-3}+\varepsilon_2\right) L_v T_v^{n},\quad v\ge v_1(\varepsilon,\varepsilon_2).
\end{align*}
Here the error term $\varepsilon_2>0$ tends to $0$ as $v\to\infty$, and we used $n\ne 3$ to avoid a logarithmic term appearing in the last sum. 
We may assume $\varepsilon_2< \varepsilon$ for all large $v$.
So if $L_v< c T_v^{-n}$ for any $c\in(0,c_n-\epsilon)$ for 
$\epsilon>0$ some small manipulation of
$\varepsilon$ and $c_n$ satisfies
\[
\lambda(B_{n-2}(\boldsymbol{0},1))> c_n\frac{ 4(2n-4)\sqrt{n-2}\lambda^{\prime}(B_{n-3}(0,1)) }{ n-3 },
\]
or equivalently
\[
c_n < \frac{n-3}{4(2n-4)\sqrt{n-2}}\cdot \frac{\lambda(B_{n-2}(\boldsymbol{0},1))}{\lambda^{\prime}(B_{n-3}(\boldsymbol{0},1))}=
\frac{n-3}{8(n-2)^{3/2}}\cdot \sqrt{\pi}\cdot \frac{ \Gamma(n-\frac{1}{2})}{\Gamma(n)},
\]
then \eqref{eq:IND} holds for some $\Delta< \lambda(B_{n-2}(\boldsymbol{0},1))$,
depending on $c_n$, and all $v\ge v_0(\Delta) = v_0(c_n)$. Hence
the complement contained in $\mathcal{Y}$ has positive $(n-2)$-dimensional
Lebesgue measure. Since
we can choose $\varepsilon>0$, and thus $\epsilon>0$, arbitrarily small we may assume that $c<c_n$ and
the proof of Lemma~\ref{lema} is complete.
\end{proof}


\begin{remark}
	Improvements of $c_n$ can be made by sharpening the contribution $2\sqrt{n-2}+1$ for the number of $y$. Indeed,
	for most $y$ the obtained intersections with the unit ball have a considerably smaller volume than $2L_v \cdot (x_3^2+\cdots+x_n^2)^{-1/2}$, so it seems we can improve by more concise estimates.
\end{remark}

\section{Proof of Theorem~\ref{jia} } \label{Sect7}

\subsection{Outline}
The method of proof is similar to that of Theorem~\ref{thm1}. Namely, the proof is done in two steps:
\begin{enumerate}
\item[Step I:] For small enough $c$, construct $\ux=(\xi_1, \xi_2)$ such that
\[
\limsup_{t\to\infty} t^n \psi_{\ux}^{1\times 2}(\Vert\cdot\Vert,t) = c.
\]
\item[Step II:] Show that for a positive measure set $(\xi_3,\ldots,\xi_n)$ with respect
to $(n-2)$-dimensional Lebesgue measure, the extended vector $\uz=(\xi_1,\ldots,\xi_n)$ satisfies
\[
\Theta^{1\times n}(\uz)=\limsup_{t\to\infty} t^n \psi_{\uz}^{1\times n}(\Vert\cdot\Vert,t) = c
\]
 as well.
\end{enumerate}
 Note that Step I follows immediately, for
any $c\ge 0$ (and $c\in[0,D]$ if $n=2$),
from the aforementioned preprint \cite[Corollary~1]{AginWeiss2024}. Note this follows from the result on  $\psi$-Dirichlet Spectrum, rather than the classical Dirichlet Spectrum, since our approximation function is $t^{n}$. However, as this appeared later than the current article, we keep our original construction. Note that \cite{AginWeiss2024} does not simplify Step II of the proof.

\subsection{Proof of Step I}  \label{st1}
This is the more intricate part of the proof. For this, we use a different construction than for the maximum norm where we employed the construction from~\cite{j2}. 
Rather we follow the proof of~\cite[Theorem~2.1]{j3} but with some notable twists. 

Write $\uxs=(\ux,1)\in\R^{3}$ and use a similar meaning of star for other 
quantities. Conversely denote by a hat over a vector in $\R^{3}$ its restriction to the first
$2$ coordinates by chopping off the last coordinate, i.e. $\widehat{\uxs}=\ux$.

We construct our real vector. 
Let $n<\tau<\mu$ be parameters to be chosen later related by the identity
\begin{equation} \label{eq:IIi}
\tau \mu-1=n\mu \quad \Longleftrightarrow \quad \mu=\frac{1}{\tau-n}.
\end{equation}
Let $\alpha_j, \beta_j, \gamma_j, \delta_j$ be strictly increasing positive integer sequences and a real number $r=r(\Vert\cdot\Vert)>0$ all to be fixed later. 
Derive integers of the form
\[
A_j= 2^{\alpha_j}3^{\gamma_j }, \qquad B_j= 5^{\beta_j }7^{\delta_j},
\]
satisfying as $j\to\infty$
\begin{equation} \label{eq:accto}
A_{j+1}= (1+o(1))r B_j^{\tau}, \qquad B_j = (1+o(1))A_j^{\mu}.
\end{equation}
We first show that such choices are possible. For $A_j$, we use that $\{ \log 2, \log 3\}$ is a $\Q$-linearly independent set (easily seen by applying exponential map to a putative vanishing linear form and using unique prime factorization to disprove it) and the following easy Proposition~\ref{proper}. Similarly 
for $B_j$ we use that $\{ \log 5, \log 7\}$ is $\Q$-linearly independent.

\begin{proposition}  \label{proper}
	Let $a,b,c$ be real numbers with $a,b$ linearly independent over $\Q$. Then ordering 
	the set $\{ k_1 a+k_2 b+c: k_i\in\mathbb{N} \}$ increasingly and denoting
	it by $(z_i)_{i\ge 1}$ we have $|z_{i+1}-z_i|\to 0$.
\end{proposition}

\begin{proof}
	Clearly, we may assume $c=0$ and suppose without loss of generality that $b<a$. We may then equivalently 
	consider 
 \begin{equation*}
     \mathcal{Z}=\{ k_1+k_2\tfrac{b}{a}: k_i\in\mathbb{N} \}
 \end{equation*}
 instead. Since $\tfrac{b}{a}$ is irrational,	Kronecker's Theorem (see e.g. \cite[Section 3.5 Theorem IV]{Cassels}) shows that the set is dense in the unit interval when considering all
	$(k_1, k_2)\in \Z\times \N$. Let $\varepsilon=1/K>0$ with a large integer $K$ and partition $[0,1]$ into $K$ intervals
 \begin{equation*}
 I_t=[(t-1)\varepsilon,t\varepsilon], \qquad 1\le t\le K   
 \end{equation*}
 of length $1/K$. Then for any $I_t$ there exist
	 $(k_1^{\prime}(t), k_2^{\prime}(t))\in\Z\times\N$ so that $k_1^{\prime}(t)+k_2^{\prime}(t)\tfrac{b}{a}\in I_t$. Take
	$T_{\varepsilon}:= \max_t |k_1^{\prime}(t)|$ and consider the set 
 \begin{equation*}
 \mathcal{Z}_{T_{\varepsilon}} = \left\{ (k_{1}(t)+N) + k_{2}(t)\tfrac{b}{a} : 1\leq t\leq K, \, N \in \N_{\geq T_{\varepsilon}} \right\} \subset \mathcal{Z}\, .
\end{equation*}
	It is easy to see that for any $M\in\N_{\geq T_{\varepsilon}}$ the interval $M+I_{t}$ contains a point from $\mathcal{Z}_{T_{\varepsilon}}$.
	This shows that $\mathcal{Z}$ is $\varepsilon$ dense on some interval $[3T_{\varepsilon},\infty)$. Since $\varepsilon>0$ 
	is arbitrary the claim is proved.
	\end{proof}


We can now apply the proposition alternatingly
to
\[
a=\log 2, \, b=\log 3, \, c=\log r, \qquad\quad a=\log 5,\, b=\log 7, \, c=0
\]
respectively to construct iteratively our sequences $\alpha_j, \gamma_j$, which correspond to $k_1=\alpha_{j+1}-\alpha_j\ge 0$ and 
$k_2=\gamma_{j+1}-\gamma_j\ge 0$, resp. $k_1=\beta_{j+1}-\beta_j\ge 0$ and 
$k_2=\delta_{j+1}-\delta_j\ge 0$. 
Thereby, having chosen all $\alpha_., \beta_., \gamma_.,\delta_.$ up 
to index $j$, we construct first pairs $\alpha_{j+1}, \gamma_{j+1}$ and then $\beta_{j+1}, \delta_{j+1}$, and the iterative construction is complete.
So, having chosen $r$, we may assume all the above is defined.

Note that \eqref{eq:accto} implies 
\begin{equation}  \label{eq:true}
A_{j+1}=(1+o(1))r A_j^{\tau \mu}, \qquad B_{j+1} =(1+o(1))r^{\mu} B_j^{\tau \mu}.
\end{equation}
Moreover, again by \eqref{eq:accto}, upon changing initial terms if necessary, we can assume
\begin{equation} \label{eq:vf}
1<A_1 < B_1 < A_2 < B_2 < \cdots.
\end{equation}
Finally let $\ux=(\xi_1, \xi_2)$ with
\[
\xi_1= \sum_{j\ge 1} A_j^{-1 } , \qquad \xi_2= \sum_{j\ge 1} B_j^{-1 }.
\]
We claim that it satisfies the assertions of the theorem upon choosing $r$ appropriately.

For $j\ge 1$ put
\[
\sigma_j := A_1^{-1}+\cdots+A_j^{-1}, \qquad \eta_j:= B_1^{-1}+\cdots+B_j^{-1}.
\]
We clearly have $\sigma_j\to \xi_1$ and $\eta_j\to \xi_2$ as $j\to \infty$, 
and 
\[
F_j= A_j\sigma_j\in \mathbb{N}, \qquad G_j=B_j \eta_j\in\mathbb{N}.
\]
Then since the sequences $(\alpha_j), (\beta_j), (\gamma_j), (\delta_j)$ are strictly increasing
we have $F_j\equiv 1\bmod 2\cdot 3$
and $G_j\equiv 1\bmod 5\cdot 7$, which imply the coprimality assertions
\begin{equation}  \label{eq:witt}
(A_j, F_j)=(F_j,6)=1, \qquad (B_j,G_j)=(G_j,35)=1.
\end{equation}
Let
\[
\vv_j=(A_j,0,-F_j), \qquad 
\w_j=(0, B_j,-G_j).
\]
Note that $\hat\vv_j=(A_j,0)$ and $\hat \w_j=(0,B_j)$.
Hence their norms $\Vert\cdot\Vert$ satisfy
\begin{equation} \label{eq:gl}
\Vert\hat\vv_j\Vert = \Vert \ee_1 \Vert A_j= d_1 A_j, \qquad d_1:=\Vert \ee_1\Vert=\Vert (1,0)\Vert,
\end{equation}
and
\begin{equation}  \label{eq:lg2}
\Vert\hat\w_j\Vert= \Vert \ee_2 \Vert B_j= d_2 B_j \overset{\eqref{eq:accto}}{=}(1+o(1))d_2 A_j^{\mu}, \qquad d_2:=\Vert \ee_2\Vert=\Vert (0,1)\Vert,
\end{equation}
Thus by \eqref{eq:vf} clearly
\begin{equation}  \label{eq:either}
\Vert\hat\vv_j\Vert< \Vert\hat\w_j\Vert< \Vert\hat\vv_{j+1}\Vert, \qquad j\ge j_0.
\end{equation}
Here $j_0$ depends on $d_1, d_2$, thus on the chosen norm.
Then by \eqref{eq:true} moreover 
\begin{align} \label{eq:ti}
|\vv_j\cdot \uxs|&= |A_j \cdot \xi_1 +0\cdot \xi_2 - F_j|= A_j(A_{j+1}^{-1}+A_{j+2}^{-1}+\cdots) \\ &= A_j A_{j+1}^{-1}(1+o(1))  \overset{\eqref{eq:true}}{=}  (1+o(1))r^{-1}A_j^{-(\tau \mu-1) }  \nonumber
\end{align}
and
\begin{align} \label{eq:tz}
|\w_j\cdot \uxs|&= |0\cdot \xi_1 + B_j \cdot \xi_2 - G_j|= B_j(B_{j+1}^{-1}+B_{j+2}^{-1}+\cdots)  \\ 
&= (1+o(1))B_jB_{j+1}^{-1} \overset{\eqref{eq:true}}{=}  (1+o(1))r^{-\mu}\cdot B_j^{-(\tau \mu-1) },  \nonumber
\end{align}
are small linear forms for $j\ge 1$.
We next show 

\begin{lemma}  \label{wichtig!}
	 As soon as $r$ is large enough, all best approximations
	 for $\ux$ of large norm have up to sign one of the following forms 
	 \[
	 \vv_j, \quad \w_j, \quad \vv_j+\w_j,\quad \w_j-\vv_j, \quad \w_j+\vv_{j+1},
	 \quad \vv_{j+1}-\w_{j}.
	 \]
\end{lemma}
 
\begin{proof}
Let $\bb$ be any best approximation. Then by \eqref{eq:either} 
there is an index $j$
such that either $\Vert\hat\vv_j\Vert\le \Vert \hat\bb\Vert<\Vert\hat\w_j\Vert$ or $\Vert\hat\w_j\Vert\le \Vert \hat\bb\Vert< \Vert\hat\vv_{j+1}\Vert$.  We show that in the first case $\bb=\pm \vv_j$ or $\pm(\vv_j\pm \w_j)$, and in the latter case $\bb=\pm \w_j$ or $\pm(\w_j\pm \vv_{j+1})$. 
Assume the first case, so
\begin{equation} \label{eq:T}
\Vert \hat\vv_j\Vert\le \Vert \hat\bb\Vert< \Vert \hat\w_j\Vert.
\end{equation}
the latter works very similarly by symmetry.
First, observe that since $\bb$ is the best approximation
of norm at least $\Vert\hat\vv_j\Vert$, we know that
\begin{equation}  \label{eq:see}
|\bb\cdot \uxs|\le |\vv_j\cdot \uxs|.
\end{equation}
We distinguish two cases.

Case 1: $\bb$ lies in the two-dimensonal subspace 
of $\R^3$ spanned by $\vv_j, \w_j$, i.e. 
$\bb\in \scp{\vv_j, \w_j}_{\mathbb{R}}\cap \mathbb{Z}^3$. The special form of $A_j, B_j$ and \eqref{eq:witt} imply the following crucial result on integer
vectors in the two-dimensional lattices $\scp{\vv_j, \w_j}_{\mathbb{R}}\cap \mathbb{Z}^3$.

\begin{proposition}  \label{prop}
	For $\vv_j, \w_j$ as above, if a linear combination
	$g\vv_j+h\w_j$ is an integer vector, then 
	in fact $g\in\Z$ and $h\in \Z$. In other words,
	$\scp{\vv_j, \w_j}_{\mathbb{R}}\cap \mathbb{Z}^3=\scp{\vv_j, \w_j}_{\mathbb{Z}}$.
\end{proposition}

\begin{proof}
	Clearly, we must have $g,h\in \mathbb{Q}$. 
	If we write $(p_1/q_1)\vv_j + (p_2/q_2)\w_j$ with $p_i/q_i$
	in lowest terms, then 
	it is clear that $q_1$ must consist of exclusively non-negative integer powers of $2$ and $3$ and $q_2$ of non-negative integer power of $5$ and $7$ to make the first two coordinates $(p_1/q_1)A_j=(p_1/q_1)2^{\alpha_j}3^{\gamma_j}$ resp. $(p_2/q_2)B_j=(p_2/q_2)5^{\beta_j}7^{\delta_j}$ of $g\vv_j+h\w_j$ integers. 
	But then by \eqref{eq:witt}
	clearly the third coordinate $(p_1/q_1)F_j+(p_2/q_2)G_j$ is not an integer unless
	$q_1=q_2=1$.  
\end{proof}

We first show that any integer linear combination $\bb=g\vv_j+h\w_j$ representing the best approximation as above must have $|g|\le 1$. 
Assume $|g|\ge 2$. 
It is easily seen that for any norm on $\R^2$ there is a constant $\theta=\theta(\Vert\cdot\Vert)$
such that $\Vert \w\Vert \le W$ for $\w=(w_1,w_2)\in\R^2$ and $W>0$ 
implies $|w_i|\le \theta W$, $i=1,2$. 
Hence \eqref{eq:T} and the special form of $\vv_j, \w_j$
imply
\[
|g|\le \theta, \qquad |h|\le \theta
\]
for some $\theta$ independent of $j$. On the other hand, since $r,\mu$ are fixed and $A_j\to\infty$, it is clear
from \eqref{eq:ti}, \eqref{eq:tz} and \eqref{eq:accto} that
\[
|\w_j\cdot \uxs| = o(|\vv_j\cdot \uxs|), \qquad j\to\infty. 
\]
Combination implies
\begin{align*}
|\bb \cdot \uxs|= | (g\vv_j+h\w_j) \cdot \uxs| &\ge |g|\cdot |\vv_j\cdot \uxs| - |h|\cdot |\w_j\cdot \uxs| \ge 2 |\vv_j\cdot \uxs| - |h|\cdot |\w_j\cdot \uxs|\\ &\ge
2 |\vv_j\cdot \uxs| - \theta\cdot |\w_j\cdot \uxs| > |\vv_j\cdot \uxs|, \qquad j\ge j_0,
\end{align*}
contradicting \eqref{eq:see}. Hence $|g|\le 1$. Now if $|h|\ge 2$ then 
for any $|g|\le 1$ using \eqref{eq:either} we get 
\[
\Vert \hat{\bb}\Vert= \Vert g \hat{\vv}_j+h\hat{\w}_j\Vert \ge |h|\cdot \Vert \hat{\w}_j\Vert - \Vert \hat{\vv}_j\Vert \ge 2 \cdot \Vert \hat{\w}_j\Vert - \Vert \hat{\vv}_j\Vert > \Vert \hat{\w}_j\Vert
\]
contradicting \eqref{eq:T}. Thus both $g,h$ are among $\{-1,0,1\}$. 
Hence we are left with the vectors of the lemma. 

Case 2: $\bb$ does not lie in the space spanned by $\vv_j, \w_j$.
For this case, we use an easy consequence of Minkowski's Second Convex
Body Theorem.

\begin{lemma}  \label{lemur}
	There exists a constant $c>0$ such that for any $\ux\in \R^2$ and
	any parameter $Q\ge 1$, the system
	\[
	|b_1|\le Q,\qquad |b_2|\le Q,\qquad  |b_1 \xi_1 + b_2 \xi_2+b_{3}| < cQ^{-2}
	\]
	does not have three linearly independent solutions in integer
	vectors $\bb=(b_1,b_2,b_3)$.
\end{lemma}

\begin{proof}
	Consider the integer lattice $\Z^3$ and the box of $(x_1,x_2,x_3)\in \R^3$ 
	with coordinates
	\[
	|x_1|\le Q, \qquad |x_2|\le Q,\quad  |\xi_1 x_1+\xi_2 x_2 + x_3|\leq cQ^{-2}.
	\]
	It has volume $8c$, independent of $Q$ and $\xi_1, \xi_2$. Hence, by Minkowski's Second Convex
	Body Theorem, the product of the induced successive minima 
	is $\ll c$, hence choosing $c$ small enough the third successive minimum is smaller than $1$. This means there
	cannot be three linearly independent integer points within the box, which in turn is equivalent to the claim.
\end{proof}

We first notice that both $\vv_j$ and $\w_j$ induce approximations of order greater than two. By \eqref{eq:lg2}, \eqref{eq:ti}
and as our choice $(\tau \mu-1)/\mu=n>2$, 
we get
\begin{equation} \label{eq:eins}
|\vv_j\cdot \uxs| = (1+o(1))r^{-1} A_j^{-(\tau \mu-1)} \ll r^{-1} B_j^{-n}\ll
r^{-1}\Vert \hat\w_j\Vert^{-n}<  r^{-1}\Vert \hat\w_j\Vert^{-2}
\end{equation}
and for $\w_j$ by \eqref{eq:tz} we have a stronger estimate that also yields
\begin{equation}  \label{eq:zwei}
|\w_j\cdot \uxs| = o(1)\cdot B_j^{-(\tau \mu-1)} = o(\Vert \hat\w_j\Vert^{-n})=  o(\Vert \hat\w_j\Vert^{-2}), \quad j\to\infty.
\end{equation}
Combined with \eqref{eq:either}, \eqref{eq:T}, \eqref{eq:see}, we have 
\[
\max\{ \Vert \hat\bb\Vert, \Vert \hat\vv_j\Vert, \Vert \hat\w_j\Vert \}= \Vert \hat\w_j\Vert, \qquad \max\{|\bb\cdot \uxs|, |\vv_j\cdot \uxs|,|\w_j\cdot \uxs| \} \ll r^{-1}\Vert \hat\w_j\Vert^{-2}.
\]
By the assumptions of Case 2, the three vectors
$\vv_j, \w_j, \bb$ are linearly independent.
So we get a contradiction
to Lemma~\ref{lemur} for $r$ sufficiently large and $Q=\Vert \hat\w_j\Vert$, as soon as $\Vert \hat\w_j\Vert$ is sufficiently large. Hence, in total, Case 2 provides only finitely many best approximations, of small norm, so Lemma~\ref{wichtig!} is proved.
\end{proof}

For studying $\psi_{\ux}^{1\times 2}$,
we may restrict to integer vectors of the form appearing in Lemma~\ref{wichtig!}.
Now $|\w_j \cdot \uxs|$ is much larger than $|\vv_{j+1} \cdot \uxs|$,
and $|\vv_j \cdot \uxs|$ is much larger than $|\w_{j} \cdot \uxs|$ by \eqref{eq:accto}, \eqref{eq:ti}, \eqref{eq:tz}. Hence
by the triangle inequality $\vv_j\pm \w_j$ induce the same approximation quality
as $\vv_j$ up to a factor $1+o(1)$, and similarly, $\w_j\pm \vv_{j+1}$ induce
the same approximation quality up to $1+o(1)$ as $\w_j$, i.e
\begin{equation} \label{eq:41}
|(\vv_j\pm \w_j)\cdot \uxs| = (1+o(1))\cdot |\vv_j \cdot \uxs|, \quad 
|(\vv_{j+1}\pm \w_j)\cdot \uxs| = (1+o(1))\cdot |\w_j \cdot \uxs|.
\end{equation}
Moreover as the norm of $\hat\w_j$
is much larger than the norm of $\hat\vv_j$ and the norm of $\hat\vv_{j+1}$
is much larger than the norm of $\hat\w_j$, by \eqref{eq:accto} and \eqref{eq:gl},
 we have by the triangle inequality that
\begin{equation} \label{eq:51}
\Vert \hat\vv_j\pm \hat\w_j\Vert = (1+o(1)) \Vert \hat\w_j\Vert, \qquad
\Vert \hat\vv_{j+1}\pm \hat\w_j\Vert = (1+o(1)) \Vert \hat\vv_{j+1}\Vert.
\end{equation}
By \eqref{eq:41}, \eqref{eq:51} and
since we have freedom up to a factor $1+o(1)$ in our claim to be proved, 
we may assume that all
best approximations are among $\vv_j, \w_j$.
Now we choose $r$ appropriately. To finish the proof we show that
choosing $\tau, \mu, r$ suitably we have
\begin{equation} \label{eq:frosch}
| \vv_j \cdot \uxs| = (1+o(1))\cdot c\Vert \hat\w_j\Vert^{-n}, \qquad
| \w_j \cdot \uxs| = o(1)\cdot \Vert \hat\vv_{j+1}\Vert^{-n}.
\end{equation}
Assume this is true.
Then indeed partitioning the interval $[\Vert \vv_1\Vert,\infty)$ into consecutive disjoint intervals
of the form $[\Vert \vv_j\Vert, \Vert \w_j\Vert)$ and $[\Vert \w_j\Vert, \Vert \vv_{j+1}\Vert)$ we see
\begin{equation*}
t^n \psi_{\ux}^{1\times 2}(\Vert\cdot\Vert, t)\le c(1+o(1)) \qquad t\to\infty.
\end{equation*}
On the other hand, since there is no better approximation than the one induced by $\vv_j$ for any
$t<\Vert \w_j\Vert$,
for $t_j:=\Vert \w_j\Vert-1$ we have the reverse inequality 
\begin{equation*}
t_j^n \psi_{\ux}^{1\times 2}(\Vert\cdot\Vert,t_j)\ge c(1+o(1)) \qquad j\to\infty.
\end{equation*}
Combining proves the equality
\[
\limsup_{t\to\infty} t^n \psi_{\ux}^{1\times 2}(\Vert\cdot\Vert, t) = c.
\]
For the left identity of \eqref{eq:frosch}
equivalently by \eqref{eq:lg2}, \eqref{eq:ti} we need
\[
(1+o(1))r^{-1}A_j^{-(\tau \mu-1) }= c\cdot d_2^{-n} A_j^{-n\mu},
\]
so we pick
\[
r=  c^{-1} d_2^n A_j^{n\mu-\tau \mu+1} \overset{ \eqref{eq:IIi}}{=}c^{-1} d_2^n.
\]
Similarly, by \eqref{eq:gl}, \eqref{eq:tz} for the right identity of \eqref{eq:frosch} it suffices to have
\[
B_j^{-(\tau \mu-1)}= o(d_1^{-n} A_{j+1}^{-n\tau}).
\]
Since $A_{j+1}^{-n\tau} \asymp
B_{j}^{-n\tau^2}$ by \eqref{eq:accto} and $d_1$ is constant,
it suffices to notice that upon \eqref{eq:IIi} we may choose $\tau<\mu$ so that 
$n\tau^2 < \tau \mu-1$. Indeed it suffices to choose $\tau$ just slightly larger than $n$ so that $\mu=1/(\tau-n)>\tau^2$ which is equivalent to $n\tau^2 < \tau \mu-1$.  
We leave the short calculation to the reader.
We have completed step I.
Assuming $c$ is small enough, 
we apply Step II below.

 \subsection{Proof of Step II}  \label{st2}
The proof of step II works analogously to the special case of maximum norm
in Section \ref{Sect6}.
Our norm $\Vert\cdot\Vert$ on $\R^n$ induces a projected norm $\Vert\cdot\Vert^{\prime}$ on $\R^2$ via
\begin{equation}  \label{eq:pno}
\Vert (x_1, x_2) \Vert^{\prime}:= \Vert (x_1,x_2,0,\ldots,0) \Vert.
\end{equation}
Then for this norm $\Vert\cdot\Vert^{\prime}$ we apply step I to get $\ux=(\xi_1, \xi_2)$ with the property
\[
\limsup_{t\to\infty} t^n \psi_{\ux}^{1\times 2}(\Vert\cdot\Vert^{\prime}, t) = c.
\]
Let $\bb_v=(b_{1,v}, b_{2,v}, b_{3,v})$ be the integer
vector sequence constructed in Step I.

By construction of $\Vert\cdot\Vert^{\prime}$,
it is again clear by the same argument as in Proposition~\ref{pro}
that any vector $\uz=(\xi_1,\ldots,\xi_n)\in \R^n$
extending $\ux$ has Dirichlet constant at most $c$ with respect to $\Vert\cdot\Vert$,
i.e.
\[
\limsup_{t\to\infty} t^n \psi_{\uz}^{1\times n}(\Vert\cdot\Vert, t) \le
\limsup_{t\to\infty} t^n \psi_{\ux}^{1\times 2}(\Vert\cdot\Vert^{\prime}, t) = c.
\]
Indeed, we can choose the embedded integer vectors $\iota(\bb_v)$
via the embedding
\[
\iota: (y_1, y_2,y_3)\to (y_1,y_2,0,\ldots,0,y_3)
\] 
from $\R^3$ to $\R^{n+1}$, which
by \eqref{eq:pno} induce the same heights and approximation quality
\[
\Vert \widehat{\iota(\bb_v)}\Vert=\Vert \hat\bb_v\Vert^{\prime}, \qquad |\iota(\bb_v)\cdot \uz^{\ast}|=|\bb_v\cdot \uxs |
\]
for any $\bb_v$, with hat notation
according to \eqref{eq:hatnot} on the respective spaces.
For the non-trivial lower estimate 
\[
\limsup_{t\to\infty} t^n \psi_{\uz}^{1\times n}(\Vert\cdot\Vert, t)\ge c
\]
with respect to $\Vert\cdot\Vert$, 
define
\[
L_v=|\iota(\bb_v)\cdot \uz^{\ast}|=|\bb_v\cdot \uxs |=|b_{1,v} \xi_1 + b_{2,v} \xi_2 + b_{3,v}|, \quad M_v=M_v(\Vert\cdot\Vert)=\Vert \hat\bb_v\Vert^{\prime}=\Vert \widehat{\iota(\bb_v)}\Vert
\]
and $T_v=M_v-1/2$ as in the proof for the maximum norm 
but where $M_v, T_v$ are now defined with respect to our norm $\Vert\cdot\Vert$.
 Then by the equivalence of norms on $\R^n$, for any integer vector of norm $\Vert (x_1,\ldots,x_n)\Vert\le T_v$
we still have
\[
\max_{1\le i\le n} |x_i|\le C(\Vert\cdot\Vert) \cdot T_v
\]
for some constant $C$ depending on the norm only. We replace \eqref{eq:ob} by this modified inequality and follow the proof of the special case Lemma~\ref{lema} where the maximum norm was treated. The twist by the constant $C$ will result in some altered value of $c_n$, depending on the chosen norm, for the same conclusion. 
We omit its explicit calculation in dependence of~$C$.

\section{Proof of Theorem~\ref{thm5}} \label{Sect8}
 We point out again that our proof below can be 
 significantly shortened
 using the preprint~\cite[Corollary~1]{AginWeiss2024}
which appeared later. We detail the parts that can be simplified in the proof below. In short, when taking into account~\cite{AginWeiss2024},
the main new substance still needed is Lemma~\ref{lemme} below. 

\subsection{Proof excluding linear independence claims} \label{sec61}

We proceed as in~\S~\ref{st1}  to construct the first two elements of the first line of our matrix $\Omega$ by $\ux=(\xi_1,\xi_2)=(\Omega_{1,1}, \Omega_{1,2})$, upon 
modifying the construction so that \eqref{eq:IIi} is replaced by
\[
\tau \mu -1 = \frac{n}{m} \mu \overset{\eqref{matrix dimension condition}}{\ge}  2\mu.
\]
Define $\psi_{\ux}^{1\times 2}(\Vert\cdot\Vert_1^{\prime}, t)$ with respect to the restriction $\Vert\cdot\Vert_1^{\prime}$
of $\Vert\cdot\Vert_1$ to $\R^2$ in the integer vector $\bb$, as in \eqref{eq:pno}.
Then by the same arguments of the proof of Theorem~\ref{jia} we will have
\begin{equation} \label{eq:fichte}
\limsup_{t\to\infty} t^{n/m} \psi_{\ux}^{1\times 2}(\Vert\cdot\Vert_1^{\prime}, t)= c.
\end{equation}
Hereby we use our assumption $n/m\ge 2$ for Case 2, based on Minkowski's Second Convex Body Theorem, to work analogously.

Now take the matrix $V\in \R^{m\times 2}$ with all $m$ lines equal to $\ux$, i.e.
\[
V_{j,1}= \xi_{1}, \qquad V_{j,2}=\xi_{2}, \qquad 1\leq j\le m.
\]
This matrix gives rise to an approximation function $\psi_{V}^{m\times 2}(\Vert\cdot\Vert_1^{\prime}, \Vert\cdot\Vert_2, t)$ with norms defined as above. First, assume for simplicity that 
$\Vert\cdot\Vert_2=\Vert\cdot\Vert_{\infty}$ is the maximum norm on $\R^m$.
We claim that
\begin{equation}  \label{eq:teir}
\psi_{V}^{m\times 2}(\Vert\cdot\Vert_1^{\prime}, \Vert\cdot\Vert_2, t)=\psi_{\ux}^{1\times 2}(\Vert\cdot\Vert_1^{\prime}, t), \qquad t>0.
\end{equation}
Indeed to see
\begin{equation} \label{eq:einer}
\psi_{V}^{m\times 2}(\Vert\cdot\Vert_1^{\prime}, \Vert\cdot\Vert_2, t)\ge \psi_{\ux}^{1\times 2}(\Vert\cdot\Vert_1^{\prime}, t)
\end{equation}
for all $t$, note that
for arbitrary $\bb=(b_1,\ldots,b_{m+2})\in\Z^{m+2}$ and for $V^{\ast}\in \R^{m\times (m+2)}$ obtained by glueing $V$ to the $m\times m$ identity matrix to the right, the induced value
\[
\Vert V^{\ast}\cdot \bb\Vert_2=
\Vert V\cdot \widehat{\bb}+ \tilde{\bb}\Vert_2=
\Vert V\cdot (b_1,b_{2}) + (b_{3},\ldots,b_{m+2})\Vert_2
\]
with vectors interpreted as 
column vectors, is at least
as large as the modulus of any entry, since we are considering $\|\cdot\|_{2}=\|\cdot\|_{\infty}$ for now. 
But for the first entry this equals $|(b_1,b_2)\cdot \ux +b_3|=| \uxs\cdot (b_1,b_2,b_3)|$. Taking the minimum over $\bb$  as in the definition of $\psi_V^{m\times 2}$ shows the inequality \eqref{eq:einer}.
Conversely, if we extend arbitrary $\bb=(b_1, b_2, b_3)\in\Z^3$ into $\Z^{m+2}$ via
$\eta(\bb):=(b_1,b_2,b_3,\ldots,b_{3})$ with last $m$ coordinates 
equal to $b_3$, 
by choice of maximum norm, we have
$\Vert V^{\ast}\cdot \eta(\bb)\Vert_2 = |\uxs \cdot (b_1,b_2,b_3)|$, so passing to the minimum
over $\bb$ again we get the reverse inequality
\begin{equation}  \label{eq:zweier}
\psi_{V}^{m\times 2}(\Vert\cdot\Vert_1^{\prime}, \Vert\cdot\Vert_2, t)\le \psi_{\ux}^{1\times 2}(\Vert\cdot\Vert_1^{\prime}, t).
\end{equation}
Combining \eqref{eq:fichte}, \eqref{eq:teir} we get
\[
\limsup_{t\to\infty} t^{n/m} \psi_{V}^{m\times2}(\Vert\cdot\Vert_1^{\prime}, \Vert\cdot\Vert_2, t)=\limsup_{t\to\infty} t^{n/m} \psi_{\ux}^{1\times 2}(\Vert\cdot\Vert_1^{\prime}, t)= c.
\]
We note that
up to this point, the claim alternatively follows directly from the aforementioned preprint~\cite[Corollary~1]{AginWeiss2024}. \par 
To finish the proof, we show that for a positive measure set of $B\in \R^{m\times (n-2)}$ inducing $\Omega=(V,B)$ via putting it to the right of $V$,
we have
\[
\limsup_{t\to\infty} t^{n/m} \psi_{\Omega}^{m\times n}(\Vert\cdot\Vert_1, \Vert\cdot\Vert_2, t)=\limsup_{t\to\infty} t^{n/m} \psi_{V}^{m\times 2}(\Vert\cdot\Vert_1^{\prime}, \Vert\cdot\Vert_2, t).
\] 
Again the inequality
\begin{equation} \label{eq:frit}
\limsup_{t\to\infty} t^{n/m} \psi_{\Omega}^{m\times n}(\Vert\cdot\Vert_1, \Vert\cdot\Vert_2, t)\le \limsup_{t\to\infty} t^{n/m} \psi_{V}^{m\times 2}(\Vert\cdot\Vert_1^{\prime}, \Vert\cdot\Vert_2, t)
\end{equation}
is easy to check for arbitrary $\Vert\cdot\Vert_2$ and extensions $\Omega\in\R^{m\times n}$
of $V$ by an arbitrary matrix $B\in \R^{m\times (n-2)}$. Indeed, for a sufficient estimate
\[
\psi_{\Omega}^{m\times n}(\Vert\cdot\Vert_1, \Vert\cdot\Vert_2, t)\le \psi_{V}^{m\times 2}(\Vert\cdot\Vert_1^{\prime}, \Vert\cdot\Vert_2, t), \qquad t>0
\]
we embed the integer best approximations of $V$ 
by adding $0$ coordinates via
\[
\sigma:\;\bb=(b_1, b_2, \ldots, b_{m+2})\to (b_1,b_2,0,\ldots,0,b_3,\ldots,b_{m+2})\in\Z^{m+n}.
\]
These induce the same
norm $\Vert \widehat{\sigma(\bb)} \Vert_1=\Vert \widehat{\bb} \Vert_1^{\prime}$ 
where
\[
\widehat{\sigma(\bb)}= (b_1,b_2,\ldots,b_n)=(b_1,b_2,0,\ldots,0), \qquad \widehat{\bb}= (b_1,b_2),
\]
with hat again as in \eqref{eq:hatnot}
on the respective spaces. With $\Omega^{\ast}$ derived from $\Omega$ likewise as for $V^{\ast}$ (i.e. add the identity matrix to the right hand side of $\Omega$),
they further induce the same approximation qualities $\Vert \Omega^{\ast}\cdot \sigma(\bb) \Vert_2=\Vert V^{\ast}\cdot \bb \Vert_2$ since in fact 
$\Omega^{\ast}\cdot \sigma(\bb) =V^{\ast}\cdot \bb$, proving \eqref{eq:frit}.
For the reverse estimate
\begin{equation}  \label{eq:trif}
\limsup_{t\to\infty} t^{n/m} \psi_{\Omega}^{m\times n}(\Vert\cdot\Vert_1, \Vert\cdot\Vert_2, t)\ge \limsup_{t\to\infty} t^{n/m} \psi_{V}^{m\times 2}(\Vert\cdot\Vert_1^{\prime}, \Vert\cdot\Vert_2, t)=c,
\end{equation}
we use a generalisation of Step II from~\S~\ref{st2} for $m>1$ 
as in Moshchevitin~\cite[Theorem~12]{ngm}. We obtain that for small enough $c$ and a positive measure set of remaining matrices $B\in \R^{m\times (n-2)}$ with respect to $m(n-2)$ dimensional Lebesgue measure
when added to the right of $V$ to obtain $\Omega$, we have \eqref{eq:trif}.
If $m=1$ then 
we require the condition $n\ne m+2=3$.
Moreover, the assumption $n\ge 2m$ is vital for the volume estimates.
See Lemma~\ref{lemme} below for a generalisation. The proof for $\Vert\cdot\Vert_2$ the maximum 
norm is complete.

Now let the norm $\Vert\cdot\Vert_2$ be arbitrary. 
Again we prove \eqref{eq:einer}, \eqref{eq:zweier}, \eqref{eq:frit}, \eqref{eq:trif}. The critical estimates where the norm $\Vert\cdot\Vert_2$ is relevant for the argument are \eqref{eq:einer}, \eqref{eq:zweier}.
First notice that by the special form of $V$, if $\bb$ is not of the form 
$\bb=(b_1,b_2,b_3,b_3,\ldots,b_3)\in \Z^{m+2}$ 
then some entry of $V^{\ast}\cdot \bb$ will be of modulus at least $1/2$, 
so $\Vert V^{\ast}\cdot \bb\Vert_{\infty}\ge 1/2$,
hence by equivalence of norms, we deduce 
\[
\Vert V^{\ast}\cdot \bb\Vert_2 \gg 1
\]
with some absolute constant depending on the norm $\Vert\cdot\Vert_2$ only.
Hence we only need to take into account vectors of the form $\bb=(b_1,b_2,b_3,b_3,\ldots,b_3)$ when studying $\psi_{V}^{m\times 2}$. But then the 
vector $V^{\ast}\cdot \bb\in \R^m$ lies on the line $\lambda\cdot (1,1,\ldots,1)\in \R^m$ and has norm $|\lambda|\cdot \Vert (1,1,\ldots,1)\Vert_2$ proportional to $|\lambda|$. Hence
we have to minimize $|\lambda|$. In~\S~\ref{st1}, see in particular Lemma~\ref{wichtig!}, we showed that this is achieved when the subvector $(b_1,b_2,b_3)$ is essentially 
of the form $\vv_j=(A_j,0,-F_j)$ and $\w_j=(0,B_j,-G_j)$ for $A_j, B_j, F_j, G_j$ defined in the same section (as shown in~\S~\ref{st1}, for the other possible vectors of Lemma~\ref{wichtig!} we may only improve the approximation quality by a negligible factor $1+o(1)$
as $j\to\infty$). These arguments show that the best integer approximations for $V$ are essentially of the form
\begin{equation} \label{eq:vecin}
\boldsymbol{c}_j:= (A_j, 0 , -F_j, \ldots , -F_j ), \qquad \boldsymbol{d}_j:= (0, B_j , -G_j, \ldots , -G_j )
\end{equation}
so that for $\vv_j, \w_j\in\Z^3$ as before, we get error vectors
\[
V^{\ast}\cdot \boldsymbol{a}_j=  V\cdot \begin{pmatrix}
A_j \\ 0
\end{pmatrix}
- F_j\cdot \begin{pmatrix}
1 \\ 1 \\ \vdots \\ 1
\end{pmatrix}=\vv_j\cdot \uxs\cdot \begin{pmatrix}
1 \\ 1 \\ \vdots \\ 1
\end{pmatrix}\in \R^m
\]
and
\[ V^{\ast}\cdot \boldsymbol{b}_j= V\cdot \begin{pmatrix}
0 \\ B_j
\end{pmatrix}- G_j\cdot \begin{pmatrix}
1 \\ 1 \\ \vdots \\ 1
\end{pmatrix}=\w_j\cdot \uxs\cdot \begin{pmatrix}
1 \\ 1 \\ \vdots \\ 1
\end{pmatrix}\in \R^m.
\]
Then for $\Gamma:=\Vert(1,1,\ldots,1)\Vert_2$ they have norms
\[
\left\Vert V^{\ast}\cdot \boldsymbol{c}_j\right\Vert_2= 
\vert \vv_j\cdot \uxs|\cdot \Gamma, \qquad \left\Vert V^{\ast}\cdot \boldsymbol{d}_j\right\Vert_2= 
\vert \w_j\cdot \uxs|\cdot \Gamma.
\]
Thus we have to adjust the proof of the special case by the constant 
factor $\Gamma$ in the calculation, so that
in order to finalize Step I we now need to pick
\[
r=c^{-1}\Gamma^{-1} d_2^n,\qquad d_2= \Vert \ee_2\Vert_1^{\prime}=\Vert (0,1)\Vert_1^{\prime}
\]
to get \eqref{eq:zweier}. The reverse inequality \eqref{eq:einer} holds as well as we noticed
that the vectors in \eqref{eq:vecin} are (essentially) 
the integer best approximations.
We observed that \eqref{eq:frit} holds independent of $\Vert\cdot\Vert_2$. 
Note that again
this step follows immediately  via a different proof from the preprint~\cite[Corollary~1]{AginWeiss2024}. 

\par 
 Finally, Step II is to again show the reverse inequalities \eqref{eq:trif} hold. This is done analogously to Lemma~\ref{lema} but in
higher dimension, building up on ideas of Moshchevitin~\cite[Theorem~12]{ngm}.
Again the claim is essentially irrespective of the norm $\Vert\cdot\Vert_2$. We state the according analogue of Lemma~\ref{lema} that we will prove analogously but without effort to make constants effective. 

\begin{lemma} \label{lemme}
	Let $m,n$ be positive integers with $n\ge 2$, and $(m,n)\ne (1,3)$. Let $\Vert\cdot\Vert_1$ and $\Vert\cdot\Vert_2$
	be any norms on $\R^n$ and $\R^m$ and $\Vert\cdot\Vert_1^{\prime}$ be the restriction of $\Vert\cdot\Vert_1$ to $\R^2$ as in \eqref{eq:pno}.
	Further, let  
	$V\in \R^{m\times 2}$ be any matrix with
	\begin{equation} \label{V property}
	\limsup_{t\to\infty} t^{n/m} 
    \psi_{V}^{m\times 2}(\Vert\cdot\Vert_1^{\prime}, \Vert\cdot\Vert_2, t)= c
	\end{equation}
	for small enough $c\in [0,c_0(m,n,\Vert.\Vert_1,\Vert.\Vert_2)]$, with $c_0$ depending on $m,n$ and the norms. Let $\bb_v\in\Z^{m+2}$ be the associated sequence
	of best approximations,  and $\hat\bb_v= (b_{v,1}, b_{v,2})$ of norms and approximation qualities 
	\[
	M_v:=\Vert \hat\bb_v \Vert_1^{\prime}, \qquad L_v:= \Vert V^{\ast}\cdot \bb_v\Vert_2.
	\]
	Then, for a positive measure set $\mathcal{Y}\subseteq B_{m(n-2)}(\boldsymbol{0},1)\subseteq \R^{m(n-2)}$ consisting
	of matrices $B\in \R^{m\times (n-2)}$, we have that for an infinite subsequence of
	integers $(v_r)_{r\in\N}$, letting $T_{v}:=M_{v+1}-1/2$ for $v\ge 1$, we have
    \begin{equation} \label{eq:ignore}
    \lim_{r\to\infty} T_{v_r}^{n/m} \psi_V^{m\times 2}(T_{v_r}) 
    = \limsup_{v\to\infty} T_{v}^{n/m}\cdot \psi_{V}^{m\times 2}(T_{v})
    =c
    \end{equation}
    i.e. realising the limsup, 
	and the integer best approximation  
	for $\Omega=(V,B)\in \R^{m\times n}$ at $T_{v_r}$ is for $r\ge 1$ given as
	\[
	b_{1,v_r}\ee_1+  b_{2,v_r}\ee_2+   b_{3,v_r}\ee_{n+1}+\cdots+b_{m+2,v_r}\ee_{n+m}=
	(b_{1,v_r},b_{2,v_r},0,\ldots,0,b_{3,v_r},\ldots,b_{m+2,v_r})\in \Z^{n+m}. 
	\]
	Thus, the best approximation function at values $T_{v_r}$ equals
	\begin{equation} \label{eq:forti}
	\psi_{\Omega}^{m\times n}(\Vert\cdot\Vert_1, \Vert\cdot\Vert_2,T_{v_r})=\psi_{V}^{m\times 2}(\Vert\cdot\Vert_1^{\prime}, \Vert\cdot\Vert_2, T_{v_r})= L_{v_r},\qquad r\ge 1.
	\end{equation}
\end{lemma}

\begin{proof}
   First assume $n\ne m+2$, this condition will be removed later in the case of $m>1$. 
    Fix $V \in \R^{m\times 2}$. 
    For the same reason as in the proof of Lemma~\ref{lema}, we may assume
    \[
    \lim_{v\to\infty} T_v^{n/m}\psi_{V}^{m\times 2}(T_v)=c
    \]
    and ignore \eqref{eq:ignore}.
    As in the proof of Lemma~\ref{lema}, we construct the set $\mathcal{Y}$ to be composed of matrices $B\in \mathcal{Y}\subseteq \R^{m\times(n-2)}$ such that the matrices $\Omega=(V,B)\in \R^{m\times n}$ satisfy
    \[
	\psi_{\Omega}^{m\times n}(\Vert\cdot\Vert_1, \Vert\cdot\Vert_2,T_{v})=\psi_{V}^{m\times 2}(\Vert\cdot\Vert_1^{\prime}, \Vert\cdot\Vert_2, T_{v})= L_v.
	\] 
 By the equivalence of norms we have that there exists $C(\|\cdot\|_{1})>0$ dependent only on $\|\cdot\|_{1}$ such that for any $\hat{\xx}\in\R^{n}$
 \begin{equation*}
     \|\hat{\xx}\|_{1}\leq T_{v} \quad \implies \quad \|\hat{\xx}\|_{\infty} \leq C(\|\cdot\|_{1})T_{v}.
 \end{equation*}
 Thus if again $\Omega^{\ast}=(V,B,I)\in \R^{m\times (m+n)}$ is obtained by putting the $m\times m$ identity matrix, I, to the right of $\Omega$ and we let
 \begin{equation*}
     S_{v}:=\left\{ B \in B_{m(n-2)}(\boldsymbol{0},1): \underset{0<\|\hat{\xx}\|_{\infty}\leq C(\|\cdot\|_{1})T_{v}}{\min_{\xx \in \Z^{n+m}}} \left\|(V,B,I)\cdot \xx \right\|_{2} \geq L_{v} \right\}\, ,  
 \end{equation*}
then $\mathcal{Y} \supseteq \limsup_{v\to \infty} S_{v}$. Hence if we can show that the set 
 \begin{equation} \label{Y condition}
    \lambda_{m(n-2)}\left( \liminf_{v\to \infty} S_{v}^{c} \right) \leq \delta < \lambda_{m(n-2)}\left(B_{m(n-2)}(\boldsymbol{0},1)\right)
 \end{equation}
 then clearly $\lambda_{m(n-2)}(\mathcal{Y})>0$. Furthermore, if we show that
 \begin{equation} \label{Svc size}
     \lambda_{m(n-2)}(S_{v}^{c})\leq \delta \qquad v\geq 1,
 \end{equation}
 then applying Proposition~\ref{pp} we have \eqref{Y condition}, so proving \eqref{Svc size} is sufficient to verify $\lambda_{m(n-2)}(\mathcal{Y})>0$. \par 
 Let $z=(z_{3},\dots,z_{n})$, $y=(x_{n+1},\dots, x_{n+m})$ and
 \begin{equation*}
     X_{v}(z,x_{1},x_{2},y)=\left\{ \ug \in B_{m(n-2)}(\boldsymbol{0},1) : \left( \begin{array}{c}
     z_{3}\gamma_{1,1} + \dots + z_{n}\gamma_{1,n-2} \\
     \vdots \quad \quad \vdots \\
     z_{3}\gamma_{m,1}+\dots + z_{n}\gamma_{m,n-2} \end{array} \right) \in J_{v}(x_{1},x_{2},y) \right\}
 \end{equation*}
 where
 \begin{equation*}
     J_{v}(x_{1},x_{2},y):=\prod_{i=1}^{m} (-x_{1}\xi_{1}-x_{2}\xi_{2} -x_{n+i}-C(\|\cdot\|_{2})L_{v}, -x_{1}\xi_{1}-x_{2}\xi_{2} -x_{n+i}+C(\|\cdot\|_{2})L_{v}).
 \end{equation*}
 Here we have again used equivalence of norms, namely that there exists $C(\|\cdot\|_{2})>0$ dependent only on $\|\cdot\|_{2}$ so that for any $X\in\R^{m}$  
 \begin{equation*}
     C(\|\cdot\|_{2})^{-1}\|X\|_{\infty}\leq \|X\|_{2} \leq C(\|\cdot\|_{2})\|X\|_{\infty}.
 \end{equation*}
 Considering one row of $X_{v}(z,x_{1},x_{2},y)$ at a time (denoted $X_{v}^{(i)}(z,x_{1},x_{2},y)$ for $1\leq i \leq m$) we have the same calculation as in \eqref{volcume calculation}, that is
 \begin{equation*}
     \lambda_{n-2}\left(X_{v}^{(i)}(z,x_{1},x_{2},y)\right) \leq 2C(\|\cdot\|_{2})L_{v}\lambda_{n-3}(B_{n-3}(\boldsymbol{0},1))\frac{1}{\|z\|_{\infty}}.
 \end{equation*}
 Hence
 \begin{align*}
     \lambda_{m(n-2)}\left(X_{v}(z,x_{1},x_{2},y)\right) &\leq 2^{m}C(\|\cdot\|_{2})^{m}L_{v}^{m}\lambda_{n-3}(B_{n-3}(\boldsymbol{0},1))^{m}\frac{1}{\|z\|_{\infty}^{m}}\\
     &=C(m,n,\|\cdot\|_{1},\|\cdot\|_{2})L_{v}^{m} \frac{1}{\|z\|_{\infty}^{m}}.
 \end{align*}
 Using similar calculations as appearing in Lemma~\ref{lema} we have
 \begin{align*}
     \lambda_{m(n-2)}\left(S_{v}^{c}\right)& \leq \sum_{0< \max_{i=1,2}|x_{i}|\leq C(\|\cdot\|_{1})T_{v}}\sum_{\|y\|_{\infty}\leq C(\|\cdot\|_{1})T_{v}}\sum_{\|z\|_{\infty}\leq C(\|\cdot\|_{1})T_{v}} \lambda_{m(n-2)}\left(X_{v}(z,x_{1},x_{2},y)\right)\\
     &\leq C(m,n,\|\cdot\|_{1},\|\cdot\|_{2})L_{v}^{m}T_{v}^{m+2}\sum_{\|z\|_{\infty}\leq C(\|\cdot\|_{1})T_{v}}\frac{1}{\|z\|_{\infty}^{m}} \\
     & \leq C'(m,n,\|\cdot\|_{1},\|\cdot\|_{2})L_{v}^{m}T_{v}^{m+2}\sum_{k=1}^{ \lfloor C(\|\cdot\|_{1})T_{v}\rfloor}k^{n-3-m}\\
     &\overset{(n\ne m+2)}{\leq} C''(m,n,\|\cdot\|_{1},\|\cdot\|_{2})L_{v}^{m}T_{v}^{m+2}T_{v}^{1+n-3-m}\\
     &\leq C''(m,n,\|\cdot\|_{1},\|\cdot\|_{2})L_{v}^{m}T_{v}^{n}.
 \end{align*}
 Thus, for any $0<\delta<\lambda_{m(n-2)}\left(B_{m(n-2)}(\boldsymbol{0},1)\right)$, if $L_{v}<cT_{v}^{-\frac{n}{m}}$ then for any  $c\in(0,C'''(m,n,\|\cdot\|_{1}))$ with
 \begin{equation*}
     C'''(m,n,\|\cdot\|_{1},\|\cdot\|_{2}) < \left(\frac{\delta}{C''(m,n,\|\cdot\|_{1},\|\cdot\|_{2})}\right)^{\frac{1}{m}},
 \end{equation*}
 we have that $\lambda_{m(n-2)}\left(S_{v}^{c}\right)<\delta$ as required.

Finally we explain how to avoid the condition $n\ne m+2$ when $m\ge 2$. Instead of extending $V\in \R^{m\times 2}$ directly to the desired $m\times n$ matrix $\Omega$, we iteratively apply Lemma~\ref{lemme}, adding a single column to $V$ at each step. Precisely, suppose $m\geq 2$ and $n=m+2$. Begin with some $V=\tilde{V}_{1}\in\R^{m\times 2}$ satisfying \eqref{V property}. Apply Lemma~\ref{lemme} to construct a set of $m$-dimensional positive measure matrices $\Omega_{1}\in \R^{m\times 3}$ with property \eqref{eq:forti}. Define $\tilde{V}_{2}:=\Omega_{1}$ to be any of the matrices in the the constructed set. Note $\tilde{V}_{2}$ satisfies
\begin{equation*}
\psi_{\tilde{V}_{2}}^{m\times 3}(\|\cdot\|^{''}_{1},\|\cdot\|_{2},T_{v_{r}})=\psi_{\tilde{V}_{1}}^{m\times 2}(\|\cdot\|^{'}_{1},\|\cdot\|_{2},T_{v_{r}})\, ,
\end{equation*}
with $\|\cdot\|^{''}_{1}$ now being the restriction of $\|\cdot\|_{1}$ to $\R^{3}$. So
\begin{equation*}
    \limsup_{t\to\infty} t^{n/m}\psi_{\tilde{V}_{2}}^{m\times 3}(\|\cdot\|^{''}_{1},\|\cdot\|_{2},t)=c\, ,
\end{equation*}
Note $\tilde{V}_{2}\in\R^{m\times 3}$ rather than $\R^{m\times 2}$, but it is readily verified that we can copy the above proof with obvious modifications. So, we construct a set of $m$-dimensional positive measure matrices $\Omega_{2}\in \R^{m\times 4}$ with the desired properties of Lemma~\ref{lemme}. Define $\tilde{V}_{3}:=\Omega_{2}$ for any of the matrices in the above set and repeat. The same argument follows as above. In this way we obtain a sequence of matrices  
\begin{equation*}
    \Omega_{1}=: \tilde{V}_{2} \in \R^{m\times 3}, \Omega_{2}=:\tilde{V}_{3}\in \R^{m\times 4},\ldots, \Omega_{m-1}=:\tilde{V}_{m+1}\in \R^{m\times(m+1)},\Omega_{m}=:\Omega\in \R^{m\times (m+2)}\, .
\end{equation*}    
    For each $2< \tilde{n} \leq m+2$ the extension from $\tilde{V}_{\tilde{n}-2} \in \R^{m\times (\tilde{n}-1)}$ to an $m\times \tilde{n}$ matrix follows readily by copying the above proof of Lemma~\ref{lemme} (with the restriction $n\neq m+2$) with obvious modifications. Note we avoid the condition $\tilde{n}= m+(\tilde{n}-1)$ since $m\geq 2$. Thus, at each step we construct a set of matrices $\Omega_{\tilde{n}-1}\in \R^{m\times \tilde{n}}$ with positive $m$-dimensional measure, each with the desired properties of \eqref{eq:forti}. It is readily checked that the required convergence estimates still work a fortiori as $n\ge \tilde{n}$ and $m$ remains constant. Using Marstrand's slicing lemma \cite{mars} for fibered sets, we increase the Hausdorff dimension by $m$ in every step. There are $m$ steps, hence we arrive at the dimension of at least $m^{2}=m(n-2)$.

\end{proof}

There is nothing essential about $V$ having two columns as the last part of the proof confirms. Note that our proof above avoids the condition $n\ne m+2$ as soon as $m>1$, contrasting the
requirement $n\ge 4$ in Lemma~\ref{lema} where $m=1$. We point out that a similar argument works to generalize the original result in~\cite{ngm}, again when $m>1$. 
All claims of Theorem~\ref{thm5} apart from the linear independence statements are proved.

\subsection{Proof of linear independence claims} \label{brand}

 The first claim (i) on linear independence
 of rows follows from Step II. 
 Aside from the metrical claim, note that
 it suffices to prove linearly independent rows can be arranged for the submatrix $B\in \R^{m\times (n-2)}$ of $\Omega=(V,B)$.
 Now since $n-2\ge m$ by \eqref{matrix dimension condition}, the matrix $B$ does not have fewer columns than rows, 
 so if $B$ does not have the property it is rank deficient. 
 However, as indicated in Section~\ref{Sek3.1}, looking at minors yields that the set of
 rank-deficient matrices has positive codimension and thus
 zero Lebesgue measure. 
 On the other
 hand, we showed that a positive Lebesgue measure set of matrices $B$ gives rise
 to $\Omega$ as in the theorem. Hence removing such rank-deficient matrices still leaves us with a set of positive $m\times (n-2)$-dimensional Lebesgue measure from which we can choose arbitrary $B$.

    For the latter claim (ii) on $\Q$ linearly independent entries,
	note that in the construction of
    Section~\ref{sec61}
 certain entries of $V$ (thus $\Omega$) are identical, so the condition does not hold. 
	
	To fulfil this additional property, let us modify our construction
	by varying signs. First construct the submatrix $V\in \R^{m\times 2}$ as above
	by taking for $A_j, B_j$ as before
	\[
	\Omega_{k,1}=V_{k,1}= \sum_{j=1}^{\infty} \delta_{k,j} A_j^{-1}, \qquad \Omega_{k,2}=V_{k,2}= \sum_{j=1}^{\infty} \delta_{k,j}^{\ast} B_j^{-1}, \quad (1\le k\le m)
	\]
	with sequences $\delta_{k,j}\in\{-1,1\}$, $\delta_{k,j}^{\ast}\in\{-1,1\}$ to be chosen later. When $\delta_{k,j}=\delta_{k,j}^{\ast}=1$ for all $k,j$ we are in the standard case of the previous sections.
	
	When $m=1$,
	all arising matrices/vectors $(V_{1,1}, V_{1,2})=(\xi_1, \xi_2)$ share the properties of the standard case above.
	For $m>1$, we get more general error vectors
	than in \eqref{eq:vecin}, now of the form
	\begin{equation}  \label{eq:errt}
	\tilde{\boldsymbol{c}}_j:= V\cdot \begin{pmatrix}
	A_j \\ 0
	\end{pmatrix}
	- \begin{pmatrix}
	F_{j,1} \\ F_{j,2} \\ \vdots \\ F_{j,m}
	\end{pmatrix}, \qquad \tilde{\boldsymbol{d}}_j:= V\cdot \begin{pmatrix}
	0 \\ B_j
	\end{pmatrix}
	- \begin{pmatrix}
	G_{j,1} \\ G_{j,2} \\ \vdots \\ G_{j,m}
	\end{pmatrix},\qquad j\ge 1,
	\end{equation}
	where
	\[
	F_{j,k}= A_j \sum_{u=1}^{j}  \delta_{k,u} A_{u}^{-1}, \qquad
	G_{j,k}= B_j \sum_{u=1}^{j}  \delta_{k,u}^{\ast} B_{u}^{-1}, \qquad (1\le k\le m),
	\]
	depending on the choice of $\delta_{k,j}, \delta_{k,j}^{\ast}$.
 They again satisfy congruence properties $$F_j\equiv \pm 1\bmod 6 \quad \text{ and} \quad G_j\equiv \pm 1\bmod 35$$ for the same reasons, sufficient for coprimality with $6$ respectively $35$.
	Thus the entries of the remainders $\tilde{\boldsymbol{c}}_j, \tilde{\boldsymbol{d}}_j$ are given as
	\[
	c_{j,k}= A_j \sum_{u=j+1}^{\infty}  \delta_{k,u} A_{u}^{-1}, \qquad
	d_{j,k}= B_j \sum_{u=j+1}^{\infty}  \delta_{k,u}^{\ast} B_{u}^{-1}, \qquad (1\le k\le m).
	\]
	However, since $A_j/A_{j+1}\to 0$ and $B_j/B_{j+1}\to 0$ we again see
	\[
	c_{j,k}= \delta_{k,j+1} \cdot A_j A_{j+1}^{-1}(1+o(1)), \qquad d_{j,k}=\delta_{k,j+1}^{\ast} \cdot B_j B_{j+1}^{-1}(1+o(1)), \quad (1\le k\le m).
	\]
 Hence, by the equivalence of norms, 
	it is clear that 
	the error vectors in \eqref{eq:errt} all have $\Vert\cdot\Vert_2$ norm of order respectively 
	\begin{equation}  \label{eq:fertig}
	\Vert \tilde{\boldsymbol{c}}_j\Vert_2= A_j A_{j+1}^{-1} \Gamma_j (1+o(1)), \qquad \Vert \tilde{\boldsymbol{d}}_j\Vert_2=B_j B_{j+1}^{-1} \tilde{\Gamma}_j (1+o(1)),\qquad j\to\infty,
	\end{equation}
 with $\Gamma_j, \tilde{\Gamma}_j$ of one of the numbers
 \[
  \Gamma_j, \tilde{\Gamma}_j \in \{ \Vert (\pm 1,\pm 1,\ldots,\pm 1)\Vert_2\}, \qquad j\ge 1,
 \]
 for any possible sign choice. 
 The values $\tilde{\Gamma}_j$ do not matter much as 
 the according evaluations 
 will be of negligible order by 
 the analogue of \eqref{eq:frosch}. 
 Now if we let 
 \[
 \Gamma:= \max\{ \Vert (\pm 1,\pm 1,\ldots,\pm 1)\Vert_2 \}
 \]
 with maximum overall $2^m$ sign choices,
 we can use the argument of Section~\ref{sec61}. Hereby we may assume that the maximizing sign choice occurs for infinitely many $(\delta_{1,j}, \ldots, \delta_{m,j})$, as otherwise we redefine $\Gamma$ via the maximum among those $m$ term sign sequences occurring infinitely often.
 Moreover, due to \eqref{eq:fertig}, the arguments of
    Section~\ref{st1} applied to each of the $m$ coordinates separately still apply very similarly
    and yield that the vectors $(A_j, 0 , -F_{j,1},\ldots,-F_{j,m})$ and
    $(0 , B_j, -G_{j,1},\ldots,-G_{j,m})$ essentially form the sequence of best approximations for the induced $V$. We leave the details to the reader. Then very similar arguments as in the classical case
    including Lemma~\ref{lemme} show that the Dirichlet constant $\Theta^{m\times n}(\Omega)$ of many such $\Omega=\Omega(\delta,\delta^{\ast})$ equals $c$, as before. 
	
	Finally we show that for many choices of $\delta_{k,j}, \delta_{k,j}^{\ast}$ all entries
	of $\Omega\in\R^{m\times n}$ are linearly independent over $\Q$ together 
	with $\{1\}$. Let us first show this for the submatrix $V\in\R^{m\times 2}$. 
	Label the entries $V_{i,j}$ of $V$ by $v_1,\ldots,v_{2m}$ with
	$v_{2h}=V_{h,2}$ and $v_{2h-1}=V_{h,1}$, $1\le h\le m$.
	Now the first line of $V$ is $\Q$-linearly independent 
	for any choice of $\delta_{1,j}, \delta_{1,j}^{\ast}$ sequences
	as otherwise the sequence
	of best approximations for $(V_{1,1}, V_{1,2})$ 
	would terminate, but the proof above shows that this is not the case. 
	Hence $\{ 1,v_1, v_2\}$ is
	$\Q$-linearly independent. Then we proceed inductively.
	Having constructed the first $\ell\ge 2$ elements $v_1, \ldots,v_{\ell}$, we choose the $\delta$ or $\delta^{\ast}$ sequence for $v_{\ell+1}$ so that we avoid
	$\Q$-linear dependence of $v_{\ell+1}$ with the already chosen 
	$1,v_1, \ldots,v_{\ell}$. This is possible
	since the $\Q$-span of $\{ 1,v_1,\ldots,v_{\ell} \}$ is countable, but we have uncountably many choices of $\delta$ or $\delta^{\ast}$ sequences. These clearly induce uncountably many pairwise distinct numbers $v_{\ell+1}$,
	the injectivity  of the corresponding maps $\varphi: \{-1,1\}^{\mathbb{N}}\to \R$ resp. $\varphi^{\ast}: \{-1,1\}^{\mathbb{N}}\to \R$
	follows from the rapid increase of the sequences $A_j$ resp. $B_j$. Repeating
	this for $\ell=2,3,\ldots,2m-1$, finally all entries of $V$ will be
	independent over $\Q$ with $\{1\}$, as desired. 
	
	Now Step II of the proof above combined with the
	following easy argument show that we can transition from $V$ to $\Omega$.
	
	\begin{proposition}  \label{p}
		Let $g, h$ be positive integers. Assume given real numbers
		$x_1, \ldots,x_g$ are $\Q$-linearly independent together with $1$. 
		Then only for a set of Hausdorff dimension $h-1$ of vectors
		$\boldsymbol{y}=(y_1,\ldots,y_h)$ in $\R^h$, the joint vector $\boldsymbol{z}=(x_1,\ldots,x_g,y_1,\ldots,y_h)$ is $\Q$-linearly dependent with $1$.
	\end{proposition}

\begin{proof}
	Since the $x_j$ are independent, if $\{ 1,\boldsymbol{z}\}$ is 
	dependent over $\Q$ then
	there are integers $M_i, N_i, N$ with
	\[
	N+ \sum_{i=1}^{g} M_i x_i +\sum_{i=1}^{h} N_i y_i = 0
	\]
	and not all $N_i$ are $0$. But this means that $\boldsymbol{y}$ lies
	in a countable union of affine hyperplanes of $\R^{h}$. Hence 
	we see that the set of $\boldsymbol{y}$ inducing dependency over $\Q$
	indeed has Hausdorff dimension at most $h-1$.
	\end{proof}
	
It suffices to take $g=2m, h=(n-2)m$ and to choose $x_i$ the entries of $V$ (first two columns of $\Omega$)
and $y_i$ the remaining entries of $\Omega$ (entries of $B$) in the proposition to conclude
that a full dimensional set of possible matrices $B\in \R^{m\times (n-2)}\cong \R^{m(n-2)}$ remains (in fact of positive $m(n-2)$ dimensional
Hausdorff measure).

\section{Proofs of Theorems~\ref{arbn} and~\ref{ganze} } \label{Sect9}

Theorems \ref{arbn} admits a rather easy proof, as a consequence of the following simple lemma. We consider maximum norms on both $\R^n$ and $\R^m$
and simply write $\psi^{m\times n}_{\Omega}(t)$
for $\psi^{m\times n}_{\Omega}(\Vert \cdot\Vert_{\infty},\Vert \cdot \Vert_{\infty},t)$.

\begin{lemma}  \label{llemma}
    Let $\Omega\in\R^{m\times n}$ be a block matrix
    consisting of rectangular blocks $\Omega_1,\ldots,\Omega_k$ 
    of sizes $(m_j,n_j), 1\le j\le k$
    along the diagonal, $m=\sum m_j$, $n=\sum n_j$.  
    Then
    \[
    \psi_{\Omega}^{m\times n}(t)= 
    \min_{1\leq j \leq k} \psi_{\Omega_j}^{m_j\times n_j}(t), \qquad t>0.
    \]
\end{lemma}

\begin{proof}
    First observe that by the form of $\Omega$, the problem
    decomposes into $k$ decoupled linear form approximation problems.
    Fix $t$ and let $\bb_1,\ldots,\bb_k$ in $\Z^{m_j+n_j}$ be the best approximation integer vectors 
    inducing $\psi_{\Omega_j}^{m_{j}\times n_{j}}(t)$ for $1\le j\le k$, i.e. $\bb=\bb_j$
    minimizes $\Vert\Omega_j^{\ast} \bb\Vert_{\infty}=\Vert\Omega_j  \hat{\bb}-\tilde{\bb}\Vert_{\infty}$ among integer vectors $\bb$ with $0<\Vert \bb\Vert_{\infty}\le t$.
    The inequality
    \[
    \psi_{\Omega}^{m\times n}(t) \le \min_{1\le j\le k} \psi_{\Omega_j}^{m_j\times n_j}(t)
    \]
    follows by taking the
    integer vector $\bb\in \Z^{m+n}$ with entries of $\bb_{j_0}$ at the according places, with $j_0$
    inducing the minimum of $\Vert\Omega_j^{\ast} \bb_j\Vert_{\infty}$ over $j$, and $0$ entries at all other places. By choice of the maximum norm, this will have the same norm 
    \[
    \Vert \Omega^{\ast}\cdot \bb\Vert_{\infty} = 
    \Vert\Omega_{j_0}^{\ast}\cdot \bb_{j_{0}}\Vert_{\infty} 
    =\psi_{\Omega_{j_0}}^{m_{j_0}\times n_{j_0}}(t)
    = \min_{1\le j\le k} \psi_{\Omega_j}^{m_j\times n_j}(t).
    \]
    Conversely, since some coordinate of the best approximation vector for given $t$ is not $0$ and the system is decoupled and we use maximum norm,
    we must have 
    \[
    \psi_{\Omega}^{m\times n}(t) \ge \min_{1\le j\le k} \psi_{\Omega_j}^{m_j\times n_j}(t).
    \]
\end{proof}

The deduction of the case $m=n$ works as follows.

\begin{proof}[Proof of Theorem~\ref{arbn}]
    As in  Lemma~\ref{llemma}, take a matrix of the form $\Omega=\rm{diag}(\Omega_1, \Omega_2)$ with diagonal
    blocks $\Omega_1\in Bad_{n-1,n-1}$ and
    $\Omega_2=\{\eta\}\notin Bad_{1,1}\cup \Q$ are arbitrary. 
    Keep in mind that as $n/n=(n-1)/(n-1)=1/1=1$, the Dirichlet exponent (remark: not constant) 
    is $1$ for all $\Omega_1$, $\eta$ and $\Omega$, for all arguments below.

    By Lemma~\ref{llemma} we have
    \begin{equation} \label{eq:ABC}
    \Theta^{n\times n}(\Omega) =
    \limsup_{t\to\infty} 
    t \cdot \psi_{\Omega}^{n\times n}(t)
    = \limsup_{t\to\infty} 
    t \cdot \min \{ \psi_{\Omega_1}^{(n-1)\times (n-1)}(t), \psi_{\eta}^{1\times 1}(t)\}.
    \end{equation}
    Now since $\Omega_1\in Bad_{n-1,n-1}\subseteq Di_{n-1,n-1}$ we have
    \[
    \Theta^{(n-1)\times (n-1)}(\Omega_1)=
    \limsup_{t\to\infty} 
    t^{\frac{n-1}{n-1}} \cdot \psi_{\Omega_1}^{(n-1)\times (n-1)}(t)=\limsup_{t\to\infty} 
    t \cdot \psi_{\Omega_1}^{(n-1)\times (n-1)}(t) < 1
    \]
    so $\Omega\in Di_{n,n}$ as well by \eqref{eq:ABC}. 
    Moreover since 
    \[
    \limsup_{t\to\infty} t\psi^{1\times 1}_{\eta}(t)\ge 1/2>0
    \]
    for any irrational $\eta$ by Khintchine~\cite{Khinchin3} and
    \[
     \liminf_{t\to\infty} t\psi_{\Omega_1}^{(n-1)\times (n-1)}(t)
    >0
    \]
    since $\Omega_1\in Bad_{n-1,n-1}$, again by \eqref{eq:ABC} we have 
    $\Theta^{n\times n}(\Omega)>0$ and thus $\Omega\notin Sing_{n,n}$.
    Finally again by Lemma~\ref{llemma}
    \[
    \tilde{\Theta}^{n\times n}(\Omega)=\liminf_{t\to\infty} t^{n/n} \psi_{\Omega}(t) \le 
    \liminf_{t\to\infty} t^{1/1} \psi_{\eta}(t) = 0
    \]
    since $\eta\notin Bad_{1,1}$, so $\Omega\notin Bad_{n,n}$. Hence
    $\Omega\in FS_{n,n}$.

    It is well-known that $Bad_{n-1,n-1}$ has full Hausdorff dimension $(n-1)^2$ and
    that the complement
    of $Bad_{1,1}$ has full $1$-dimensional Lebesgue measure,
    see Section~\ref{s12}. Removing the countable set $\Q$ clearly has no effect on the Lebesgue measure.
    In particular, we can choose regular (full $\R$-rank) matrices $B$ as the complementary set
    has strictly lower Hausdorff dimension, and by the diagonal form
    and $\eta\ne 0$ the same will be true for any arising $\Omega$.
    This argument together with \eqref{eq:tric} also shows that the dimension of the set of matrices $\Omega$ constructed above is at least $(n-1)^2+1$.
\end{proof}

We want to formulate another consequence 
of Lemma \ref{llemma} for maximum norms that is used
in the deduction of some claims of Theorem \ref{neuth}.

\begin{corollary} \label{multiples}
   With $\mathcal{D}_{m,n}$ 
   as in \eqref{eq:dirspe}, for any positive integers $m,n,k$ and $c>0$ we have
   \[
   \mathcal{D}_{m,n} \subseteq \mathcal{D}_{km,kn}, 
   \]
   and
   \[
   \dim_H FS_{m,n} \le \dim_H FS_{km,kn} \;, \qquad \dim_H FS_{m,n}(c) \le \dim_H FS_{km,kn}(c).
   \]
\end{corollary}

\begin{proof}
    For any $m\times n$ matrix $\Omega$, by Lemma \ref{llemma} the $km\times kn$-matrix $\Omega^{\prime}:=\rm{diag}(\Omega,\ldots,\Omega)$
    with $k$ diagonal blocks $\Omega$ induces the same approximation function
    \[
    \psi_{\Omega}^{m\times n}(t)= \psi_{\Omega^{\prime} }^{km\times kn}(t), \qquad t>0.
    \]
    Since the Dirichlet exponents coincide via $(kn)/(km)=n/m$, the 
    claims follow directly.
\end{proof}

In fact we now know there is identity in the inclusion
thanks to \cite{AginWeiss2024}.
For the proof of Theorem~\ref{ganze}
we combine Lemmas~\ref{llemma}, \ref{lemme} with the main result from~\cite{j1} on simultaneous approximation $n=1$.

\begin{proof}[Proof of Theorem~\ref{ganze}]

Let us introduce $\ell:= \left\lceil \frac{m}{n}\right\rceil \ge 2$, so that 
\begin{equation} \label{eq:ojsi}
    (\ell-1)n< m \le \ell n.
\end{equation}
    For simplicity, first we treat the case $\ell=2$ (that suffices
    for the sake of Theorem \ref{non-empty main}). 

    {\em Case $\ell=2$:}
    Then
    \begin{equation} \label{nm bound}
    \frac{1}{2}\le \frac{n}{m} < 1.
    \end{equation}
    Moreover we can assume $m\ge 4$ and $n\ge 2$, as the otherwise 
    our exclusion of $(m,n)=(3,2)$ forces
    $n=1$, and this
    case is a trivial consequence of \cite{j1}.
    Let us first consider maximum norms
    \begin{equation} \label{eq:maxno}
\Vert\cdot\Vert_1=\Vert\cdot\Vert_2=\Vert\cdot\Vert_{\infty}.
\end{equation}
    Then~\cite[Theorem~2.2]{j1} in the case of simultaneous approximation to two numbers 
    can be applied to $\Phi(t)=ct^{-n/m}$ for
    any $c\in (0,1)$, as the conditions (d1)-(d3) of \cite{j1} are satisfied by our choice of $\Phi$ and \eqref{nm bound}. This yields $\ux\in \R^{2\times 1}$ with
    \[
    \limsup_{t\to\infty} t^{n/m} \psi^{2\times 1}_{\ux}(t) = c, \quad \text{and} \quad \liminf_{t\to\infty} t^{n/m} \psi^{2\times 1}_{\ux}(t) = 0.
    \]
    Now let $\uz\in Bad_{m-2,1}$ arbitrary 
    unless if $m=4$ then $\uz=\ux$. 
    Since $m-2\ge 2$ by assumption, it is easily seen that in either case
     \[
    \psi_{\ux}^{2\times 1}(t) \le \psi_{\uz}^{(m-2)\times 1}(t), \qquad t\ge t_0.
    \]
    Hence by Lemma~\ref{llemma} the matrix
    $U=\rm{diag}(\ux,\uz)\in \R^{ m \times 2}$ 
    satisfies
    \begin{equation} \label{Vt property}
    \limsup_{t\to\infty} t^{n/m} \psi^{m\times 2}_{U}(t) = c, \qquad \liminf_{t\to\infty} t^{n/m} \psi^{m\times 2}_{U}(t) = 0,
    \end{equation}
    as well. \par 
    If $n=2$ then we must have $m=4$ due to our assumptions $\tfrac{n}{m}\geq \frac{1}{2}$ and $m\ge 4$, and 
    then we can just take $\Omega=U$.
    Note also that in this case
    our claimed bound for Hausdorff dimension becomes zero, so there is nothing more to prove.\par 
    If $n\geq 3$, then in Lemma~\ref{lemme} take $V=U\in \R^{m\times2}$. 
    Note that the problematic
    case $n=m+2$ of the lemma is excluded
    here via \eqref{nm bound}.
    Notice further $V$ satisfies \eqref{V property} by \eqref{Vt property}. Hence, assuming $c>0$ is small enough, there exists a set of positive $m(n-2)$-dimensional Lebesgue measure $\mathcal{Y}\subseteq \R^{m\times(n-2)}$ of matrices so that for any $B\in\mathcal{Y}$, the $\R^{m\times n}$ matrix $\Omega=(B,U)$ satisfies
    \begin{equation*}
        \limsup_{t\to\infty}t^{n/m}\psi_{\Omega}^{m\times n}(t)=c\, .
    \end{equation*}
    Finally, we show $\Omega$ is not badly approximable.
    By \eqref{Vt property} there exists a sequence of best approximations such that on the sequence $(t_{j})_{j\geq 1}$ we have
    \begin{equation*}
        \lim_{j\to \infty} t_{j}^{n/m}\psi_{U}^{m\times 2}(t_{j})=0.
    \end{equation*}
    Considering $\Omega$, take the same sequence of the best approximations and place a zero in all entries not related to $U$ (that is starting from the third) to obtain the same statement, i.e.
    \begin{equation*}
        \lim_{j\to \infty} t_{j}^{n/m}\psi_{\Omega}^{m\times n}(t_{j})=0.
    \end{equation*}
    Combining these two claims we infer $\Omega \in FS_{m,n}(c)$. Using \eqref{eq:tric} the dimension statement follows since we have shown that (when $m\geq 5$)
    \begin{equation*}
        \dim_H FS_{m,n}(c) \geq \dim_H \mathcal{Y} + \dim_H Bad_{m-2,1} = m(n-2) + m-2, 
    \end{equation*}
   where we used that $Bad_{m-2,1}$ has full Hausdorff dimension and $\mathcal{Y}$ has positive Lebesgue measure. In the case $(m,n)=(4,3)$ we obtain
    \begin{equation*}
        \dim_H FS_{4,3}(c) \geq \dim_H \mathcal{Y} = m(n-2)=4. 
    \end{equation*}
    We remark that for $(m,n)=(4,2)$ the according estimate holds as well, but the right hand side vanishes. This completes the proofs for the maximum norm, i.e. \eqref{eq:maxno}.
    

    Now let $\Vert \cdot\Vert_1$ be arbitrary and still $\Vert\cdot\Vert_2=\Vert \cdot\Vert_{\infty}$. Let $\Vert\cdot \Vert_1^{\prime}$
    be the projected norm
    \[
    \Vert (x_1,x_2)\Vert_1^{\prime}
    := \Vert (x_1,x_2,0,\ldots,0)\Vert_1.
    \]
    Let $\gamma= \Vert\ee_1\Vert_1^{\prime}$ be the norm of the first base vector of $\R^2$. By \cite[Theorem 2.2]{j1} and again \eqref{nm bound}, for small enough $c>0$ we may pick $\ux\in\R^{2\times 1}$ such that 
      \begin{equation} \label{eq:invo}
    \limsup_{t\to\infty} t^{n/m} \psi^{ 2\times1}_{\ux}(t) = c_1:=c \gamma^{n/m}, \qquad \liminf_{t\to\infty} t^{n/m} \psi^{2\times1}_{\ux}(t) = 0.
    \end{equation}
    Suppose we can show that, for any $U= \rm{diag}(\ux,\uz)\in \R^{m\times 2}$ as above, which we can identify with elements in $\{ \ux\}\times Bad_{m-2,1}\subseteq \R^{m}$, we have that
     \begin{equation} \label{claim}
    \psi_U^{m\times 2}(\Vert\cdot\Vert_1^{\prime}, \Vert\cdot\Vert_2,t) = \psi_{\ux}^{2\times 1}(|\cdot |,\Vert \cdot\Vert_{\infty},\gamma^{-1} t),\qquad t\ge t_0 \, .
    \end{equation}
    If \eqref{claim} holds, then indeed
    \[
    \limsup_{t\to\infty} t^{n/m} \psi_U^{m\times 2}(\Vert\cdot\Vert_{1}',\Vert\cdot\Vert_{2},t) =\gamma^{-n/m}  \limsup_{t\to\infty} t^{n/m} \psi_{\ux}^{2\times 1}(|\cdot |,\Vert \cdot\Vert_{\infty},\gamma^{-1} t)= \gamma^{-n/m} c_1=c,
    \]
    and
    \[
    \liminf_{t\to\infty} t^{n/m}\psi_U^{m\times 2}(\Vert\cdot\Vert_{1}',\Vert\cdot\Vert_{2},t) = \gamma^{-n/m} \liminf_{t\to\infty}t^{n/m}\psi_{\ux}^{2\times 1}(|\cdot |,\Vert \cdot\Vert_{\infty},\gamma^{-1} t)= 0.
    \]
    Hence for $c$ small enough we again we may again apply 
    Lemma~\ref{lemme}, to find some set $\mathcal{Y}\subseteq \R^{m\times (n-2)}$ of positive Lebesgue measure
    such that for any $B\in \mathcal{Y}$ we have \begin{equation*}
        \limsup_{t\to\infty} t^{n/m}\psi^{m\times n}_{(U,B)}(\Vert\cdot\Vert_{1},\Vert\cdot\Vert_{2},t)=c \, ,\qquad \liminf_{t\to\infty} t^{n/m}\psi^{m\times n}_{(U,B)}(\Vert\cdot\Vert_{1},\Vert\cdot\Vert_{2},t)=0\, .
    \end{equation*}
    So $(U,B)\in FS_{m,n}(\Vert\cdot\Vert_{1},\Vert\cdot\Vert_{2},c)$, and analogously to the maximum norm by \eqref{eq:tric} we conclude
    \begin{equation*}
        \dim_H FS_{m,n}(\Vert\cdot\Vert_{1},\Vert\cdot\Vert_{2},c) \geq \dim_H \mathcal{Y} + \dim_H Bad_{m-2,1} = m(n-2) + m-2, 
    \end{equation*}
    for $m\ge 5$, and accordingly for $(m,n)=(4,3)$ we get the bound 
    \[
    \dim_H FS_{4,3}(\Vert\cdot\Vert_{1},\Vert\cdot\Vert_{2},c) \geq
    m(n-2)=4.
    \]
    Since $\lceil m/n\rceil=2$ by \eqref{nm bound}, we indeed recognize these as special cases
    of our claim.
    Thus to complete the proof for $\ell=2$ we need to show \eqref{claim} holds.
    
    For the inequality
    \begin{equation} \label{eq:unn}
    \psi_U^{m\times 2}(\Vert\cdot\Vert_1^{\prime}, \Vert\cdot\Vert_2,t) \le \psi_{\ux}^{2\times 1}(|\cdot |,\Vert \cdot\Vert_{\infty},\gamma^{-1} t), \qquad t\geq t_{0}
    \end{equation}
    observe that any $(x,\boldsymbol{w})\in \Z\times \Z^2$ with $| x\vert\le \gamma^{-1} t$ gives rise to
    $\boldsymbol{z}=(x,0)=x \ee_1\in \Z^2$ of norm $\Vert \boldsymbol{z} \Vert_1^{\prime}= |x|\gamma\le t$. Moreover since
      \begin{equation} \label{eq:noon}
    U\cdot \boldsymbol{z}-(\boldsymbol{w},0,\ldots,0)= (\ux\cdot x - \boldsymbol{w} , 0, \ldots,0)\in\R^m,
    \end{equation}
    with the same first two coordinates
    followed by $m-2$ zeros, we have the identity
    \begin{equation}  \label{eq:spet}
        \Vert U\cdot \boldsymbol{z}-(\boldsymbol{w},0,\ldots,0)\Vert_{\infty}=  \Vert\ux\cdot x-\boldsymbol{w}\Vert_{\infty}\, ,
    \end{equation}
   where we used $\Vert\cdot\Vert_2$ being the maximum norm. As $t$ is arbitrary \eqref{eq:unn} is implied.

    Now consider the reverse inequality 
    \begin{equation} \label{eq:rever}
    \psi_U^{m\times 2}(\Vert\cdot\Vert_1^{\prime}, \Vert\cdot\Vert_2,t) \ge \psi_{\ux}^{2\times 1}(|\cdot |,\Vert \cdot\Vert_{\infty},\gamma^{-1} t), \qquad t\geq t_{0}.
    \end{equation}
    Begin with $(m,n)\ne (4,2)$.
    Note that if $\boldsymbol{z}=(z_1,z_2)\in \mathbb{Z}^2$ with $\Vert \boldsymbol{z}\Vert_{1}^{\prime} \le t$ does
    not lie in the space $z_2=0$,
    then by equivalence of norms
    for any $\textbf{y}\in\mathbb{Z}^m$
    \begin{equation} \label{eq:impco}
    \Vert U\cdot \boldsymbol{z}- \boldsymbol{y}
    \Vert_{\infty}
    \gg \Vert\boldsymbol{z}\Vert_{\infty}^{-1/(m-2)}\gg \Vert\boldsymbol{z}\Vert_{1}^{\prime -1/(m-2)}\ge t^{-1/(m-2)}
    \end{equation}
    where the most left estimate holds
    since
    the system is decoupled and 
    $\uz\in Bad_{m-2,1}$. As we noticed above that $m\ge 4$ and $1/2\le n/m$, in our case $(m,n)\ne (4,2)$ we have
    the strict inequality $1/(m-2)< n/m$. Then, in view of \eqref{eq:invo}, 
    approximation quality \eqref{eq:impco}
    is beaten by $\psi_{\ux}^{2\times 1}(|\cdot |,\Vert \cdot\Vert_{\infty},\gamma^{-1} t)$, so extending accordingly
    the involved integer best approximation approximations
    to $\Z^n$ by adding $0$ entries 
    leads to the same smaller value. So we can exclude
    such vectors. But for vectors $\boldsymbol{z}$ in the linear space $z_2=0$,
    the integer
    vector $\textbf{y}$ minimizing 
    $\Vert U\cdot \textbf{z}-\textbf{y}\Vert_{\infty}$
    for given $\textbf{z}$ is easily seen to be of the form
    $\textbf{y}=(\textbf{w},0,\ldots,0)$ with $\textbf{w}\in\mathbb{Z}^2$, as otherwise some entry of $U\cdot \textbf{z}-\textbf{y}$ is of modulus at least $1$. Then the
    output $U\cdot \textbf{z}-\textbf{y}\in \R^m$ again satisfies \eqref{eq:noon} and hence \eqref{eq:spet}, settling \eqref{eq:rever} via a short calculation. For $m=4, n=2$ a very similar argument applies
    using $\uz=\ux$, however we need
    $\Vert\cdot\Vert_1$ to be expanding to exclude
    better approximations outside
    the space $z_2=0$. If
    the expanding property is assumed, then it is again easy to see that we may restrict $\textbf{z}$ to
    this space and the claims follows easily. We leave the details to the reader.

    Finally, let $\Vert\cdot\Vert_2$ be any expanding norm on 
    $\mathbb{R}^m$
    and consider again the induced 
    projected norm $\Vert\cdot\Vert_2^{\prime}$ 
    on $\mathbb{R}^2$ by
    \[
    \Vert (x_1,x_2)\Vert_2^{\prime}:= \Vert (x_1,x_2,0,\ldots,0)\Vert_2.
    \]
    It is easy to see that $\Vert\cdot\Vert_2^{\prime}$ is an expanding norm
    on $\mathbb{R}^2$.
    A minor adaption of~\cite[Corollary~2]{j1} to more general $\Phi$ satisfying (d1)-(d3) of \cite{j1} can be proved analogously. When applied with respect to simultaneous approximation of two reals, and
    as (d1)-(d3) hold for our $\Phi(t)=ct^{-n/m}$
    in view of \eqref{nm bound} and $c\le 1$, 
    it yields the following:
    For any small enough $c>0$, we can similarly arrange some $\ux\in\mathbb{R}^{2\times 1}$ to satisfy
      \[
    \limsup_{t\to\infty} t^{n/m} \psi^{ 2\times1}_{\ux}(|\cdot|,\Vert\cdot\Vert_2^{\prime},t) = c_1:=c \gamma^{n/m},\quad \text{ and } \quad \liminf_{t\to\infty} t^{n/m} \psi^{2\times1}_{\ux}(|\cdot|,\Vert\cdot\Vert_2^{\prime},t) = 0.\]
    We can then copy the proof above
    replacing $\Vert\cdot \Vert_{\infty}$ 
    by $\Vert\cdot\Vert_2^{\prime}$ whenever considering $\psi^{2\times 1}_{\ux}$.
    Again by the same argument, it suffices to show the accordingly altered variant of \eqref{claim}. For the estimate \eqref{eq:unn},
    the only place where we needed
    $\Vert\cdot\Vert_2$ to be
    the maximum norm is to deduce
    that the projected vectors
    satisfy \eqref{eq:spet}. However, 
    in view of \eqref{eq:noon},
    the defining property of 
    $\Vert\cdot\Vert_2^{\prime}$ means that the according equality 
    still holds. The proof of the reverse inequality 
     \eqref{eq:rever}
    is analogous to the above, only
    with different implied constants
    in \eqref{eq:impco}. The case $\ell=2$ is done up to the full rank claim. 
    
    For the rank claim, consider first
    the $n\times 2$ matrix obtained from
    any $n$ lines of $V$ containing its first and last line, and complete it to an $n\times n$ matrix $H$ by 
    adding to the right an $n\times (n-2)$ matrix of free variable entries $x_{i,j}$, $1\le i\le n$, $3\le j\le n$. Then writing
    $\ux=(\xi_1,\ldots,\xi_{\ell})$, $\uz=(\zeta_1,\ldots,\zeta_{m-\ell})$,
    the first and last line of $H$ are of the form
    $(\xi_1, 0, x_{1,3}, \ldots, x_{1,n})$ resp.
    $(0, \zeta_{m-\ell}, x_{n,3}, \ldots, x_{n,n})$. Since $\xi_i\ne 0$ and $\zeta_j\ne 0$ for all $1\le i\le \ell, 1\le j\le m-\ell$
    is easy to see (otherwise $\ux$ or $\uz$ would be singular, contradicting 
    $\ux\in DI_{\ell}(c)$ resp. 
    $\uz\in Bad_{m-\ell,1}$), it is further readily checked
    that the determinant of $H$ is not the constant $0$ polynomial in the
    $x_{i,j}$. Hence only a Lebesgue nullset of $G:=(x_{i,j})\in \R^{n\times (n-2)}$ will induce non invertible $H$. Now as $G\in \R^{n\times (n-2)}$ can be considered a submatrix of $B\in \R^{m\times (n-2)}$ above, by considering the Cartesian product
    $G\times \R^{(m-n)\times (n-2)}$
    only a nullset of $B$ will lead to $\Omega$ of deficient rank. On the other hand, we are equipped with a positive measure set $B$ by Lemma~\ref{lemme}. Hence
    indeed we can restrict to full rank  
    matrices without affecting the metrical claim.

    {\em Case $\ell\ge 3$:}
   Then by \eqref{eq:ojsi} similarly to \eqref{nm bound} we have
    \begin{equation} \label{eq:Dd}
    \frac{1}{\ell}\le \frac{n}{m} < \frac{1}{\ell-1}.
    \end{equation}
    Now again~\cite[Theorem~2.2]{j1} can be applied, with respect to simultaneous approximation to $\ell$ numbers, to $\Phi(t)=ct^{-n/m}$ for
    any $c\in (0,1)$, as its conditions (d1)-(d3) are satisfied by \eqref{eq:Dd}. This yields $\ux\in \R^{\ell \times 1}$ so that
    \[
    \limsup_{t\to\infty} t^{n/m} \psi^{\ell \times 1}_{\ux}(t) = c, \quad \text{and} \quad  \liminf_{t\to\infty} t^{n/m} \psi^{\ell \times 1}_{\ux}(t) = 0.
    \]
    Now let $\uz\in Bad_{m-\ell ,1}$ arbitrary 
    unless if $m=2\ell$ then $\uz=\ux$.
    We again consider $U=\rm{diag}(\ux, \uz)\in \mathbb{R}^{m\times 2}$ as above.
    As soon as strict inequality
    \begin{equation} \label{eq:mpso}
         m> 2\ell
    \end{equation}
    holds, we have $1/(m-\ell)< n/m$ so that
    the analogue of \eqref{eq:impco} similarly implies that we can restrict
    to the space $z_2=0$.
    It can be checked that
    the analogous arguments of the case $\ell=2$ above apply. Notably when 
    transitioning to any expanding norm $\Vert\cdot\Vert_2$, here we consider
    the restriction
    \[
     \Vert (x_1,\ldots,x_{\ell})\Vert_2^{\prime}:= \Vert (x_1,\ldots,x_{\ell},0,\ldots,0)\Vert_2,
    \]
    and apply an analogous variant of~\cite[Corollary~2]{j1} with respect to simultaneous approximation to $\ell$ numbers.
    
    Finally we need to discuss the separate cases when \eqref{eq:mpso} is satisfied, and evaluate the Hausdorff dimensions.
    Recall we have assumed $\ell\ge 3$, moreover
    we can assume $n\ge 2$ as $n=1$ was done in \cite{j2}. Via \eqref{eq:ojsi}, condition \eqref{eq:mpso} is readily checked to be satisfied if $n\ge 3$. Thus we are left with the subset of cases $n=2$. The cases
    $n=2, m\ge 4$ even, leading to $m=2\ell$, can be dealt with similarly to the case $(m,n)=(4,2)$ for $\ell=2$ above. This indeed leaves open only the cases $n=2$, and $m$ odd, as excluded in the theorem. 
    
    By construction, and the repeatedly used observation that
    $Bad_{m-\ell ,1}$ has full Hausdorff dimension $m-\ell$, we obtain the Hausdorff dimension bound
     \[
    \dim_H FS_{m,n}(c) \ge
    m\cdot (n-2)+(m-\ell)=m(n-1)-\left\lceil\frac{m}{n}\right\rceil, \qquad m\ge 2\ell+1,
    \]
    of the theorem. The remaining case $m=2\ell=2\lceil m/n\rceil$ 
    is readily checked to lead precisely to the cases $n=2$ and $m$ even 
    or $(m,n)=(4,3)$. When $(m,n)=(4,3)$ we have $\ell=2$, which is already treated above. It should be remarked that for $n=2$ our Hausdorff dimension bound is just $m(n-2)=0$, thus of no interest. The full rank claim is also proved analogously to the case $\ell=2$.
\end{proof}  



\section{Proof of Lemma~\ref{Folklore transpose}} \label{Sect10}

For matrix $\Omega \in \R^{m\times n}$ and real numbers $A,B\in\R_{+}$ define the convex bodies
\begin{align*}
    M_{A,B}(\Omega)&:=\left\{ \boldsymbol{z}\in \R^{n+m}: \max_{1\leq j\leq m} \left| \sum_{i=1}^{n}\Omega_{j,i}z_{i}+z_{n+j}\right|\leq A, \quad \max_{1\leq i \leq n }|z_{i}|\leq B \right\}, \\
    \widehat{M}_{A,B}(\Omega)&:=\left\{ \boldsymbol{z}\in \R^{n+m}: \max_{1\leq i \leq n} \left|z_{i}-\sum_{j=1}^{m}\Omega_{i,j}z_{n+j}\right|\leq B, \quad \max_{1\leq j \leq m} |z_{n+j}|\leq A \right\}.
\end{align*}
The following theorem proven by German is crucial.
\begin{theorem}{\cite[Theorem 7]{german}} \label{transference}
Let $\Omega \in \R^{m\times n}$ and $A,B\in\R_{+}$. If
\begin{equation*}
    M_{A,B}(\Omega)\cap \Z^{n+m}\backslash\{ \boldsymbol{0} \} \neq \emptyset \, ,
\end{equation*}
then
\begin{equation*}
    \widehat{M}_{A^{*},B^{*}}(\Omega)\cap \Z^{n+m}\backslash \{ \boldsymbol{0} \} \neq \emptyset 
\end{equation*}
for 
\begin{equation*}
    A^{*}=C(A^{n}B^{1-n})^{\frac{1}{n+m-1}} \quad \text{ and } \quad B^{*}=C(A^{1-m}B^{m})^{\frac{1}{n+m-1}},
\end{equation*}
with $C>0$ constant dependent only on $n$ and $m$.
\end{theorem}
 \begin{proof}[Proof of Lemma~\ref{Folklore transpose}]
     Since $\Omega \in FS_{m,n}(\|\cdot\|_{1},\|\cdot\|_{2},c)\subseteq \R^{m\times n}$ we have
     $\Theta^{m\times n}(\|\cdot\|_{1},\|\cdot\|_{2},\Omega)=c$. By the equivalence of norms we have that the same is true up to some constant, and so
      \begin{equation*}
         \Theta^{m\times n}(\|\cdot\|_{\infty},\|\cdot\|_{\infty},\Omega)\asymp c\, .
     \end{equation*}
     Hence, for all sufficiently large $t\in\R$ and some $\rho>0$ there is an integer
     \begin{equation*}
         \boldsymbol{b} \in M_{\rho t^{-n/m},t}(\Omega)\cap\Z^{n+m}\backslash \{\textbf{0}\}\, .
     \end{equation*}
     By Theorem~\ref{transference}, for all sufficiently large $u\in\R$ there exists integer
     \begin{equation*}
         \boldsymbol{b} \in \widehat{M}_{u,\delta u^{-m/n}}(\Omega)\cap\Z^{n+m}\backslash \{\textbf{0}\}\, , \quad \text{ for } \quad \delta=C^{\tfrac{m+n}{n}}\rho^{\tfrac{m}{n(n+m-1)}}.
     \end{equation*}
     Hence $\Omega^{T} \in Di_{n,m}(\delta)$, and by equivalence of norms there exists some $\delta^{*}>0$ such that $\Omega^{T}\in Di_{n,m}(\Vert\cdot\Vert_{2},\Vert\cdot\Vert_{1},\delta^{*})$. 
     If $D^{*}=D^{n\times m}(\Vert\cdot\Vert_1,\Vert\cdot\Vert_2)$ is the Dirichlet constant from \eqref{eq:obben} for the transpose problem,
     for small enough $\delta$ and thus $\delta^{*}<D^{*}$ this will induce a Dirichlet
     improvable matrix $\Omega^{T}$.
     Conversely, for any $0<\gamma<c$ there exists an unbounded sequence of $t\in\R$ for which 
     \begin{equation*}
         \widehat{M}_{t,\gamma t^{-n/m}}(\Omega^{T})\cap \Z^{n+m}\backslash\{\textbf{0}\} = \emptyset\, ,
     \end{equation*}
     otherwise $\Omega \in Di_{m,n}(\gamma)$, which is clearly false since $\Omega \in DI_{m,n}(c)$. Applying the contrapositive of Theorem~\ref{transference} we obtain that there exist certain arbitrarily large $u\in\R$ and a constant $C^{*}>0$ for which
     \begin{equation*}
         M_{C^{*}u^{-n/m},u}(\Omega^{T})\cap\Z^{n+m}\backslash\{\textbf{0}\} = \emptyset\, .
     \end{equation*}
     Hence $\Omega^{T} \not\in Sing_{n,m}$. Furthermore, one can similarly show that $\Omega^{T} \not \in Bad_{n,m}$, and so $$\Omega^{T} \in Di_{n,m}(\Vert\cdot\Vert_{2},\Vert\cdot\Vert_{1},\delta^{*})\backslash (Sing_{n,m}\cup Bad_{n,m})\subseteq FS_{n,m}(\Vert\cdot\Vert_{2},\Vert\cdot\Vert_{1}).$$ 
 \end{proof}

 \section{Appendix: Improvements using subsequent work by Agin and Weiss }

 A previously announced, we can show some more advanced results using Agin and Weiss~\cite{AginWeiss2024}.

 \begin{theorem} \label{AWe}
     Let $m,n$ be positive integers and $\Vert.\Vert_1$ and $\Vert.\Vert_2$ be arbitrary norms on $\mathbb{R}^n$ and $\mathbb{R}^m$ respectively. Then there 
     exists $c_0=c_0(m,n,\Vert.\Vert_1,\Vert.\Vert_2)>0$ so that
     for any $c\in [0,c_0]$ we have the following:
     \begin{itemize}
         \item[(a)] 
          As soon as $m\ne n$ and $(m,n)\ne (1,3)$  
         we
         have
         \[
     \dim_H FS_{m,n}(\Vert \cdot \Vert_{1},\Vert \cdot \Vert_{2},c)\ge m\cdot (n-2).
     \]
     For $m=n$ we still have
     \[
     \dim_H DI_{n,n}(\Vert \cdot \Vert_{1},\Vert \cdot \Vert_{2},c)\ge n\cdot (n-2).
     \]
     \item[(b)] If $m>n$ we have the stronger bound 
     \[
     \dim_H FS_{m,n}(\Vert \cdot \Vert_{1},\Vert \cdot \Vert_{2},c)\ge m\cdot (n-1).
     \]
     \end{itemize}
 \end{theorem}

 \begin{remark}
     In part (a), 
     for $m=n$ we cannot exclude that the matrices in $DI_{n,n}(\Vert \cdot \Vert_{1},\Vert \cdot \Vert_{2},c)$ we obtain are badly approximable and thus outside our sets $FS_{n,n}(\Vert \cdot \Vert_{1},\Vert \cdot \Vert_{2},c)$.
     In private communication with the authors of \cite{AginWeiss2024}, 
     they were optimistic that some refinement of their method may close the gap, however it is by no means immediate from their work.
     We can prove unconditionally a weaker positive lower bound for $\dim_H FS_{n,n}(\Vert \cdot \Vert_{1},\Vert \cdot \Vert_{2},c)$ when $n\ge 3$, by combining the diagonal block method from Theorem \ref{arbn} with the case $m=n=2$ of \cite{AginWeiss2024}, however we prefer to omit it.
 \end{remark}

 Part (a) improves on 
 Theorem \ref{thm5} by avoiding the condition $n\ge 2m$ mainly.
 Part (b) improves on Theorem~\ref{ganze} in three aspects: Firstly, the restriction on the norms is removed. Secondly the Hausdorff dimension estimate is improved. 
 Finally we do not need to exclude $(m,n)\in \{(3,2), (5,2), (7,2), \ldots \}$.
 Theorem \ref{AWe} can also be viewed as an improvement of~\cite{AginWeiss2024} via providing Hausdorff dimension estimates, for the cost of restricting $c$ to small enough values and some conditions on $m,n$.

 Moreover the proof is significantly shorter as for both aforementioned theorems of ours. 
 For Theorem \ref{ganze}
 the reason is that using \cite{AginWeiss2024} we are able to bypass an obstacle from~\cite{j1} which is present in our above proof. This stems mainly from the fact that in~\cite{j1} the uniform exponent of the involved real column vector needs to be in the range $[1/m,1/(m-1))$ for the result to be valid. This condition disappears in~\cite{AginWeiss2024}. In particular we do not need to introduce $\ell$ from the proof of Theorem~\ref{ganze}, moreover we do not need the intermediate step of blowing up a column vector $\ux$
 first to $m\times 2$ matrix, but can directly go to $m\times n$ matrices via Lemma~\ref{lemme} or a variant of it. The details are as follows.

 \begin{proof}[Proof of Theorem~\ref{AWe}]

 Proof of part (a): As in \eqref{eq:pno},
 let $\Vert.\Vert_1^{\prime}$ be the restriction of $\Vert.\Vert_1$ to the first two coordinates, that is
 \[
 \Vert (x_1,x_2) \Vert_1^{\prime}:= \Vert (x_1,x_2,0,\ldots,0)\Vert_1.
 \]
 By \cite[Corollary~1]{AginWeiss2024}, we have that for any given $c\ge 0$, there are uncountably many $V\in \mathbb{R}^{m\times 2}$ such that
 \[
   \limsup_{t\to\infty} t^{n/m}\psi_{V}^{m\times 2}(\Vert \cdot \Vert_{1}^{\prime},\Vert \cdot \Vert_{2},t)=c\, .
     \]
     Now Lemma~\ref{lemme}, combined with the argument in Section~\ref{st2}, readily implies the following: If 
     $(m,n)\ne (1,3)$ and
     $c$ is small enough, then 
     for a positive $m(n-2)$-dimensional Lebesgue measure set of matrices $B\in \mathbb{R}^{m\times (n-2)}$,
     the extension of $V$ to a matrix $\Omega=(V,B)\in \mathbb{R}^{m\times n}$
     via adding $B$ to the right of $V$ has Dirichlet constant $c$, that is
        \[
        \Theta^{m\times n}(\Vert\cdot\Vert_1,\Vert\cdot\Vert_2,\Omega)=\limsup_{t\to\infty} t^{n/m}\psi_{\Omega}^{m\times n}(\Vert \cdot \Vert_{1},\Vert \cdot \Vert_{2},t)=
     \limsup_{t\to\infty} t^{n/m}\psi_{V}^{m\times 2}(\Vert \cdot \Vert_{1}^{\prime},\Vert \cdot \Vert_{2},t)=c\, .
     \]
     Finally, for $m\ne n$ and $(m,n)\ne (1,3)$, we need to show $\tilde{\Theta}^{m\times n}(\Vert\cdot\Vert_1,\Vert\cdot\Vert_2,\Omega)=0$ to conclude.
     As otherwise part (b) proved independently below gives even stronger bounds, we may assume $n/m\ge 1$, so by our assumption $m\ne n$ in fact $n/m>1$. Thus,
     we may apply non-trivial estimates relating uniform versus ordinary exponents of approximation
    for a system of linear forms in two variables obtained in~\cite{nikmo} to $V$, to see
     that for some explicitly computable $\alpha>n/m$
     \[
     \liminf_{t\to\infty} t^{\alpha}\psi_{V}^{m\times 2}(\Vert \cdot \Vert_{1}^{\prime},\Vert \cdot \Vert_{2},t)=0.
     \]
     Thus indeed further
       \[
        \tilde{\Theta}^{m\times n}(\Vert\cdot\Vert_1,\Vert\cdot\Vert_2,\Omega)=\liminf_{t\to\infty} t^{n/m}\psi_{\Omega}^{m\times n}(\Vert \cdot \Vert_{1},\Vert \cdot \Vert_{2},t)\le 
     \liminf_{t\to\infty} t^{n/m}\psi_{V}^{m\times 2}(\Vert \cdot \Vert_{1}^{\prime},\Vert \cdot \Vert_{2},t)=0.
     \]
     The used inequality is rather immediate,
     see Section \ref{st2} for details.
     Hence any such $B$ induces $\Omega\in FS_{m,n}(\Vert.\Vert_1, \Vert.\Vert_2,c)$. Clearly the Hausdorff dimension of $\Omega$ obtained is at least $m(n-2)$, in fact it has positive $m(n-2)$-dimensional Lebesgue measure.
 
 Proof part (b): Similar to \eqref{eq:pno}, by abuse of notation
 let now $\Vert.\Vert_1^{\prime}$ be the restriction of $\Vert.\Vert_1$ to the first coordinate, that is
 \[
 \Vert x_1 \Vert_1^{\prime}:= \Vert (x_1,0,\ldots,0)\Vert_1.
 \]
 Clearly this is a multiple of the absolute value.

 As in the notation in~\cite[Corollary~1]{AginWeiss2024} set $\gamma=\tfrac{n}{m}$ and consider the matrices $M_{m,1}$. Then, since we assume $\tfrac{n}{m}<1$, by \cite[Corollary~1]{AginWeiss2024} we have that for any given $c\ge 0$, there are uncountably many $\ux\in \mathbb{R}^{m\times 1}$ such that
 \[
   \limsup_{t\to\infty} t^{n/m}\psi_{\ux}^{m\times 1}(\Vert \cdot \Vert_{1}^{\prime},\Vert \cdot \Vert_{2},t)=c\, .
     \]
     Now, a variant of Lemma~\ref{lemme}, with $V=\ux\in \mathbb{R}^{m\times 1}$ a column vector (proved analogously, see the
     last part of its proof and the comment below Lemma~\ref{lemme}), combined with the argument in Section~\ref{st2} reads as follows: If $c$ was chosen small enough, then 
     for a positive $m(n-1)$-Lebesgue measure set of matrices $B\in \mathbb{R}^{m\times (n-1)}$,
     the extension of $V=\ux$ to a matrix $\Omega=(V,B)\in \mathbb{R}^{m\times n}$
     via adding $B$ to the right of $\ux$ has the property
        \[
        \Theta^{m\times n}(\Vert\cdot\Vert_1,\Vert\cdot\Vert_2,\Omega)=\limsup_{t\to\infty} t^{n/m}\psi_{\Omega}^{m\times n}(\Vert \cdot \Vert_{1},\Vert \cdot \Vert_{2},t)=
     \limsup_{t\to\infty} t^{n/m}\psi_{\ux}^{m\times 1}(\Vert \cdot \Vert_{1}^{\prime},\Vert \cdot \Vert_{2},t)=c\, .
     \]
     The obtained condition $n\ne m+1$
     is irrelevant since $m>n$. Now notice that we can assume $n
     \ge 2$ (otherwise our statement is trivial) and thus $n/m>1/m$ exceeds the trivial Dirichlet exponent for
     an $m\times 1$ column vector. It then follows from non-trivial estimates relating uniform versus ordinary exponents of approximation for simultaneous approximation (see~\cite{mamo} for optimal estimates, weaker bounds from~\cite{SchmidtSummerer} suffice as well) that
     for some explicit $\beta>n/m$
     \[
     \liminf_{t\to\infty} t^{\beta}\psi_{\ux}^{m\times 1}(\Vert \cdot \Vert_{1}^{\prime},\Vert \cdot \Vert_{2},t)=0
     \]
     and thus similarly as in part (a) further
       \[
        \tilde{\Theta}^{m\times n}(\Vert\cdot\Vert_1,\Vert\cdot\Vert_2,\Omega)=\liminf_{t\to\infty} t^{n/m}\psi_{\Omega}^{m\times n}(\Vert \cdot \Vert_{1},\Vert \cdot \Vert_{2},t)\le 
     \liminf_{t\to\infty} t^{n/m}\psi_{\ux}^{m\times 1}(\Vert \cdot \Vert_{1}^{\prime},\Vert \cdot \Vert_{2},t)=0.
     \]
     Hence again any such $B$ induces $\Omega\in FS_{m,n}(\Vert.\Vert_1, \Vert.\Vert_2,c)$ and we conclude as in part (a). 
 \end{proof} 
 
 Combining Theorem~\ref{AWe} with Lemma \ref{Folklore transpose} and Theorem \ref{arbn} we may infer that the Folklore set is not empty for
 any $(m,n)\ne (1,1)$, and obtain reasonably good metrical bounds, however
 still weaker than \cite[Theorem~3.15]{Dasetal} (for max norms) recalled in Section~\ref{recentadvances}. Without metrical bounds this consequence of Theorem~\ref{AWe} is already contained in~\cite{AginWeiss2024} as well.



\begin{thebibliography}{10}

\bibitem{algo}
A.~Agin.
\newblock Constructing displacement vectors.
\newblock {\em Combinatorics and Number Theory}, 13(3);239--264, 2024.

\bibitem{akmo}
R.~K. Akhunzhanov and N.~G. Moshchevitin.
\newblock A note on {D}irichlet spectrum.
\newblock {\em Mathematika}, 68(3):896--920, 2022.

\bibitem{akh}
R.~K. Akhunzhanov and D.~O. Shatskov.
\newblock On {D}irichlet spectrum for two-dimensional simultaneous
  {D}iophantine approximation.
\newblock {\em Mosc. J. Comb. Number Theory}, 3(3-4):5--23, 2013.

\bibitem{AginWeiss2024} A.~Agin, B. Weiss. The Dirichlet spectrum. {\em arXiv: 2412.05858}.

\bibitem{aduke}
N.~Andersen and W.~Duke.
\newblock On a theorem of {D}avenport and {S}chmidt.
\newblock {\em Acta Arith.}, 198(1):37--75, 2021.

\bibitem{BDGW_null}
V.~Beresnevich, S.~Datta, A.~Ghosh, and B.~Ward.
\newblock Bad is null.
\newblock \url{https://arxiv.org/abs/2307.10109}, preprint 2023.

\bibitem{BerGuaMarRamVel}
V.~Beresnevich, L.~Guan, A.~Marnat, F.~Ram\'{\i}rez, and S.~Velani.
\newblock Dirichlet is not just bad and singular.
\newblock {\em Adv. Math.}, 401:Paper No. 108316, 57, 2022.

\bibitem{Cassels}
J.~W.~S. Cassels.
\newblock {\em An introduction to {D}iophantine approximation}.
\newblock Cambridge Tracts in Mathematics and Mathematical Physics, No. 45.
  Cambridge University Press, New York, 1957.

\bibitem{Cheung}
Y.~Cheung.
\newblock Hausdorff dimension of the set of singular pairs.
\newblock {\em Ann. of Math. (2)}, 173:no. 1, 127--167, 2011.

\bibitem{CheungChevallier}
Y.~Cheung and N.~Chevallier.
\newblock Hausdorff dimension of singular vectors.
\newblock {\em Duke Math. J.}, 165(12):2273--2329, 2016.

\bibitem{Chevalliersurvey}
N.~Chevallier.
\newblock Best simultaneous {D}iophantine approximations and multidimensional
  continued fraction expansions.
\newblock {\em Mosc. J. Comb. Number Theory}, 3(1):3--56, 2013.

\bibitem{DFSU_singular}
T.~Das, L.~Fishman, D.~Simmons, and M.~Urba\'nski.
\newblock A variational principle in the parametric geometry of numbers, with
  applications to metric {D}iophantine approximation.
\newblock {\em C. R. Math. Acad. Sci. Paris}, 355(8):835--846, 2017.

\bibitem{Dasetal}
T.~Das, L.~Fishman, D.~Simmons, and M.~Urba\'{n}ski.
\newblock A variational principle in the parametric geometry of numbers.
\newblock {\em Adv. Math.}, 437:Paper No. 109435, 130, 2024.

\bibitem{dav}
H.~Davenport and W.~M. Schmidt.
\newblock Dirichlet's theorem on diophantine approximation. {II}.
\newblock {\em Acta Arith.}, 16:413--424, 1969/70.

\bibitem{DavenportSchmidt3}
H.~Davenport and W.~M. Schmidt.
\newblock Dirichlet's theorem on diophantine approximation.
\newblock In {\em Symposia {M}athematica, {V}ol. {IV} ({INDAM}, {R}ome,
  1968/69)}, pages 113--132. Academic Press, London, 1970.

\bibitem{Falconer_book2013}
K.~Falconer.
\newblock {\em Fractal geometry}.
\newblock John Wiley \& Sons, Ltd., Chichester, third edition, 2014.
\newblock Mathematical foundations and applications.

\bibitem{german}
O.~N. German.
\newblock On {D}iophantine exponents and {K}hintchine's transference principle.
\newblock {\em Mosc. J. Comb. Number Theory}, 2(2):22--51, 2012.

\bibitem{Jarnik3}
V.~Jarn\'ik.
\newblock {D}iophantische {A}pproximationen und {H}ausdorffsches {M}ass.
\newblock {\em Mat. Sb.}, 36:371--382, 1929.

\bibitem{Khinchin3}
A.~Khinchin.
\newblock {\"Uber eine Klasse linearer diophantischer Approximationen}.
\newblock {\em {Rend. Circ. Mat. Palermo}}, 50:170--195, 1926.

\bibitem{kleinbock}
D.~Kleinbock and A.~Rao.
\newblock Abundance of {D}irichlet-improvable pairs with respect to arbitrary
  norms.
\newblock {\em Mosc. J. Comb. Number Theory}, 11(1):97--114, 2022.

\bibitem{KleinbockRao}
D.~Kleinbock and A.~Rao.
\newblock Weighted uniform {D}iophantine approximation of systems of linear
  forms.
\newblock {\em Pure Appl. Math. Q.}, 18(3):1095--1112, 2022.

\bibitem{Lagarias1}
J.~C. Lagarias.
\newblock Best simultaneous {D}iophantine approximations. {I}. {G}rowth rates
  of best approximation denominators.
\newblock {\em Trans. Amer. Math. Soc.}, 272(2):545--554, 1982.

\bibitem{Lagarias2}
J.~C. Lagarias.
\newblock Best simultaneous {D}iophantine approximations. {II}. {B}ehavior of
  consecutive best approximations.
\newblock {\em Pacific J. Math.}, 102(1):61--88, 1982.

\bibitem{mamo}
A.~Marnat and N.~G. Moshchevitin.
\newblock An optimal bound for the ratio between ordinary and uniform exponents
  of {D}iophantine approximation.
\newblock {\em Mathematika}, 66(3):818--854, 2020.

\bibitem{mars} J.M. Marstrand. \newblock The dimension of Cartesian product sets. 
\newblock
{\em Proc. Cambridge Philos. Soc.} 50 (1954),
198–202.

\bibitem{moshchevitinsurvey}
N.~G. Moshchevitin.
\newblock Best diophantine approximations: the phenomenon of degenerate
  dimension.
\newblock {\em London Math. Soc. Lecture Note Ser}, 338:158--182, 2007.

\bibitem{ngm}
N.~G. Moshchevitin.
\newblock Singular {D}iophantine systems of {A}. {Y}a. {K}hinchin and their
  application.
\newblock {\em Uspekhi Mat. Nauk}, 65(3(393)):43--126, 2010.

\bibitem{nikmo} N.G. Moshchevitin, Diophantine exponents for systems of linear forms in two variables. Acta Sci. Math. (Szeged) 79 (2013), no. 1--2, 347--367.

\bibitem{Roy}
D.~Roy
\newblock On {S}chmidt and {S}ummerer parametric geometry of numbers
\newblock {\em Ann. of Math. (2)}, 182(2): 739--786, 2015. 

\bibitem{j2}
J.~Schleischitz.
\newblock Dirichlet spectrum for one linear form.
\newblock {\em Bull. Lond. Math. Soc.}, 55(3):1330--1339, 2023.

\bibitem{j1}
J.~Schleischitz.
\newblock Exact uniform approximation and {D}irichlet spectrum in dimension at
  least two.
\newblock {\em Selecta Math. (N.S.)}, 29(5):Paper No. 86, 2023.

\bibitem{j3}
J.~Schleischitz.
\newblock Metrical results on the geometry of best approximations for a linear
  form.
\newblock {\em Acta Arith.},213(1):59-74, 2024.

\bibitem{schmidt1960}
W.~Schmidt.
\newblock A metrical theorem in diophantine approximation.
\newblock {\em Canadian Journal of Mathematics}, 12:619--631.

\bibitem{schmidtsystemoflinearforms}
W.~M. Schmidt.
\newblock Badly approximable systems of linear forms.
\newblock {\em J. Number Theory}, 1:139--154, 1969.

\bibitem{SchmidtSummerer}
W.~M. Schmidt and L.~Summerer,
\newblock Diophantine approximation and parametric geometry of numbers.
\newblock {\em Monatsh. Math.}, 169(1):51--104, 2013.

\bibitem{FabianSuessThesis}
F.~S{\"u}ess.
\newblock Simultaneous diophantine approximation on affine subspaces and
  dirichlet improvability.
\newblock {\em PhD Thesis, University of York}, 2017.

\bibitem{tricot}
C.~Tricot, Jr.
\newblock Two definitions of fractional dimension.
\newblock {\em Math. Proc. Cambridge Philos. Soc.}, 91(1):57--74, 1982.



\end{thebibliography}

\end{document}